\documentclass{article}     

\usepackage{amssymb}
\usepackage{graphicx}
\usepackage{setspace}
\usepackage{amsfonts,amsbsy}
\usepackage[pdftex]{hyperref}

\usepackage{graphicx}
\usepackage{setspace}
\usepackage{amsmath}
\usepackage{amssymb,amsfonts,amsbsy}
\usepackage{amsthm}
\usepackage[pdftex]{hyperref}
\newtheorem{theorem}{Theorem}[section]
\newtheorem{lemma}[theorem]{Lemma}
\newtheorem{corollary}[theorem]{Corollary}
\newtheorem{proposition}[theorem]{Proposition}

\newtheorem*{Proposition 0}{Proposition A}
\newtheorem*{Proposition 1}{Proposition B}

\newcommand{\X}{\mathfrak{X}}

\newcommand{\Conf}{\mathrm{Conf}}

\newcommand{\Z}{\mathbb{Z}}

\newcommand{\R}{\mathbb{R}}
\newcommand{\GT}{\mathbb{J}}
\newcommand{\T}{\mathcal{T}}
\newcommand{\Prob}{\mathrm{Prob}}

\newcommand{\Pk}{J_k^{(a,b)}}

\newcommand{\Oinf}{O(\infty)}

\begin{document}

\title{Random surface growth with a wall and Plancherel measures for $O(\infty)$}
\author{Alexei Borodin \and Jeffrey Kuan}
\date{}
\maketitle

\begin{abstract} We consider a Markov evolution of lozenge tilings
of a quarter-plane and study its asymptotics at large times. One of
the boundary rays serves as a reflecting wall.

We observe frozen and liquid regions, prove convergence of the local
correlations to translation-invariant Gibbs measures in the liquid
region, and obtain new discrete Jacobi and symmetric Pearcey
determinantal point processes near the wall.

The model can be viewed as the one-parameter family of Plancherel
measures for the infinite-dimensional orthogonal group, and we use
this interpretation to derive the determinantal formula for the
correlation functions at any finite time moment.
\end{abstract}

\tableofcontents

\section{Introduction}\label{Introduction}

The principal object of study in this paper is a one-parameter
family of probability measures on certain interlacing
two-dimensional particle systems that can be defined in at least
three different ways.

\smallskip

\noindent\textbf{Random lozenge tilings.} Consider the domain
pictured on the left in Figure~\ref{tiling} drawn on the regular
triangular lattice, and consider all possible tilings of this domain
by \textit{lozenges}\footnote{A lozenge consists of two neighboring
elementary triangles glued together.}.
\begin{center}
\begin{figure}[htp]
\caption{Lozenge tiling and nonintersecting paths} \medskip
\includegraphics[totalheight=5cm]{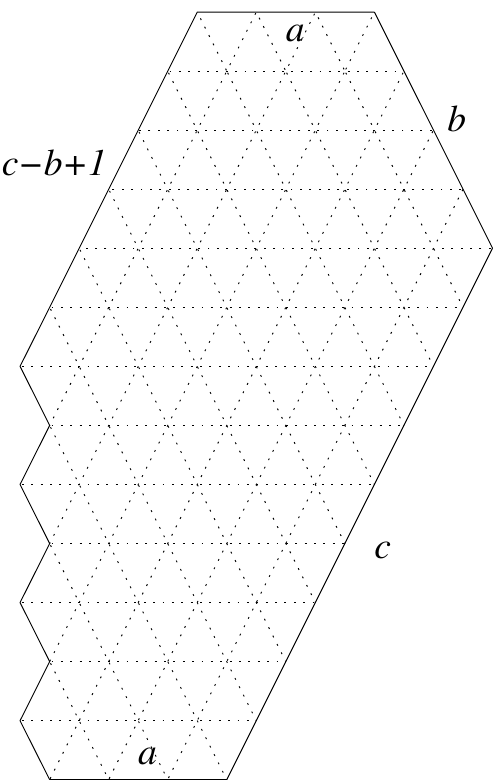}
\quad\quad
\includegraphics[totalheight=5cm]{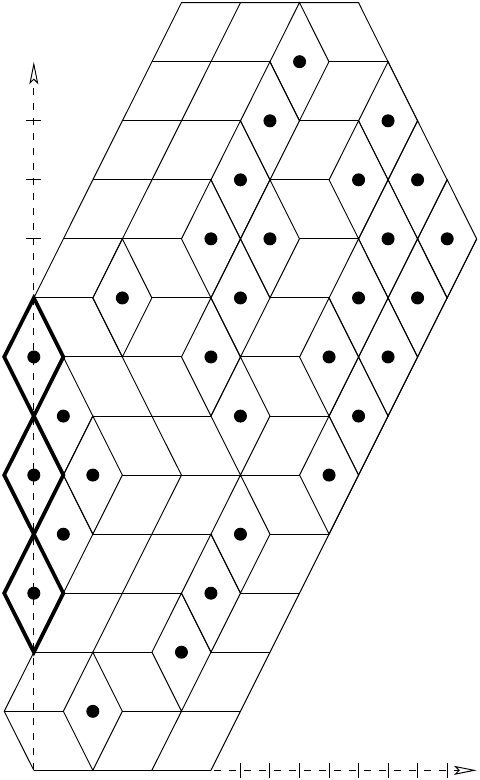}
\quad\quad
\includegraphics[totalheight=5cm]{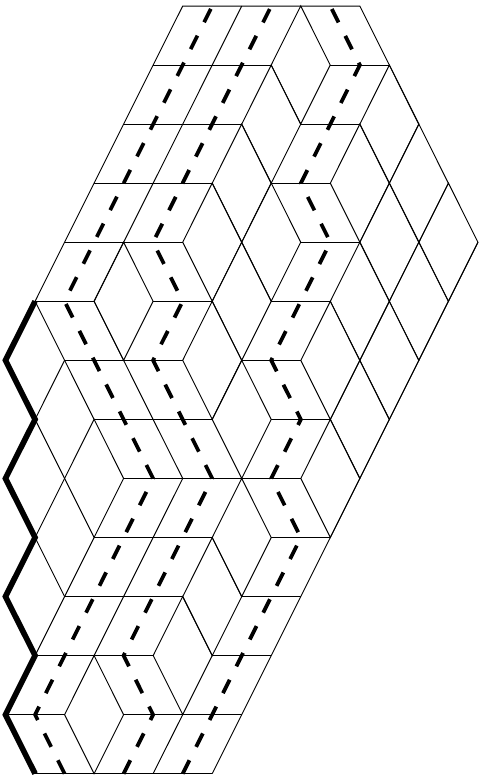}
\label{tiling}
\end{figure}
\end{center}
An example of lozenge tiling can be seen in the middle of Figure
\ref{tiling}. To each tiling we assign a weight equal to $\frac 12$
raised to the number of vertical lozenges on the left border of the
domain (three such lozenges are highlighted on the figure). Let us
normalize the weights so that the total weight of all tilings is 1;
then we obtain a probability distribution with three parameters
$a,b,c$ that represent side lengths of our domain.

Let us further consider the limit $a,b,c\to \infty$ so that
$2ab/c\to t>0$, and focus on the part of the tiling that is of
finite distance to the bottom-left corner of the domain. One can
show that in this limit our probability distributions weakly
converge to a probability measure $\mathcal M_t$ on lozenge tilings
of the quarter-plane, and it is the limiting measure that we are
interested in.

Lozenge tilings are also commonly viewed as stepped surfaces (when
three types of lozenges are interpreted as three faces of $1\times
1\times 1$ cubes in a three-dimensional space), as nonintersecting
paths (see the right-most part of Figure~\ref{tiling}), and as
dimers on the hexagonal lattice (see Figure~\ref{configur} in
Section \ref{EIoPiJ} below). Theory of dimer models is a rapidly
developing subject, see \cite{kn:Kenyon-lectures} for a recent
review and references.

In terms of nonintersecting paths, the initial $(a,b,c)$-measures
give an extra factor of 2 every time the left-most path passes by
the wall. Thus, it is natural to say that this path \emph{reflects
off} the wall.

\smallskip

\noindent\textbf{Random surface growth.} Any lozenge tiling is
uniquely determined by locations of lozenges of a single type. Let
us introduce coordinates on the plane as shown in Figures
\ref{tiling} and \ref{configur}, and mark the midpoints of all
vertical lozenges; call them \textit{particles}. Denote the
horizontal coordinates of all particles with vertical coordinate $m$
by $y^m_1>y^m_2>\dots$. Then $\mathcal M_t$ is a probability measure
on particle configurations
$$
\{y_k^m\mid k=1,2,\dots,[\tfrac {m+1}2];\ m=1,2,\dots\}\subset
\Z_{\ge 0}
$$
that satisfy the interlacing conditions
$y_{k+1}^{m+1}<y_k^m<y_k^{m+1}$ for all meaningful values of $k$ and
$m$.

We show that $\mathcal M_t$ is the time $t$ distribution of a
continuous time Markov chain defined as follows.

The initial condition is a single particle configuration when all
the particles are as much to the left as possible, i.e.
$y^m_k=m-2k+1$ for all $k,m$. Now let us describe the evolution.

We say that a particle $y^m_k$ is blocked on the right if
$y^m_k+1=y^{m-1}_{k-1}$, and it is blocked on the left if
$y^m_k-1=y^{m-1}_k$ (if the corresponding particle $y^{m-1}_{k-1}$
or $y^{m-1}_{k}$ does not exist, then $y^m_k$ is not blocked).

Each particle has two exponential clocks of rate $\frac 12$; all
clocks are independent. One clock is responsible for the right
jumps, while the other is responsible for the left jumps. When the
clock rings, the particle tries to jump by 1 in the corresponding
direction. If the particle is blocked, then it stays still. If the
particle is against the wall (i.e. $y^m_{[\frac {m+1}2]}=0$) and the
left jump clock rings, the particle is reflected, and it tries to
jump to the right instead.

When $y^m_k$ tries to jump to the right (and it is not blocked on
the right), we find the largest $r\in \Z_{\ge 0}\sqcup\{+\infty\}$
such that $y_k^{m+i}=y_k^{m}+i$ for $0\le i\le r$, and the jump
consists of all particles $\bigl\{y_k^{m+i}\bigr\}_{i=0}^r$ moving
to the right by 1. Similarly, when $y^m_k$ tries to jump to the left
(not being blocked on the left), we find the largest $l\in \Z_{\ge
0}\sqcup\{+\infty\}$ such that $y_{k+j}^{m+j}=y_k^m-j$ for $0\le
j\le l$, and the jump consists of all particles
$\bigl\{y_{k+j}^{m+j}\bigr\}_{j=0}^l$ moving to the left by 1.

In other words, the particles with smaller upper indices can be
thought of as heavier than those with larger upper indices, and the
heavier particles block and push the lighter ones so that the
interlacing conditions are preserved.

\begin{center}
\begin{figure}[htp]
\caption{First three jumps} \medskip
\includegraphics[height=2.8cm]{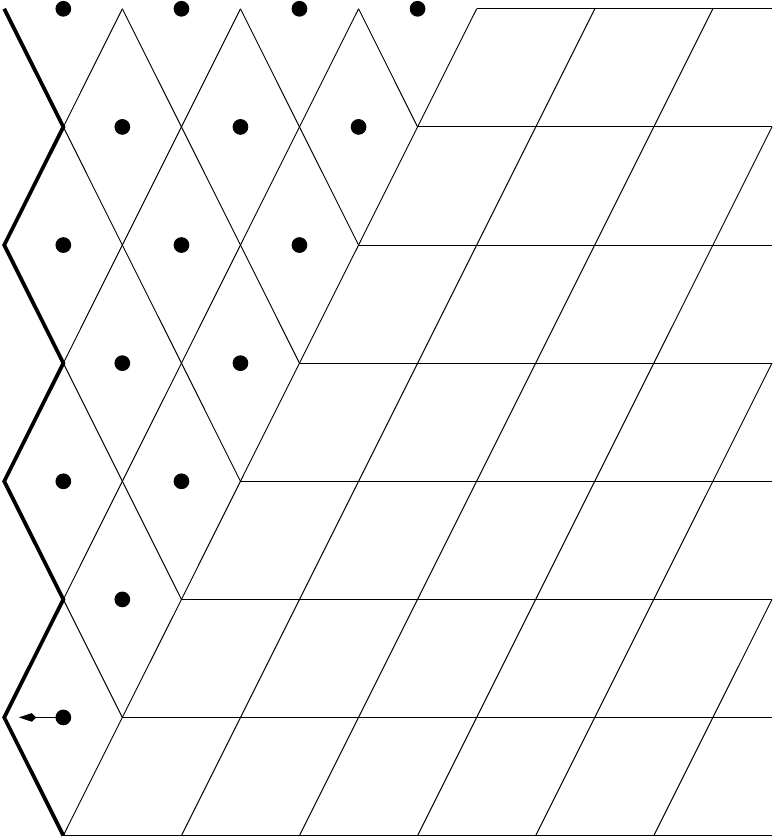}
\quad
\includegraphics[height=2.8cm]{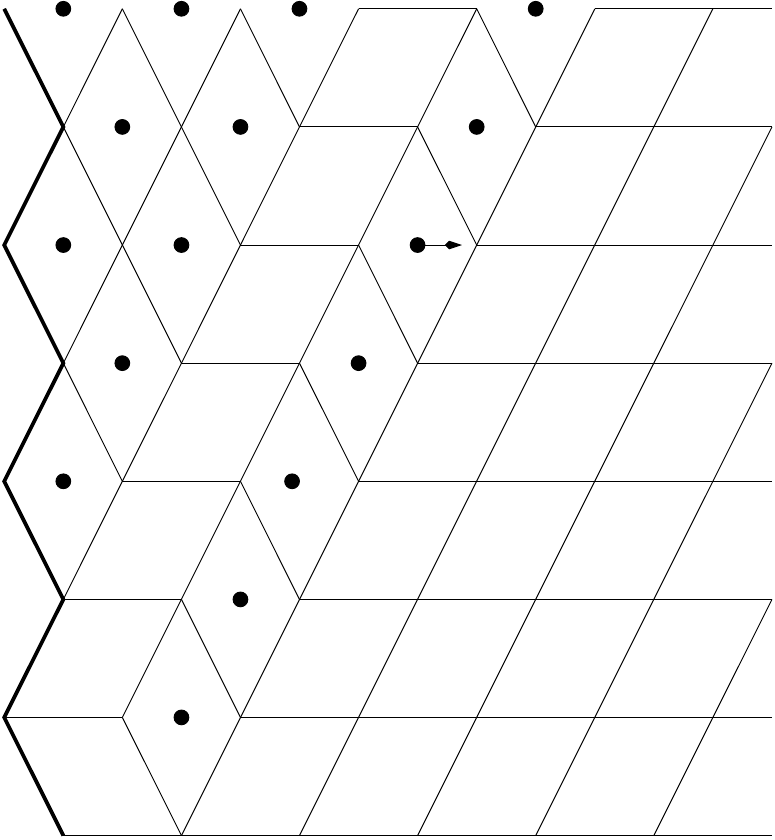}
\quad
\includegraphics[height=2.8cm]{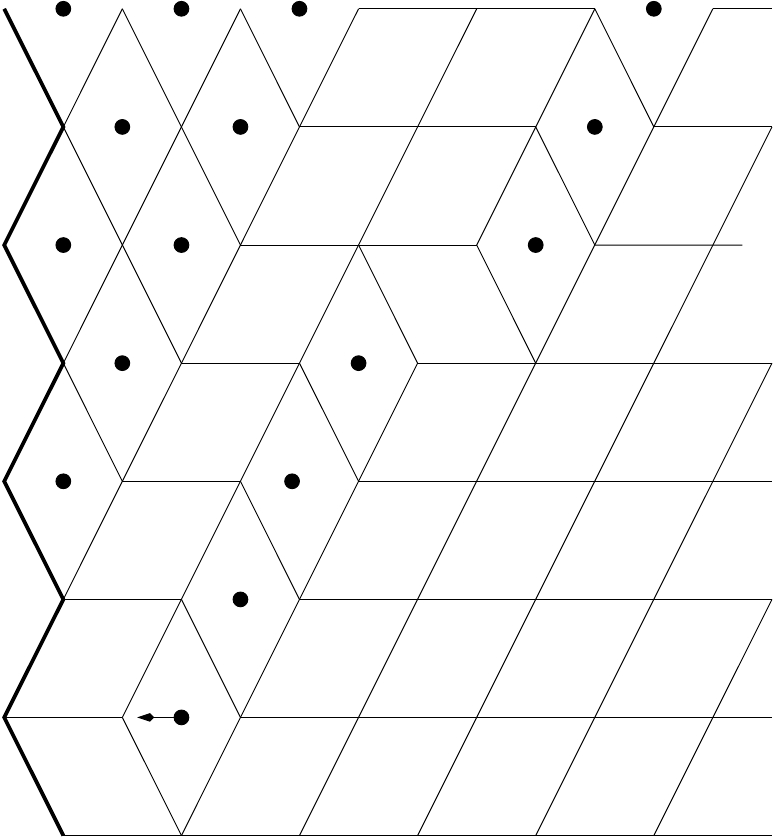}
\quad
\includegraphics[height=2.8cm]{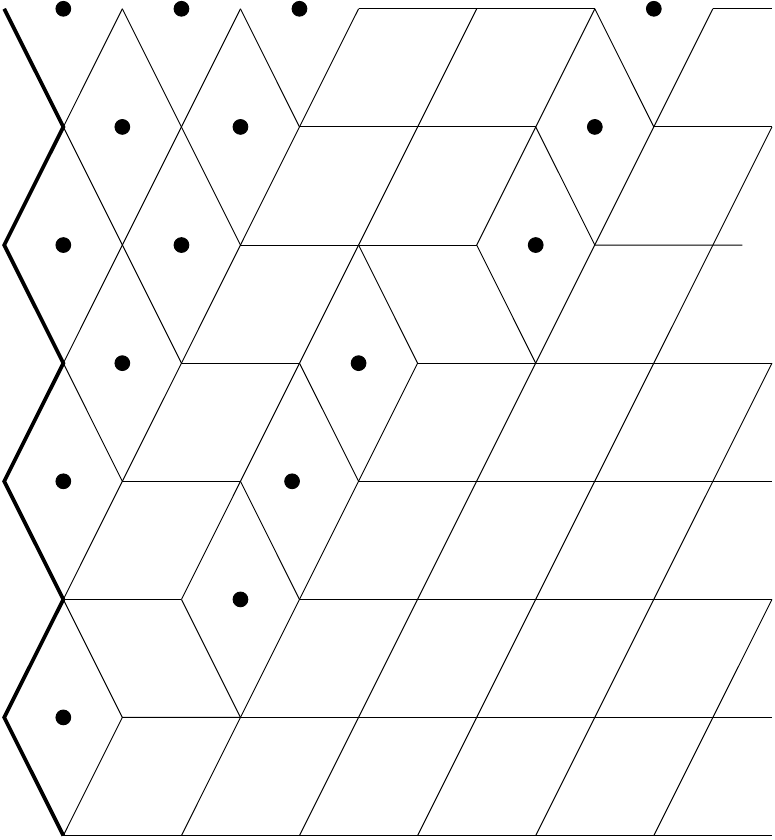} \label{jumps}
\end{figure}
\end{center}
Figure~\ref{jumps} depicts three possible first jumps: Left clock of
$y_1^1$ rings first (it gets reflected by the wall), then right
clock of $y_1^5$ rings, and then left clock of $y_1^1$ again.

In terms of the underlying stepped surface, the evolution can be
described by saying that we add possible ``sticks'' with base
$1\times 1$ and arbitrary length of a fixed orientation with rate
1/2, remove possible ``sticks'' with base $1\times 1$ and a
different orientation with rate 1/2, and the rate of removing sticks
that touch the left border is doubled.\footnote{This phrase is based
on the convention that
\includegraphics[height=0.4cm]{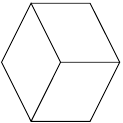} is a figure of a
$1\times1\times1$ cube. If one uses the dual convention that this is
a cube-shaped hole then the orientations of the sticks to be added
and removed have to be interchanged, and the tiling representations
of the sticks change as well.}
\begin{center}
\begin{figure}[htp]
\caption{Adding and removing ``sticks''} \medskip
\begin{center}
\includegraphics[height=1.8cm]{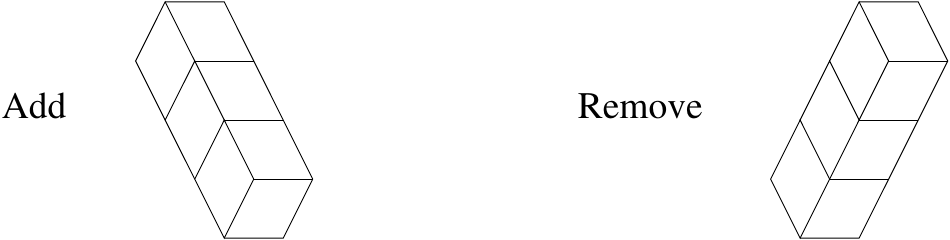} \label{sticks}
\end{center}
\end{figure}
\end{center}

A computer simulation of this dynamics can be found at\\
\href{http://www.math.caltech.edu/papers/Orth_Planch.html}
{$\mathtt{http://www.math.caltech.edu/papers/Orth\_Planch.html}$}.

Similar Markov chains have been previously studied in \cite{kn:BF}
without the wall, and in \cite{kn:WW} with a different
(``symplectic'') interaction with the wall.

\smallskip
\noindent\textbf{Representation Theory.} Let $O(N)$ be the group of
$N\times N$ orthogonal matrices with real entries. The group $O(N)$
is embedded in $O(N+1)$ as a subgroup of matrices fixing the
$(N+1)$st basis vector. Let $O(\infty)=\bigcup_{N=1}^\infty O(N)$ be
the infinite-dimensional orthogonal group.

The measures $\mathcal M_t$ are the Fourier transforms of the
distinguished one-pa\-ra\-me\-ter family of indecomposable
characters of $O(\infty)$ (the indecomposable characters of
$O(\infty)$ were classified in \cite{kn:OO} as a part of a solution
of a much more general problem). It is natural to call them the
\emph{Plancherel measures}. Details can be found in Section
\ref{MoP}.

Similarly defined Plancherel measures for the infinite symmetric
group $S(\infty)$ and the infinite-dimensional unitary group
$U(\infty)$ have been thoroughly studied, see \cite{kn:LS, kn:VK1,
kn:VK2, kn:BDJ1, kn:BDJ2, kn:OK, kn:BOO, kn:J, kn:K2} for
$S(\infty)$ and \cite{kn:K0, kn:Bi, kn:BK} for $U(\infty)$.

\smallskip
\noindent\textbf{Results.} We first prove, see Theorem \ref{exp}
below, that representation theoretic and Markov chain descriptions
of $\mathcal M_t$ given above are equivalent (the lozenge tiling
description of $\mathcal M_t$ is a simple corollary of the
representation theoretic one and Theorem 1.4 of \cite{kn:OO}). This
equivalence is far from being obvious, and we employ the general
formalism of \cite{kn:BF} to give a proof.

Our second result (Theorem \ref{theorem2}) shows that $\mathcal
M_t$, viewed as a measure on particle configurations $\{y^m_k\}$, is
a determinantal random point process (see Appendix \ref{Appendix A}
for basic definitions), and it also provides an explicit formula for
the correlation kernel. In fact, we prove such a result for random
point processes associated with arbitrary indecomposable characters
of $O(\infty)$.

We then focus on the asymptotics of $\mathcal M_t$ as $t\to\infty$. Note that, at first reading, one could look at the asymptotic
results without the construction in Sections \ref{MoP} and \ref{SD}.

As one might anticipate from previous results on dimer models and
Plancherel measures, cf. \cite{kn:Ken, kn:KO, kn:BF, kn:BK}, as
$t\to \infty$ the quarter-plane should split into ``frozen'' parts
and a ``liquid'' part. In each frozen part the tiling asymptotically
consists of lozenges of only one type, while in the liquid part the
random tiling locally (i.e. on the lattice scale) converges to the
unique (thanks to \cite{kn:Sh}) translation invariant Gibbs measure
of a certain slope; the slope depends on the location in the liquid
region. The underlying random surface should also converge, in a
suitable metric, to the deterministic smooth limit surface, and the
slopes of the Gibbs measures are the slopes of the tangent planes to
this limit shape.

In Theorem \ref{BulkLimits} we prove the statements about local
convergence. The frozen and liquid phases can be clearly seen in
Figure~\ref{RefereeFigure}. More exactly, we prove the
convergence of our correlation kernel to the \emph{incomplete
beta-kernel} first obtained in \cite{kn:Ken0, kn:OR0}, see
\cite{kn:BS} for a detailed discussion of the Gibbs properties of
the corresponding determinantal process. In Section
\ref{limitshape}, we also provide a formula for the hypothetical
limit shape, although we do not address the concentration of measure
phenomenon.

\begin{center}
\begin{figure}[htp]
\caption{The figure on the left is a computer simulation of the Markov chain at time 27. The first 320 levels are drawn. The figure on the right shows where the symmetric Pearcey and discrete Jacobi kernels appear.}
\begin{center}\includegraphics[height=5cm]{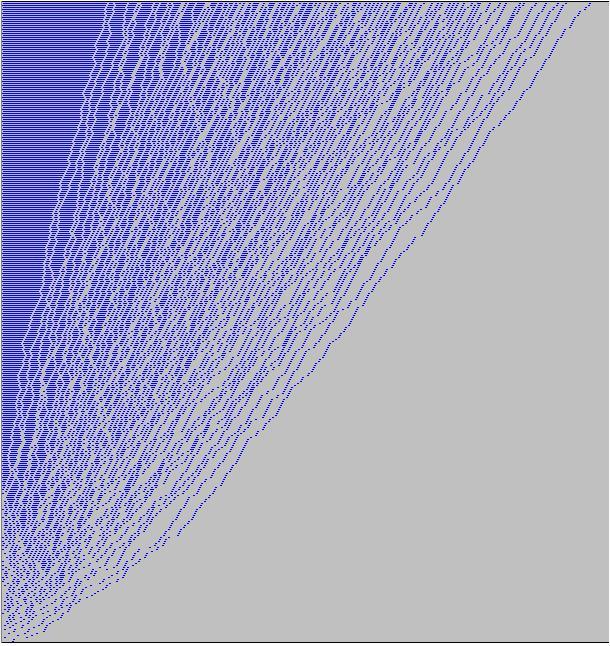}\ \ \ \ \  \includegraphics[height=5cm]{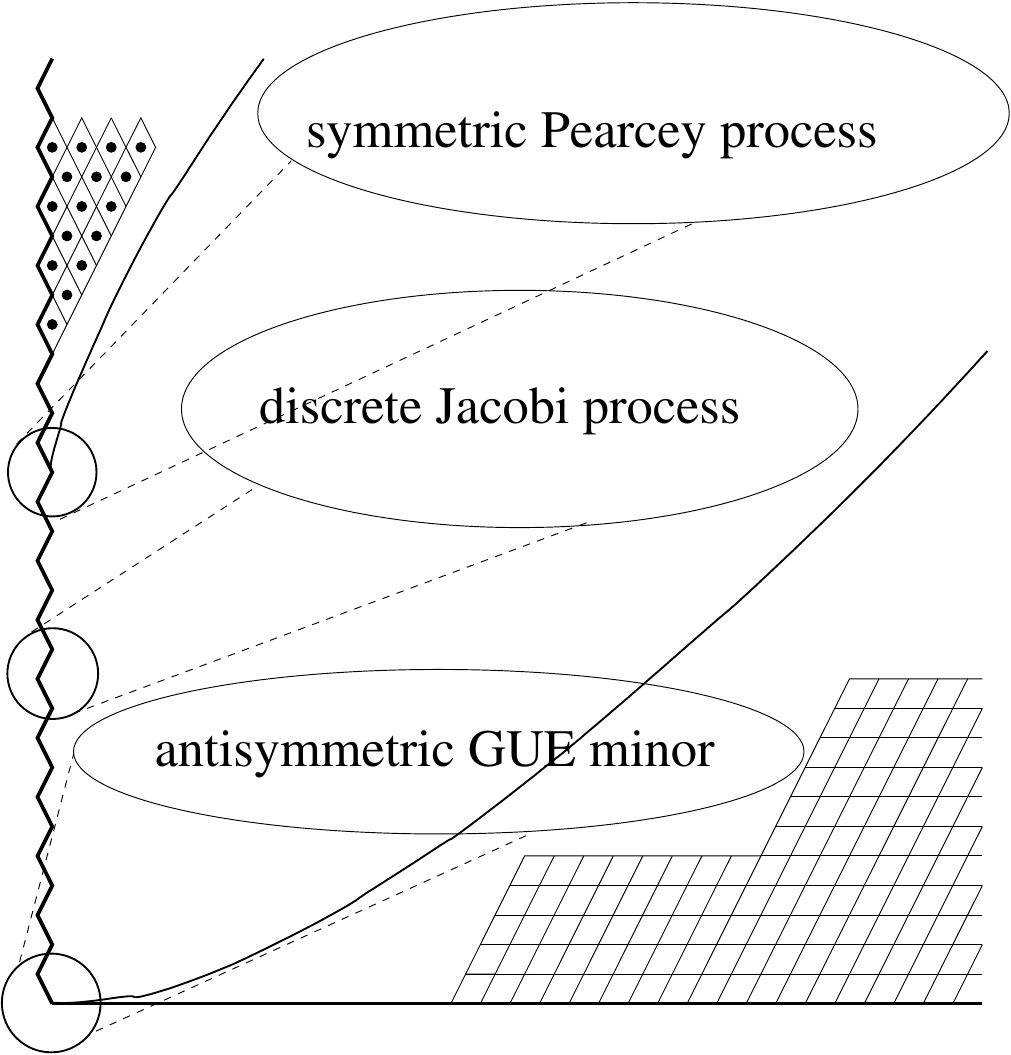}\end{center}
\label{RefereeFigure}
\end{figure}
\end{center}

From previously known results it is also natural to expect that near
the boundaries between frozen and liquid regions away from the wall,
our determinantal process converges in an appropriate scaling to the
so-called Airy process, see e.g. Section 4.5 of \cite{kn:BK} for an
analogous results in the case of Plancherel measures for
$U(\infty)$. This is indeed correct, and since the result and the
method of proving it are well known by now, we did not include them
in this paper.

The main novel feature of the model analyzed in this paper is the
wall, and we focus on the corresponding scaling limits.

The simplest case is the neighborhood of the origin (the corner of
the quarter-plane). Taking $t\to\infty$ asymptotics in Theorem
\ref{theorem1}, one can easily show (although we do not do this in
the paper) that as one scales the horizontal coordinate by
$\sqrt{t}$ and keeps the vertical coordinate finite, $\mathcal M_t$
converges to the \emph{antisymmetric GUE minor process} (aGUEM) of
\cite{kn:FN}, see also \cite{kn:D1, kn:D2}. Note that the way this
process was obtained in \cite{kn:FN} from lozenge tilings of a
half-hexagon is also similar to what we are doing. The aGUEM process
can also be obtained from the evolution of interacting Brownian
motions with a reflecting wall, which can be seen as a limit of the
Markov chain described above; see \cite{kn:BFS, kn:BFPSW} for
details.

The first genuinely new limit that we obtain takes place in the
region where the liquid part meets the wall. In Theorem
\ref{DiscreteJacobiKernel} we show that on the lattice scale our
determinantal process converges to a limiting determinantal process
on $\Z_{\ge 0}\times\Z$ that is translation invariant in the second
coordinate. We use the term \emph{discrete Jacobi kernel} for the
correlation kernel of the limiting process.

The second new determinantal process arises when we look near the
location where the boundary between frozen and liquid phases meets
the wall. In Theorem \ref{SymmetricPearceyKernel} we prove that as
one scales the horizontal coordinate by $t^\frac 14$ and the
vertical one by $t^\frac 12$, the correlation functions of our point
process converge to the determinants of the kernel $\mathcal
K(\sigma_1,\eta_1;\sigma_2,\eta_2)$ on $\R_+\times\R$ defines as
follows.

Let $C$ be the contour in $\mathbb C$ consisting of rays from
$\infty e^{i\pi/4}$ to $0$ to $\infty e^{-i\pi/4}$. Then

\begin{multline*}
\mathcal{K}(\sigma_1,\eta_1;\sigma_2,\eta_2)= \frac{2}{\pi^2 i}
\int\limits_{u\in C}\int\limits_{x\in \R_+}
e^{-\eta_1 x^2 + \eta_2 u^2 - x^4 + u^4}\cos(\sigma_1 x) \cos(\sigma_2 u)\frac{u\,dxdu}{u^2-x^2}\\
-\frac{1}
{2\sqrt{\pi(\eta_1-\eta_2)}}\left(\exp\frac{(\sigma_1+\sigma_2)^2}{4(\eta_2-\eta_1)}+
\exp\frac{(\sigma_1-\sigma_2)^2}{4(\eta_2-\eta_1)}\right)\mathbf{1}_{\eta_1>\eta_2}.
\end{multline*}

We call it the \emph{symmetric Pearcey kernel} because of the
similarity of the above expression to the Pearcey kernel that has
previously appeared in \cite{kn:ABK, kn:BK, kn:BH, kn:BH2, kn:OR,
kn:TW2}. Using the nonintersecting paths interpretation mentioned
above, it seems plausible that the symmetric Pearcey kernel should
also appear in the model treated in \cite{kn:KMW} at the critical
location when the paths touch the wall. Indeed, we were informed by the authors of \cite{kn:KMW} that this is indeed the case, cf. \cite{kn:KMW2}.

\textbf{Acknowledgements}. The authors are very grateful to Grigori
Olshanski for a number of valuable remarks. The first named author
(A.~B.) was partially supported by the NSF grant DMS-0707163.

\section{Measures on partitions}\label{MoP}
\subsection{Representations of Orthogonal Groups}\label{RoOG}
Let $O(N)$ denote the group of all real-valued $N\times N$ orthogonal matrices.
For each $N$, $O(N)$ is naturally embedded in $O(N+1)$ as the
subgroup fixing the $(N+1)$-st basis vector. Equivalently, $O\in
O(N)$ can be thought of as an $(N+1)\times(N+1)$ matrix by setting
$O_{i,N+1}=O_{N+1,j}=0$ for $1\leq i,j\leq N$ and $O_{N+1,N+1}=1$.
The union $\bigcup_{N=1}^{\infty} O(N)$ is denoted by $O(\infty)$.

Let us review some basic results from the representation theory of
finite- and infinite-dimensional orthogonal groups, see e.g.
\cite{kn:OO}.

A \textit{character} of $\Oinf$ is a positive definite function $\chi:\Oinf\rightarrow\mathbb{C}$ which is constant on conjugacy classes and normalized, i.e. $\chi(e)=1$. We further assume that $\chi$ is continuous on each $O(N)\subset \Oinf$. The set of all characters of $\Oinf$ is convex, and the extreme points of this set are called \textit{extreme characters}.

The set of extreme characters can be parametrized. Let $\R^{\infty}$ denote the product of countably many copies of $\R$. Let $\Omega$ be the set of all $(\alpha,\beta,\delta)$ such that

\[\alpha=(\alpha_1\geq\alpha_2\geq\ldots\geq 0)\in\mathbb{R}^{\infty},\ \beta=(\beta_1\geq\beta_2\geq\ldots\geq 0)\in\mathbb{R}^{\infty},\ \delta\in\R,\]
\[\displaystyle\sum_{i=1}^{\infty}(\alpha_i+\beta_i)\leq\delta.\]
Set
\[\gamma=\delta-\displaystyle\sum_{i=1}^{\infty}(\alpha_i+\beta_i)\geq 0.\]

The special orthogonal group, denoted by $SO(N)$, is the subgroup of
$O(N)$ consisting of matrices with determinant $1$. Let $O\in
SO(N)$. If $N=2m$ is even, then the spectrum of any $O$ is of the
form $\{z_1,z_1^{-1},\ldots,z_m,z_m^{-1}\}$, while if $N=2m+1$ is
odd, then the spectrum of any $O\in SO(N)$ is of the form
$\{z_1,z_1^{-1},\ldots,z_m,z_m^{-1},1\}$, where in both cases $z_i$
are complex numbers having absolute value $1$. For this paper, if
$\chi$ is a character of $O(\infty), SO(2m)$ or $SO(2m+1)$, then $\chi(O)$ is written interchangably with
$\chi(z_1,\ldots,z_m)$. Using this notation, any $\omega\in\Omega$
defines a function on $SO(\infty)=\bigcup_{N=1}^{\infty} SO(N)$ by (see Theorem 1.4 of
\cite{kn:OO})

\begin{equation}\label{ChiOmega}
\chi^{\omega}(O)=\displaystyle\prod_{j=1}^{m} E^{\omega}\left(\frac{z_j+z_j^{-1}}{2}\right),
\end{equation}
where

\[E^{\omega}\left(\frac{z+z^{-1}}{2}\right)=e^{\frac{\gamma}{2}(z+z^{-1}-2)}\displaystyle\prod_{i=1}^{\infty}\frac{(1+\frac{\beta_i}{2}(z-1))(1+\frac{\beta_i}{2}(z^{-1}-1))}{(1-\frac{\alpha_i}{2}(z-1))(1-\frac{\alpha_i}{2}(z^{-1}-1))},\]
or by setting $x=(z+z^{-1})/2$,
\[
E^{\omega}(x)=e^{\gamma(x-1)}\displaystyle\prod_{i=1}^{\infty}\frac{1-\beta_i(1-x)+\beta_i^2(1-x)/2}{1+\alpha_i(1-x)+\alpha_i^2(1-x)/2}.
\]

Note that the infinite product converges because $\sum(\alpha_i+\beta_i)$ is finite. As $\omega$ ranges over $\Omega$, the functions $\chi^{\omega}$ range over all the extreme characters of $\Oinf$ (Theorem 5.2 of \cite{kn:OO}).

A \textit{partition of length $\le N$} is a sequence of
nonincreasing nonnegative integers
$\lambda=(\lambda_1\geq\ldots\geq\lambda_N\geq 0)$. It is a
classical result that the set of all irreducible representations of
$SO(2N+1)$ over $\mathbb{C}$ is parameterized by partitions of length $\le N$. The
character of the irreducible representation of $SO(2N+1)$
parameterized by $\lambda$ will be denoted by
$\chi^{\lambda}_{SO(2N+1)}$, and its dimension by
$\dim_{SO(2N+1)}\lambda$. Similarly, the set of all irreducible
representations of $SO(2N)$ over $\mathbb{C}$ is parameterized by sequences of
integers $\lambda=(\lambda_1,\ldots\,\lambda_N)$ satisfying
$\lambda_1\geq\ldots\geq\lambda_{N-1}\geq\vert\lambda_N\vert$, and
the corresponding character and dimension are denoted by
${\chi}^{\lambda}_{SO(2N)}$ and $\dim_{SO(2N)}\lambda$. Let $\GT_N$
denote the set of all partitions of length $\le N$ ($\mathbb{J}$
stands for Jacobi, see below). For convenience of notation, let
$\lambda^*=(\lambda_1,\lambda_2,\ldots,\lambda_{N-1},-\lambda_N)$
for any $\lambda\in\GT_N$.

For any $\omega\in\Omega$, the restriction of $\chi^{\omega}$ to any $SO(M)$ defines two measures $P^{\omega}_{N,-1/2}$ and $P^{\omega}_{N,1/2}$ on $\GT_N$ by
\begin{eqnarray}\label{Consistency1}
\chi^{\omega}\vert_{SO(2N+1)}&=&\displaystyle\sum_{\lambda\in\GT_N} P^{\omega}_{N,1/2}(\lambda)\,\frac{\chi^{\lambda}_{SO(2N+1)}}{\dim_{SO(2N+1)}\lambda}\,,\\
\chi^{\omega}\vert_{SO(2N)}&=&\displaystyle\sum_{\lambda\in\GT_N}
P^{\omega}_{N,-1/2}(\lambda)\,\frac{\chi^{\lambda}_{SO(2N)}+\chi^{\lambda^*}_{SO(2N)}}{2\dim_{SO(2N)}\lambda}\,,
\label{Consistency2}
\end{eqnarray}
where $\chi^{\lambda}$ is the character of $SO(2N+1)$ or $SO(2N)$
parameterized by $\lambda$. Evaluating both sides of the equation at
the identity of the group shows that the sum of the weights is one.
Furthermore, the weights are nonnegative because the characters are
positive definite, so we obtain probability measures. Note that if the parameters of $\omega$ are
$\alpha=0,\beta=0,\gamma=0$, then $P_N^{\omega}$ is the delta
measure supported at the partition $(0,0,\ldots,0)$.

There is a useful explicit formula for $\chi^{\lambda}$. Let
$J_k^{(a,b)}(x)$ denote the $k$-th Jacobi polynomial with parameters
$a,b$ see e.g. \cite{kn:S}. Define the constant $c_k$ to be
\[
c_k=
\begin{cases}
\dfrac{1\cdot 3\cdot\ldots\cdot(2k-1)}{2\cdot 4\cdot\ldots\cdot 2k},&\text{if}\ \ k>0,\\
1,&\text{if}\ \  k=0,
\end{cases}
\]
and let $\mathsf{J}_k^{(a,b)}(x)=\Pk(x)/c_k$. The character of $O\in SO(2N)$ or $SO(2N+1)$ is
\begin{gather}\label{chiformula}
\chi^{\lambda}_{SO(2N+1)}(z_1,\ldots,z_N)
=\frac{\det\left[\mathsf{J}_{\lambda_i-i+N}^{(1/2,-1/2)}\left(\frac{z_j+z_j^{-1}}{2}\right)\right]_{i,j=1}^N}
{\det\left[\bigl(z_j+z_j^{-1}\bigr)^{N-i}\right]_{i,j=1}^N}\,,\\
\label{chiformula2}
(\chi^{\lambda}_{SO(2N)}+\chi^{\lambda^*}_{SO(2N)})(z_1,\ldots,z_N)
=\frac{
\det\left[2\mathsf{J}_{\lambda_i-i+N}^{(-1/2,-1/2)}\left(\frac{z_j+z_j^{-1}}{2}\right)\right]_{i,j=1}^N}{\det\left[\bigl(z_j+z_j^{-1}\bigr)^{N-i}\right]_{i,j=1}^N}\,.
\end{gather}
Expressions \eqref{chiformula} and \eqref{chiformula2} can be simplified using
\begin{align}\label{ztothes}
\mathsf{J}_s^{(1/2,-1/2)}\left(\frac{z+z^{-1}}{2}\right) =& \frac{z^{s+1/2}-z^{-s-1/2}}{z^{1/2}-z^{-1/2}},\\
\mathsf{J}_s^{(-1/2,-1/2)}\left(\frac{z+z^{-1}}{2}\right) =& \frac{z^s+z^{-s}}{2}.
\label{ztothes2}
\end{align}
The Jacobi polynomials also satisfy
\begin{eqnarray}\label{JacobiIdentities}
x\mathsf{J}_k^{(a,-1/2)}(x)&=&\frac{1}{2}\mathsf{J}_{k+1}^{(a,-1/2)}+\frac{1}{2}\mathsf{J}_{k-1}^{(a,-1/2)},\ \ a=\pm\frac{1}{2},\ \ k>0\\
x\mathsf{J}_0^{(1/2,-1/2)}(x)&=&-\frac{1}{2}\mathsf{J}_0^{(1/2,-1/2)}(x)+\frac{1}{2}\mathsf{J}_1^{(1/2,-1/2)}\\
x\mathsf{J}_0^{(-1/2,-1/2)}(x)&=&\mathsf{J}_1^{(-1/2,-1/2)}(x).
\label{JacobiIdentities2}
\end{eqnarray}

Explicit formulas for the dimensions are
\begin{multline*}
\dim_{SO(2N+1)}\lambda = \displaystyle\prod_{1\leq i<j\leq N}\frac{l_i^2-l_j^2}{m_i^2-m_j^2}\prod_{1\leq i\leq N}\frac{l_i}{m_i},\\
\text{where}\ l_i=\lambda_i+N-i+\tfrac 12,\ m_i=N-i+\tfrac 12
\end{multline*}
and
\[
\dim_{SO(2N)}\lambda = \displaystyle\prod_{1\leq i<j\leq
N}\frac{l_i^2-l_j^2}{m_i^2-m_j^2}, \ \text{where}\
l_i=\lambda_i+N-i,\ m_i=N-i.
\]

Let $h_k^{(a,b)}$ denote the squared norm of $\Pk$,
\[
h_k^{(a,b)}=\int_{-1}^1 J_k^{(a,b)}(x)^2 (1-x)^a (1+x)^b dx.
\]
Then
\begin{equation}\label{hk}
h_k^{(a,b)}=\pi c_k^2/W^{(a,b)}(k),\ \ \text{for}\
a=\pm\tfrac{1}{2},\ b=-\tfrac 12,
\end{equation}
where
\[
W^{(a,b)}(k)=
\begin{cases}
2,\ \ \text{if}\ \ k>0,a=b=-\frac{1}{2}\\
1,\ \ \text{if}\ \ k=0,a=b=-\frac{1}{2}\\
1,\ \ \text{if}\ \ k\geq 0,a=\frac{1}{2},b=-\frac{1}{2}
\end{cases}
\]
For proofs of these equations, see \S 1 of \cite{kn:OO2}, Chapter 4 of \cite{kn:S}, and Chapter 24 of \cite{kn:FH}.

In Section \ref{EFftM}, a formula for $P^{\omega}_{N,a}$ will be proved.

\subsection{Central Measures}\label{CM}
For $a=\pm 1/2$, the measure $P^{\omega}_{N,a}$ on $\GT_N$ can be extended to a more general measure $P^{\omega}$. The purpose of this section is to explain how $P^{\omega}$ is constructed.

Let $\GT_{N,-}$ and $\GT_{N,+}$ be two copies of $\GT_N$. Set
$\GT=\bigcup_{N\geq 1} (\GT_{N,-}\cup\GT_{N,+})$. Turn $\GT$ into a
graph as follows. Draw an edge between $\lambda\in\GT_{N,-}$ and
$\mu\in\GT_{N,+}$ if
$0\leq\lambda_N\leq\mu_N\leq\ldots\leq\lambda_1\leq\mu_1$. Draw an
edge between $\lambda\in\GT_{N,+}$ and $\mu\in\GT_{N+1,-}$ if
$\mu_{N+1}\leq\lambda_N\leq\mu_{N}\leq\lambda_{N-1}\leq\ldots\leq\lambda_1\leq\mu_1$.
It will be convenient to set $\lambda_{N+1}=0$, which gives the
additional inequality $\lambda_{N+1}\leq\mu_{N+1}$. In either case,
use the notation $\lambda\prec\mu$.

Note that $\lambda\prec\mu$ is equivalent to the following relation from representation theory. Let $V_{\mu}$ be the representation of $SO(M)$ corresponding to $\mu$ and let $V_{\lambda}$ be the representation of $SO(M-1)$ corresponding to $\lambda$. With this notation, $V_{\lambda}$ is a subrepresentation of $V_{\mu}\vert_{SO(M-1)}$ iff $\lambda\prec\mu$. See \cite{kn:Z}.

For any $(\lambda,\mu)\in\GT_{N-1,+}\times\GT_{N,-}$, set
\[
\varkappa(\lambda,\mu)=
\begin{cases}
1, \ \text{if} \ \lambda\prec\mu,\\
0, \ \text{if} \ \lambda\not\prec\mu.
\end{cases}
\]
If $(\lambda,\mu)\in\GT_{N,-}\times\GT_{N,+}$, set
\[
\varkappa(\lambda,\mu)=
\begin{cases}
2, \ \text{if} \ \lambda\prec\mu,\ \lambda_N>0,\\
1, \ \text{if} \ \lambda\prec\mu,\ \lambda_N=0,\\
0, \ \text{if} \ \lambda\not\prec\mu.
\end{cases}
\]
The definition of $\varkappa(\lambda,\mu)$ is motivated by the branching rules
\begin{eqnarray}\label{Branching1}
V_{\mu}\vert_{SO(2N)} &=& \displaystyle\bigoplus_{\begin{subarray}{c} \lambda\prec\mu\\ \lambda_N>0 \end{subarray}} (V_{\lambda}\oplus V_{\lambda^*}) \oplus \bigoplus_{\begin{subarray}{c} \lambda\prec\mu\\ \lambda_N=0 \end{subarray}} V_{\lambda},\ \lambda\in\GT_{N,-},\ \mu\in\GT_{N,+}\\
V_{\mu}\vert_{SO(2N-1)} &=& \bigoplus_{\lambda\prec\mu} V_{\lambda},\ \lambda\in\GT_{N-1,+},\ \mu\in\GT_{N,-}.
\label{Branching2}
\end{eqnarray}

A \textit{path} in $\GT$ is a sequence $t=(t^{1,-}\prec t^{1,+}\prec t^{2,-}\prec\ldots)$ such that $t^{i,-}\in\GT_{i,-}$ and $t^{i,+}\in\GT_{i,+}$ for $1\leq N$. Let $\GT_{paths}$ denote the set of all paths in $\GT$. There are also \textit{finite paths}, which are sequences $u=(u^{1,-}\prec u^{1,+}\prec u^{2,-} \prec \ldots \prec u^{N,\pm})$ that end at some $u^{N,\pm}$. Given such a finite path $u$, define the cylindrical set
\[
C_u = \{t\in\GT_{paths} :  t^{1,-}=u^{1,-},t^{1,+}=u^{1,+},\ldots,t^{N,\pm}=u^{N,\pm}\}.
\]
Also for a finite path, define the weight $w_u$ to be
\[
w_u = \varkappa(u^{1,-},u^{1,+})\varkappa(u^{1,+},u^{2,-})\ldots.
\]
The brancing rules imply that if we sum $w_u$ over all finite paths $u$ that end at $\lambda$, we get $\dim\lambda=\dim V_{\lambda}$. Note that $\dim V_{\lambda}=\dim V_{\lambda^*}$.

A probability measure $\vert\cdot\vert$ on $\GT_{paths}$ is called \textit{central} if
\[
\frac{\vert C_u\vert}{w_u} = \frac{\vert C_v \vert}{w_v}
\]
for any two finite paths $u,v$ that end at the same partition. For a
more general definition of central measures, see section 6 of
\cite{kn:K1}.

Each $\omega\in\Omega$ defines a central measure $P^{\omega}$ on $\GT_{paths}$ as follows. The measure of any cylindrical set $C_u$ is given by $P_{N,\pm}^{\omega}(u_{N,\pm})/\dim u_{N,\pm}$. This measure is well-defined because of the consistency relations (cf. \eqref{Consistency1}-\eqref{Consistency2})
\[
\displaystyle\sum_{\mu\in\GT_{N,+}} P_{N,+}^{\omega}(\mu)\varkappa(\lambda,\mu) = P_{N,-}^{\omega}(\lambda),
\]
\[
\displaystyle\sum_{\mu\in\GT_{N,-}} P_{N,-}^{\omega}(\mu)\varkappa(\lambda,\mu) = P_{N-1,+}^{\omega}(\lambda).
\]

\subsection{\texorpdfstring{Equivalent Interpretations of Paths in $\mathbb{J}$}{Equivalent Interpretations of Paths in J}}\label{EIoPiJ}
It will be useful to interpret measures on $\mathbb{J}_{paths}$ as random point processes on a two-dimensional lattice or as random lozenge tilings of a quarter plane.

Set
\[
\mathfrak{X}=\Z_{\geq 0}\times\Z_{>0}\times\{\pm\tfrac{1}{2}\}, \ \
\mathfrak{Y}=\Z_{\geq 0}\times\Z_{>0}.
\]

For $(n_j,a_j)\in\Z_{>0}\times\{\pm \frac 12\}$, $j=0,1$, we write
$(n_0,a_0)\triangleleft (n_1,a_1)$ if $2n_0+a_0<2n_1+a_1$. Also
write $(n_0,a_0)\trianglelefteq (n_1,a_1)$ if $2n_0+a_0\leq
2n_1+a_1$. That is,
\[
(1,-1/2)\triangleleft (1,1/2) \triangleleft (2,-1/2) \triangleleft\ldots.
\]
Furthermore, set
$d(n_0,a_0;n_1,a_1)=\vert 2(n_1-n_0)+a_1-a_0 \vert$. In other words,
$d(n_0,a_0;n_1,a_1)$ is the distance between the levels $(n_0,a_0)$
and $(n_1,a_1)$.

We identify $\mathfrak{X}$ and $\mathfrak{Y}$ via the bijection
\begin{equation*}
\iota:\mathfrak{X} \rightarrow \mathfrak{Y},\qquad (x,n,a) \mapsto
 \left(2x+a+\tfrac{1}{2},2n+a-\tfrac{1}{2}\right).
\end{equation*}
To any finite or infinite path
$\boldsymbol{\lambda}=(\lambda^{(1),-1/2}\prec\lambda^{(1),1/2}\prec\ldots)$ in
$\GT$, we associate two point configurations (subsets)
$\mathcal{L}_{\mathfrak{X}}(\boldsymbol{\lambda})\subset\mathfrak{X}$ and
$\mathcal{L}_{\mathfrak{Y}}(\boldsymbol{\lambda})\subset\mathfrak{Y}$ as follows:
\[
\mathcal{L}_{\mathfrak{X}}(\boldsymbol{\lambda})=\left\{(x_k^{(n),a},n,a) : 1\leq
k\leq n,\ \ a\in\{\pm\tfrac{1}{2}\},\ \ n\geq 1\right\},
\]
\[
\mathcal{L}_{\mathfrak{Y}}(\boldsymbol{\lambda})=\left\{\left(y_k^{2n+a-1/2},2n+a-\tfrac{1}{2}\right)
: 1\leq k\leq n,\ \ a\in\{\pm\tfrac{1}{2}\},\ \ n\geq 1\right\},
\]
where
\[
x_k^{(n),a}=\lambda_k^{(n),a}+n-k, \ \
y_k^{2n+a-1/2}=2(\lambda_k^{(n),a}+n-k)+a+\tfrac{1}{2}.
\]
Note that
$\iota(\mathcal{L}_{\mathfrak{X}}(\boldsymbol{\lambda}))=\mathcal{L}_{\mathfrak{Y}}(\boldsymbol{\lambda})$
for any $\boldsymbol{\lambda}$.

In Figure~\ref{configur}, black dots mark the elements of
$\mathcal{L}_{\mathfrak{Y}}(\boldsymbol{\lambda})$ for
$\boldsymbol{\lambda}=((1)\prec (1)\prec (1,0)\prec (2,1)\prec
(2,1,1)\prec (3,1,1))$.

\begin{center}
\begin{figure}[htp]
\caption{Elements of $\mathbb{J}_{paths}$ can be interpreted as particles, lozenges, or dimers.}
\begin{center}\includegraphics[height=5cm]{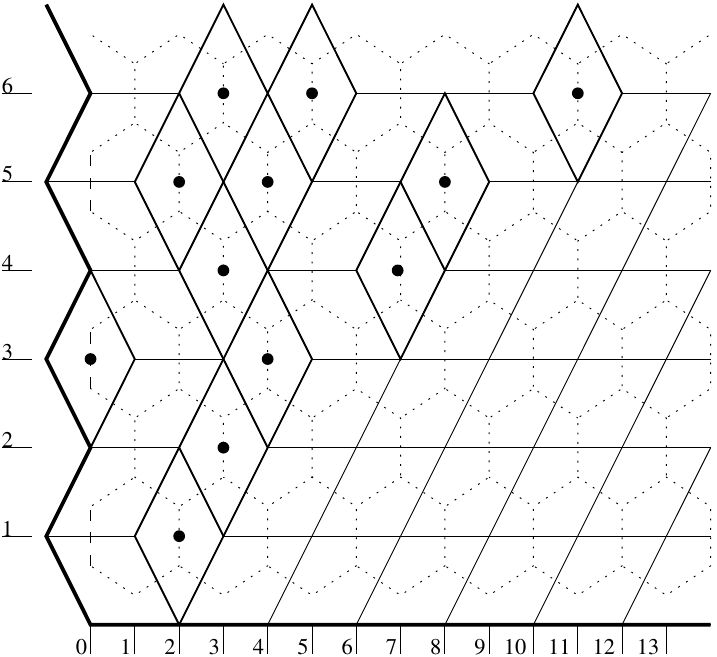}\end{center}
\label{configur}
\end{figure}
\end{center}

The interlacing property for paths in $\GT$ turns into
\[
y_{k+1}^{m+1}<y_k^{m}<y_k^{m+1},\ \ 1\leq m, k\leq \left[\frac{(m+1)}{2}\right].
\]

This construction gives rise to a bijection between infinite paths
in $\mathbb{J}$ and certain lozenge tilings of the quarter plane
with boundary as indicated in Figure~\ref{configur}. Elements of
$\mathcal{L}_{\mathfrak{Y}}(\boldsymbol{\lambda})$ correspond to
centers of lozenges of one specific type, as in Figure
\ref{configur}. The image of $\GT_{paths}$ consists of lozenge
tilings with $\left[\frac{m+1}{2}\right]$ lozenges of this type on
the $m$th horizontal row, for $m\geq 1$.

Equivalently, lozenge tilings can be viewed as dimers on the dual
hexagonal lattice, see e.g. \cite{kn:Kenyon-lectures}. In this language,
paths in $\GT$ correspond to dimers with exactly
$\left[\frac{m+1}{2}\right]$ vertical edges crossing the $m$th
horizontal line.

Thus, measures on $\GT_{paths}$ yield random point processes on
$\mathfrak{X}$ and $\mathfrak{Y}$, a measure on lozenge tilings, and
a measure on dimers. The \textit{centrality} of such a measure on
dimers can now be phrased in the following Gibbs-like manner: Assign
a weight of $1/2$ to all the vertical edges of the hexagonal lattice
that lie on the veritcal line with coordinate $0$. These are marked
by dashed lines in Figure~\ref{configur}. All other edges have a
weight of $1$. Then, given the set of $\left[\frac{m+1}{2}\right]$
vertical edges crossing the $m$th horizontal line, the conditional
distribution of the dimers below this line is proportional to the
product of the edge weights of the dimers.

Recall from sections \ref{RoOG} and \ref{CM} that any
$\omega\in\Omega$ defines a probability measure $P^{\omega}$ on
$\GT_{paths}$. Let $\mathcal{P}^{\omega}_{\mathfrak{X}}$ and
$\mathcal{P}^{\omega}_{\mathfrak{Y}}$ denote the resulting point
processes on $\mathfrak{X}$ and $\mathfrak{Y}$, respectively. If the
parameters of $\omega$ are such that all $\alpha_i=\beta_i=0$, then
let $\rho_{k,\mathfrak{X}}^{\gamma}$ and
$\rho_{k,\mathfrak{Y}}^{\gamma}$ denote the $k$th correlation
functions of $\mathcal{P}^{\omega}_{\mathfrak{X}}$ and
$\mathcal{P}^{\omega}_{\mathfrak{Y}}$, respectively. See Appendix
\ref{Appendix A} for general definitions of point processes. We will
need the correlation functions in section \ref{AotK}.

\subsection{Preliminary Lemmas}
Before continuing, a couple of lemmas will be needed. Since they will be used several times throughout this paper, it is convenient to gather them in this section. Lemma \ref{Hua} is a variant of the Cauchy-Binet formula.

\begin{lemma}\label{Hua}
For each nonnegative integer $k$, let $e_k$ and $g_k$ be some functions on $\mathbb{C}$. For $1\leq i\leq N$,  let $x_i,\zeta_i$ be complex numbers such that $\sum_{k=0}^{\infty} e_k(\zeta_j)g_k(x_i)$ converges absolutely for all $1\leq i,j\leq N$. Then
\begin{multline*}
\displaystyle\sum_{k_1>\ldots>k_N\geq
0}\det[e_{k_i}(\zeta_j)]_{i,j=1}^N \det[g_{k_i}(x_j)]_{i,j=1}^N
=\det\left[\displaystyle\sum_{k=0}^{\infty}e_k(\zeta_j)g_k(x_i)\right]_{i,j=1}^N.
\end{multline*}
\end{lemma}
\begin{proof}
The proof is almost identical to the proof of Theorem 1.2.1 of \cite{kn:H}.
\end{proof}

The next lemma is also useful.

\begin{lemma}\label{DeltaDistribution}
For $a=\pm 1/2, -1\leq\zeta \leq 1$, and a test function $T\in C^1[-1,1]$,
\[
\sum_{k=0}^{\infty}\int_{-1}^1 \frac{J_k^{(a,-1/2)}(x)J_k^{(a,-1/2)}(\zeta)}{h_k^{(a,-1/2)}}\, T(x)(1-x)^a(1+x)^{-1/2}dx=T(\zeta).
\]
\end{lemma}
\begin{proof}
Plugging in $a=-1/2$ and setting $x=\cos\phi$, $\zeta=\cos\theta$, (4.1.7) of \cite{kn:S} and \eqref{hk} give
$$
\frac{J_k^{(-1/2,-1/2)}(x)J_k^{(-1/2,-1/2)}(\zeta)}{h_k^{(-1/2,-1/2)}} =
\begin{cases}
\frac{1}{\pi} & \text{\ \ if } k = 0, \\
\frac{2}{\pi}\cos k\phi\cos k\theta & \text{\ \ if } k > 0.
\end{cases}
$$
Since $T$ is $C^1$, the Fourier series of $T$ converges to $T$ (see Chapter 3, Section 6 of \cite{kn:T}):
\[
T(\cos\phi)=\hat{T}_0+\hat{T}_1\cos\phi+\hat{T}_2\cos 2\phi+\ldots,
\]
where
\[
\hat{T}_k=
\begin{cases}
\dfrac{1}{\pi}\displaystyle\int_0^{\pi} T(\cos\phi)d\phi,&\ k=0,\\
\dfrac{2}{\pi}\displaystyle\int_0^{\pi} T(\cos\phi)\cos k\phi
d\phi,&\ k>0.
\end{cases}
\]
Therefore
\begin{eqnarray*}
\displaystyle&&\sum_{k=0}^{\infty}\int_{-1}^1 T(x)\frac{\Pk(x)\Pk(\zeta)}{h_k^{(a,b)}}(1-x)^{a}(1+x)^{b}dx\\
&=&\frac{1}{\pi}\int_0^{\pi}T(\cos\phi)d\phi+\frac{2}{\pi}\displaystyle\sum_{k=1}^{\infty}\int_0^{\pi}T(\cos\phi)\cos k\phi\cos k\theta d\phi\\
&=& \hat{T}_0+\hat{T}_1\cos\theta+\hat{T}_2\cos 2\theta+\ldots\\
&=& T(\zeta).
\end{eqnarray*}

In the case when $a=1/2$, (4.1.8) of \cite{kn:S} and \eqref{hk} tell us
\[
\frac{J_k^{(1/2,-1/2)}(x)J_k^{(1/2,-1/2)}(\zeta)}{h_k^{(1/2,-1/2)}}=\frac{\sin(k+1/2)\phi}{\sin\phi/2}\frac{\sin(k+1/2)\theta}{\sin\theta/2},
\]
and the rest of the argument is similar.
\end{proof}
The previous two lemmas also imply the next one.
\begin{lemma}\label{BorodinFerrari}
Let $-1\leq\zeta_1,\ldots,\zeta_N\leq 1$, and suppose $\varphi\in C^1[-1,1]$. Let $l_1,\ldots,l_N$ be nonnegative
integers. Set
\[
e_k^{(a,-1/2)}=\frac{J_k^{(a,-1/2)}(x)}{[h_k^{(a,-1/2)}]^{1/2}}
\]
where $a=\pm 1/2$.
Then
\begin{equation*}
\sum_{k_1>\ldots>k_N\geq
0}\det\left[e_{k_j}^{(a,-1/2)}(\zeta_i)\right]_{i,j=1}^N\det[g(k_j,l_i)]_{i,j=1}^N
=\varphi(\zeta_1)\ldots
\varphi(\zeta_N)\det\left[e_{l_j}^{(a,-1/2)}(\zeta_i)\right]_{i,j=1}^N
\end{equation*}
where
\[
g(k,l)=\int_{-1}^1 e_k(x)e_l(x)\varphi(x)(1-x)^{a}(1+x)^{-1/2}dx.
\]
\end{lemma}
\begin{proof}
By Lemma \ref{Hua},
\begin{equation*}
\sum_{k_1>\ldots>k_N\geq 0}\det\left[e_{k_j}^{(a,-1/2)}(\zeta_i)\right]_{i,j=1}^N
\det[g(k_j,l_i)]_{i,j=1}^N
=\det\left[\displaystyle\sum_{k=0}^{\infty}e_k^{(a,-1/2)}(\zeta_i)g(k,l_j)\right]_{i,j=1}^N.
\end{equation*}
By Lemma \ref{DeltaDistribution},
\begin{eqnarray*}
\displaystyle\sum_{k=0}^{\infty}e_k^{(a,-1/2)}(\zeta_i)g(k,l_j) &=& \sum_{k=0}^{\infty}\int_{-1}^1 e_k^{(a,-1/2)}(x)e_k^{(a,-1/2)}(\zeta_i)e_{l_j}^{(a,-1/2)}(x)\varphi(x)(1-x)^{a}(1+x)^{-1/2}dx\\
&=& e_{l_j}^{(a,-1/2)}(\zeta_i)\varphi(\zeta_i).
\end{eqnarray*}
\end{proof}

\begin{lemma}\label{UsefulIdentities} The normalized Jacobi polynomials $\mathsf{J}^{(\pm 1/2,-1/2)}_s$
satisfy the following properties:

\textrm{(a)}\quad $\displaystyle\sum_{r=0}^s
W^{(-1/2,-1/2)}(r)\mathsf{J}_r^{(-1/2,-1/2)}(x)=\mathsf{J}_s^{(1/2,-1/2)}(x)$,

\textrm{(b)}\quad $\displaystyle\sum_{r=0}^{s-1}
\mathsf{J}_r^{(1/2,-1/2)}(x) =
\frac{\mathsf{J}_s^{(-1/2,-1/2)}(x)-1}{x-1}$,

\textrm{(c)}\quad $\displaystyle\frac{1}{\pi}\int_{-1}^1
\mathsf{J}_s^{(1/2,-1/2)}(x)(1-x)^{-1/2}(1+x)^{-1/2}dx=1.$
\end{lemma}
\begin{proof}
(a),(b) Let $z$ be on the unit circle such that $(z+z^{-1})/2=x$. Using \eqref{ztothes} and \eqref{ztothes2}, the sum becomes a geometric series, which can be evaluated explicitly.

(c) By part (a), the integral equals
\[
\displaystyle\sum_{r=0}^s \frac{W^{(-1/2,-1/2)}(r)}{\pi} \int_{-1}^1 \mathsf{J}_r^{(-1/2,-1/2)}(x)\mathsf{J}_0^{(-1/2,1/2)}(x)(1-x)^{-1/2}(1+x)^{-1/2}dx,
\]
which equals $1$ by the orthogonality relations.
\end{proof}

\begin{lemma}\label{NeededSum} Let $T\in C^1[-1,1]$. The following identities hold:

(a)
\[
\displaystyle\sum_{r=0}^{\infty}\frac{W^{(-1/2,-1/2)}(r)}{\pi}\int_{-1}^1 \mathsf{J}_r^{(-1/2,-1/2)}(x)T(x)(1-x)^{-1/2}(1+x)^{-1/2}dx =T(1).
\]

(b)
\begin{multline*}
\displaystyle\sum_{r=s+1}^{\infty}\frac{W^{(-1/2,-1/2)}(r)}{\pi}\int_{-1}^1 \mathsf{J}_r^{(-1/2,-1/2)}(x)T(x)(1-x)^{-1/2}(1+x)^{-1/2}dx\\
=\frac{1}{\pi}\int_{-1}^1 \mathsf{J}_s^{(1/2,-1/2)}(x)(T(1)-T(x))(1-x)^{-1/2}(1+x)^{-1/2}dx.
\end{multline*}
\end{lemma}
\begin{proof}
(a) Note that
\begin{align*}
&\displaystyle\sum_{r=0}^{\infty}\frac{W^{(-1/2,-1/2)}(r)}{\pi}\int_{-1}^1 \mathsf{J}_r^{(-1/2,-1/2)}(x)T(x)(1-x)^{-1/2}(1+x)^{-1/2}dx\\
=&\displaystyle\sum_{r=0}^{\infty}\int_{-1}^1 \frac{J_r^{(-1/2,-1/2)}(x)J_r^{(-1/2,-1/2)}(1)}{h_r^{(-1/2,-1/2)}}T(x)(1-x)^{-1/2}(1+x)^{-1/2}dx\\
=&T(1)
\end{align*}
The first equality follows from \eqref{hk} and $\mathsf{J}_r^{(-1/2,-1/2)}(1)=1$, and the second equality follows from Lemma \ref{DeltaDistribution}.

(b) By Lemma \ref{UsefulIdentities}(c),
\[
T(1)=\frac{1}{\pi}\int_{-1}^1 T(1)\mathsf{J}_s^{(1/2,-1/2)}(x)(1-x)^{-1/2}(1+x)^{-1/2}dx.
\]
By Lemma \ref{UsefulIdentities}(a),
\begin{multline*}
\displaystyle\sum_{r=0}^{s}\frac{W^{(-1/2,-1/2)}(r)}{\pi}\int_{-1}^1 \mathsf{J}_r^{(-1/2,-1/2)}(x)T(x)(1-x)^{-1/2}(1+x)^{-1/2}dx\\
=\frac{1}{\pi}\int_{-1}^1 T(x)\mathsf{J}_s^{(1/2,-1/2)}(x)(1-x)^{-1/2}(1+x)^{-1/2}dx.
\end{multline*}
Subtracting this from the sum in (a) proves (b).
\end{proof}

Since the series in Lemma \ref{NeededSum}(a) converges,
\begin{corollary}\label{NeededLimit} For $T\in C^1[-1,1]$,
\[
\displaystyle\lim_{r\rightarrow\infty} \int_{-1}^1 \mathsf{J}_r^{(-1/2,-1/2)}(x)T(x)(1-x)^{-1/2}(1+x)^{-1/2}dx=0.
\]
\end{corollary}

\begin{lemma}\label{NeededSum2} For $T\in C^1[-1,1]$,
\begin{multline*}
\displaystyle\sum_{r=s}^{\infty}\frac{W^{(1/2,-1/2)}(r)}{\pi}\int_{-1}^1 \mathsf{J}_r^{(1/2,-1/2)}(x)T(x)(1-x)^{1/2}(1+x)^{-1/2}dx\\
=\frac{1}{\pi}\int_{-1}^1 \mathsf{J}_s^{(-1/2,-1/2)}(x)T(x)(1-x)^{-1/2}(1+x)^{-1/2}dx.
\end{multline*}
\end{lemma}
\begin{proof}
By Lemma \ref{UsefulIdentities}(b),
\begin{multline*}
\displaystyle\sum_{r=0}^{s-1}\frac{W^{(1/2,-1/2)}(r)}{\pi}\int_{-1}^1 \mathsf{J}_r^{(1/2,-1/2)}(x)T(x)(1-x)^{1/2}(1+x)^{-1/2}dx\\
=\displaystyle\frac{1}{\pi}\int_{-1}^1 \frac{\mathsf{J}_s^{(-1/2,-1/2)}(x)-1}{x-1} T(x)(1-x)^{1/2}(1+x)^{-1/2}dx
\end{multline*}

Taking $s\rightarrow\infty$ and using Corollary \ref{NeededLimit},
\begin{multline*}
\displaystyle\sum_{r=0}^{\infty}\frac{W^{(1/2,-1/2)}(r)}{\pi}\int_{-1}^1 \mathsf{J}_r^{(1/2,-1/2)}(x)T(x)(1-x)^{1/2}(1+x)^{-1/2}dx\\
=\frac{1}{\pi}\displaystyle\int_{-1}^1 T(x)(1-x)^{-1/2}(1+x)^{1/2}dx.
\end{multline*}

Subtracting these two sums proves the lemma.
\end{proof}

\subsection{Explicit Formula for the Measures}\label{EFftM}
In this section, we prove the following statement:
\begin{theorem}\label{theorem1}
For any $\lambda=(\lambda_1\geq\ldots\geq\lambda_N)\in\GT_N$, $a=\pm 1/2$, and $\omega\in\Omega$,
\[P^{\omega}_{N,a}(\lambda)=C_{N,a}\cdot\det[f_j^{(N,a)}(\lambda_i-i+N)]_{1\leq i,j\leq N} \cdot \dim_{SO(2N+1/2+a)}\lambda.\]
where
\begin{equation}\label{Definitionforakj}
f_j^{(N,a)}(k)=\frac{W^{(a,-1/2)}(k)}{\pi}\int_{-1}^1 x^{N-j}E^{\omega}(x)\mathsf{J}_k^{(a,-1/2)}(x)(1-x)^{a}(1+x)^{-1/2}dx
\end{equation}
and

\[
C_{N,a}=
\begin{cases}
2^{(N-1)N/2},&\ \text{if}\ \ a=1/2,\\
2^{(N-2)(N-1)/2},&\ \text{if}\ \ a=-1/2.
\end{cases}
\]
\end{theorem}
Theorem \ref{theorem1} follows from \eqref{ChiOmega} and the following statement with
$E=E^{\omega}$. It is an orthogonal group analog of Lemma 6.5 of
\cite{kn:O}.

\begin{lemma}\label{E}
Let $E(x)\in C^1[-1,1]$. Fix complex numbers $z_1,\ldots,z_N$ on the unit circle and set $\zeta_k=(z_k+z_k^{-1})/2$. Then
\begin{multline*}
E(\zeta_1)\ldots E(\zeta_N)=\displaystyle 2\cdot 4\cdot\ldots\cdot 2^{N-1}\\
\times\sum_{\lambda\in\GT_N}\det\left[f_j^{(N,1/2)}(\lambda_i-i+N)\right]_{1\leq
i,j\leq N} \cdot \chi^{\lambda}_{SO(2N+1)}(z_1,\ldots,z_N)
\end{multline*}
and
\begin{multline*}
E(\zeta_1)\ldots E(\zeta_N)=\displaystyle 2\cdot 4\cdot\ldots\cdot 2^{N-1}\cdot 2^{-N}\\
\times
\sum_{\lambda\in\GT_N}\det\left[f_j^{(N,-1/2)}(\lambda_i-i+N)\right]_{1\leq
i,j\leq N} \cdot
(\chi^{\lambda}_{SO(2N)}+\chi^{\lambda^*}_{SO(2N)})(z_1,\ldots,z_N),
\end{multline*}
where $f_j^{(N,a)}(k)$ is defined by \eqref{Definitionforakj}, with $E(x)$ in place of $E^{\omega}(x)$.
\end{lemma}
\begin{proof} Let $E_i(x)=x^{N-i}E(x)$ for $1\leq i\leq N$. Using equation (\ref{hk}) and Lemma \ref{DeltaDistribution} shows that for $1\leq i\leq N$,
\begin{eqnarray*}
E_i(x) &=&\displaystyle\sum_{k=0}^{\infty} \int_{-1}^1 t^{N-j}E(t)\frac{J_k^{(a,-1/2)}(t)J_k^{(a,-1/2)}(x)}{h_k^{(a,-1/2)}}(1-t)^{a}(1+t)^{-1/2}dt\\
&=&\displaystyle\sum_{k=0}^{\infty} \frac{W^{(a,-1/2)}(k)}{\pi}\int_{-1}^1 t^{N-j}E(t)\mathsf{J}_k^{(a,-1/2)}(t)\mathsf{J}_k^{(a,-1/2)}(x)(1-t)^{a}(1+t)^{-1/2}dt\\
&=&\displaystyle\sum_{k=0}^{\infty} f_i^{(N,a)}(k)\mathsf{J}_k^{(a,-1/2)}(x).
\end{eqnarray*}
By Lemma \ref{Hua},
\begin{multline}\label{eqp}
\det\left[\zeta_j^{N-i}E(\zeta_j)\right]_{1\leq i,j\leq N}=\\
\displaystyle\sum_{\lambda_1\geq\ldots\geq\lambda_N\geq
0}\det\left[f_j^{(N,a)}(\lambda_i-i+N)\right]_{1\leq i,j\leq N}
\cdot
\det\left[\mathsf{J}_{\lambda_i-i+N}^{(a,-1/2)}(\zeta_j)\right]_{1\leq
i,j\leq N},
\end{multline}
where $\lambda_i=k_i+i-N$.
Equation \eqref{eqp} can be rewritten as
\begin{multline}\label{eqp2}
E(\zeta_1)\ldots E(\zeta_N) =\\
\displaystyle\sum_{\lambda\in\GT_N}\det\left[f_j^{(N,a)}(\lambda_i-i+N)\right]_{1\leq
i,j\leq N} \cdot\frac{
\det\left[\mathsf{J}_{\lambda_i-i+N}^{(a,-1/2)}(\zeta_j)\right]_{1\leq
i,j\leq N}}{\det\left[\zeta_j^{N-i}\right]_{1\leq i,j\leq N}}.
\end{multline}
Set $\zeta_k=(z_k+z_k^{-1})/2$. First consider the case when $a=1/2$. Using \eqref{chiformula}, equation \eqref{eqp2} becomes
\begin{multline*}
E(\zeta_1)\ldots E(\zeta_N)=\displaystyle 2\cdot 4\cdot\ldots\cdot 2^{N-1}\\
\times\sum_{\lambda\in\GT_N}\det\left[f_j^{(N,1/2)}(\lambda_i-i+N)\right]_{1\leq
i,j\leq N} \cdot \chi^{\lambda}_{SO(2N+1)}(z_1,\ldots,z_N).
\end{multline*}
Now consider the case when $a=-1/2$. Using \eqref{chiformula2}, equation \eqref{eqp2} becomes
\begin{multline*}
E(\zeta_1)\ldots E(\zeta_N)=\displaystyle 2\cdot 4\cdot\ldots\cdot 2^{N-1}\cdot 2^{-N}\\
\times
\sum_{\lambda\in\GT_N}\det\left[f_j^{(N,-1/2)}(\lambda_i-i+N)\right]_{1\leq
i,j\leq N} \cdot
(\chi^{\lambda}_{SO(2N)}+\chi^{\lambda^*}_{SO(2N)})(z_1,\ldots,z_N).
\end{multline*}
\end{proof}

\section{Stochastic Dynamics}\label{SD}
\subsection{Markov Chain on One Level}\label{MCoOL}
Let $E(x)\in C^1[-1,1]$ such that $E(1)\neq 0$. Plugging $z_1=\ldots=z_N=1$ into Lemma \ref{E} shows that $E(x)$ defines a (possibly signed) normalized (i.e. all the weights add up to one) measure $P_{N,a}$ on $\GT_{N,a}$ by
\[
P_{N,a}(\lambda)=\mathrm{const}\cdot\det\left[f_j^{(N,a)}(\lambda_i-i+N)\right]_{1\leq
i,j\leq N}\dim_{SO(2N+1/2+a)}\lambda
\]
where  $f_j^{(N,a)}$ is defined by \eqref{Definitionforakj}, with $E$ instead of $E^{\omega}$.

Fix $\varphi\in C^1[-1,1]$. Define
\[
\tilde{E}(x)=E(x)\varphi(x)
\]
and define $\tilde{f}_j^{(N,a)}$ and $\tilde{P}_{N,a}$ as in Theorem \ref{theorem1}, except with $\tilde{E}$ instead of $E^{\omega}$.

For $a=\pm 1/2$, let $I_a^{\varphi}$ be defined on $\Z_{\geq 0}\times\Z_{\geq 0}$ by
\[
W^{(a,-1/2)}(i)\mathsf{J}_i^{(a,-1/2)}(x)\varphi(x) = \displaystyle\sum_{k=0}^{\infty}W^{(a,-1/2)}(k)\mathsf{J}_k^{(a,-1/2)}(x)I_a^{\varphi}(k,i),\ \ i\geq 0.
\]
Multiply both sides by $\mathsf{J}_l^{(a,-1/2)}(x)$ and integrate over $[-1,1]$. Then the above definition is equivalent to
\[
I_a^{\varphi}(l,i) = \frac{W^{(a,-1/2)}(i)}{\pi}\displaystyle\int_{-1}^1\mathsf{J}_i^{(a,-1/2)}(x)\mathsf{J}_l^{(a,-1/2)}(x)\varphi(x)(1-x)^a(1+x)^{-1/2}dx.
\]

\begin{proposition}\label{FormulaTheorem}
With the above notation,
\[
\frac{\tilde{P}_{N,a}(\lambda)}{\dim_{SO(2N+1/2+a)}\lambda}=\displaystyle\sum_{\mu\in\GT_{N,a}}\det[I_a^{\varphi}(\mu_i-i+N,\lambda_j-j+N)]_{1\leq i,j\leq N}\frac{P_{N,a}(\mu)}{\dim_{SO(2N+1/2+a)}\mu}.
\]
\end{proposition}
\begin{proof}
Let $A$, $B$, $C$ be the following matrices:
\begin{align*}
C(i,j)=& \tilde{f}_j^{(N,a)}(i),\ \ & 1\leq j\leq N,\ 1\leq i<\infty\\
A(i,j)=& I_a^{\varphi}(j,i),\ \ & 1\leq i,j<\infty\\
B(i,j)=& f_j^{(N,a)}(i),\ \ & 1\leq j\leq N,\ 1\leq i<\infty
\end{align*}
Then $C=AB$, so by the Cauchy-Binet formula,
\begin{multline*}
\det[\tilde{f}_j^{(N,a)}(\lambda_i-i+N)]_{1\leq i,j\leq N}=\\
\displaystyle\sum_{\mu\in\GT_N}\det[I_a^{\varphi}(\mu_i-i+N,\lambda_j-j+N)]_{1\leq i,j\leq N}\det[f_j^{(N,a)}(\mu_i-i+N)]_{1\leq i,j\leq N},
\end{multline*}
which implies the proposition.
\end{proof}

For $a=\pm 1/2$, define the matrix $T_{N,a}^{\varphi}$ on $\GT_N\times\GT_N$ by
\[
T_{N,a}^{\varphi}(\mu,\lambda)=\displaystyle\det[I_a^{\varphi}(\mu_i-i+N,\lambda_j-j+N)]_{1\leq i,j\leq N}\frac{\dim_{SO(2N+1/2+a)}\lambda}{\dim_{SO(2N+1/2+a)}\mu}.
\]
Proposition \ref{FormulaTheorem} suggests a Markov Chain with state space $\GT_{N,a}$ with transition probabilities given by $T_{N,a}^{\varphi}$.

\begin{proposition}\label{RowsSumto1}
For any $\varphi\in C^1[-1,1]$, the rows of $T_{N,a}^{\varphi}$ sum to $\varphi(1)^N$. In particular, if $\varphi(1)=1$, then the rows of $T_{N,a}^{\varphi}$ sum to $1$.
\end{proposition}
\begin{proof}
Let $z_1,\ldots,z_N$ be complex numbers on the unit circle and set $\zeta_i=(z_i+z_i^{-1})/2$. Using the notation and statement of Lemma \ref{BorodinFerrari},
\begin{multline*}
\displaystyle\sum_{\lambda\in\GT_{N,a}}
\det\left[e^{(a,-1/2)}_{\lambda_j-j+N}(\zeta_i)\right]_{1\leq i,j\leq N}
\det[g(\mu_i-i+N,\lambda_j-j+N)]_{1\leq
i,j\leq N}\\
=\det\left[e_{\mu_j-j+N}^{(a,-1/2)}(\zeta_i)\right]_{1\leq i,j\leq
N}\varphi(\zeta_1)\ldots \varphi(\zeta_N).
\end{multline*}
By \eqref{hk}, for any $k,l\geq 0$,
\[
e_k^{(a,-1/2)}=\mathsf{J}_k^{(a,-1/2)}\sqrt{\frac{W(k)}{\pi}},\ \
g(k,l)=\sqrt{\frac{W(k)}{W(l)}}I_a^{\varphi}(k,l).
\]
Therefore
\begin{multline*}
\displaystyle\sum_{\lambda\in\GT_{N,a}}
\det\left[\mathsf{J}^{(a,-1/2)}_{\lambda_j-j+N}(\zeta_i)\right]_{1\leq
i,j\leq N}
\det\left[I_a^{\varphi}(\mu_i-i+N,\lambda_j-j+N)\right]_{1\leq
i,j\leq N}\\
=\det\left[\mathsf{J}^{(a,-1/2)}_{\mu_j-j+N}(\zeta_i)\right]_{1\leq
i,j\leq N}\varphi(\zeta_1)\ldots \varphi(\zeta_N),
\end{multline*}
or equivalently, by \eqref{chiformula} and \eqref{chiformula2},
\begin{multline*}
\displaystyle\sum_{\lambda\in\GT_{N,a}}
\chi^{\lambda}_{SO(2N+1/2+a)}(z_1,\ldots,z_N)
\det[I_a^{\varphi}(\mu_i-i+N,\lambda_j-j+N)]_{1\leq
i,j\leq N}\\
=\chi^{\mu}_{SO(2N+1/2+a)}(z_1,\ldots,z_N)\varphi(\zeta_1)\ldots
\varphi(\zeta_N).
\end{multline*}
Taking $z_1,\ldots,z_N=1$ shows that the rows of $T_{N,a}^{\varphi}$ sum to $\varphi(1)^N$.
\end{proof}

\begin{proposition}\label{NonnegativeEntries} If $\varphi(x)=p_0+p_1x$ with $p_0> p_1 \geq 0$, then each entry of $T_{N,a}^{\varphi}$ is nonnegative. The same holds if $\varphi(x)=e^{t(x-1)}$, where $t\geq 0$. Additionally, the diagonal entries of $T_{r,\pm 1/2}^{p_0+p_1x}$ are bounded below by
\begin{equation}\label{SmallestDet}
\left(\frac{R_+^r(R_+-p_1/2)-R_-^r(R_--p_1/2)}{\sqrt{p_0^2-p_1^2}}\right),
\end{equation}
where
\[
R_{\pm}=\frac{p_0\pm\sqrt{p_0^2-p_1^2}}{2}.
\]
\end{proposition}
\begin{proof}
First consider the situation when $\varphi(x)=p_0+p_1x$. Note that by \eqref{JacobiIdentities}--\eqref{JacobiIdentities2},  $I_{a}^{\varphi}(k,l)=0$ if $\vert k-l\vert>1$.

If $\mu_i<\lambda_i-1$ for some $i$ then $\mu_k<\lambda_l-1$ for $k\geq i$ and $l\leq i$, which implies that $I_a^{\varphi}(\mu_k,\lambda_l)=0$ for such $k,l$, and thus the determinant in question is $0$. If $\mu_i>\lambda_i+1$ then $\mu_k>\lambda_l+1$ for $k\leq i$ and $l\geq i$, which means $I_a(\mu_k,\lambda_l)=0$, and the determinant is $0$ again. So $\det[I_a^{\varphi}]$ is zero if $\vert\lambda_i-\mu_i\vert>1$ for some $i$. Hence, it remains to consider the case when $\vert\lambda_i-\mu_i\vert\leq 1$ for all $1\leq i\leq N$.

Split $\{\mu_i-i+N\}_{i=1}^N$ into blocks of neighbouring integers wth distance between blocks being at least $2$. Then it is easy to see that $\det[I_a(\mu_i-i+N,\lambda_j-j+N)]$ splits into the product of determinants corresponding to blocks. It suffices to show that the determinant corresponding to each block is nonnegative, so assume without loss of generality that $\{\mu_i-i+N\}_{i=1}^N$ is one such block. In other words, assume that all $\mu_i$ are equal. Then there exist $m$ and $n$, $1\leq m\leq n\leq N$ such that $\mu_i=\lambda_i-1$ for $1\leq i<m$, and $\mu_i=\lambda_i$ for $m\leq i<n$, and $\mu_i=\lambda_i+1$ for $n\leq i\leq N$. The determinant is the product of determinants of three matrices. We shall examine each of these matrices.

Because $I_a^{\varphi}(k,k\pm 1)\geq 0$, the matrix parametrized by $1\leq i,j<m$ is triangular with nonnegative diagonal entries, so has nonnegative determinant. Similarly, the matrix parametrized by $n\leq i,j\leq N$ also is triangular with nonnegative diagonal entries. It remains to consider the matrix parametrized by $m\leq i,j<n$.

For now, assume $\lambda_{n-1}\neq 0.$ Then this matrix is tridiagonal, with $p_0$ in the diagonal entries and $p_1/2$ in the subdiagonal and superdiagonal entries. If this matrix has size $r\times r$, let $D_r$ denote its determinant. Then $D_r$ satisfies the recurrence relation
\[
D_r = p_0D_{r-1} - \frac{p_1^2}{4}D_{r-2}, \ \ D_1 = p_0, \ \ D_2 = p_0^2-\frac{p_1^2}{4}.
\]
Solving this explicitly yields
\[
D_r = \frac{1}{\sqrt{p_0^2-p_1^2}}\left(\left(\frac{p_0+\sqrt{p_0^2-p_1^2}}{2}\right)^{r+1} - \left(\frac{p_0-\sqrt{p_0^2-p_1^2}}{2}\right)^{r+1}\right),
\]
which shows that $D_r$ is nonnegative.

If $\lambda_{n-1}=0$ and $a=1/2$, then the entry in the $r$th row and $r$th column is $p_0-p_1/2$ instead of $p_0$. The other entries are $p_0,p_1/2$, and $0$, as before. In this case, the determinant is
\begin{eqnarray*}
\left(p_0-\frac{p_1}{2}\right)D_{r-1} - \frac{p_1^2}{4} D_{r-2} &=& D_r-\frac{p_1}{2}D_{r-1}\\
&=& A_r\left(\frac{p_0+\sqrt{p_0^2-p_1^2}}{2}\right)-A_r\left(\frac{p_0-\sqrt{p_0^2-p_1^2}}{2}\right),\\
&& \text{where}\ A_r(x) = \frac{1}{\sqrt{p_0^2-p_1^2}}\left(x^{r+1}-\frac{p_1}{2}x^r\right).
\end{eqnarray*}
Note that $A_r$ is positive on the interval $(p_1/2,\infty)$, which contains $(p_0+\sqrt{p_0^2-p_1^2})/2$, and is negative on the interval $(0,p_1/2)$, which contains $(p_0-\sqrt{p_0^2-p_1^2})/2$. Therefore the above expression is nonnegative.

If $\lambda_{n-1}=0$ and $a=-1/2$, then the only modified entry is in the $r$th row and $(r-1)$st column. It equals $p_1$ instead of $p_1/2$. In this case, the determinant is
\[
p_0D_{r-1}-\frac{p_1^2}{2}D_{r-2} = D_r - \frac{p_1^2}{4}D_{r-2}.
\]
We have already shown that $D_r\geq (p_1/2)D_{r-1}$, so therefore $D_r\geq(p_1^2/4)D_{r-2}$. So the above expression is also nonnegative.

For the last claim in the proposition, note that $D_r\geq D_r-(p_1/2)^2D_{r-2} \geq D_r-(p_1/2)D_{r-1}=\eqref{SmallestDet}$.

Now let $\varphi(x)=e^{t(x-1)}$. For $n>2t$, let $\varphi_n(x)=\left(1+\frac{t(x-1)}{n}\right)^n$. By Lemma \ref{GroupHomomorphism} below, $T_{N,a}^{\varphi_n}=T_{N,a}^{1+t(x-1)/n}\ldots T_{N,a}^{1+t(x-1)/n}$. We just showed that all entries of each $T_{N,a}^{1+t(x-1)/n}$ are nonnegative. Therefore all entries of $T_{N,a}^{\varphi_n}$ are nonnegative.
By Lemma \ref{UniformConvergence}, $T_{N,a}^{\varphi}(\mu,\lambda)$ equals $\lim_{n\rightarrow\infty} T_{N,a}^{\varphi_n}(\mu,\lambda)$, so all entries of $T_{N,a}^{\varphi}$ are nonnegative.
\end{proof}

\begin{lemma}\label{GroupHomomorphism}
If $\varphi_1,\varphi_2\in C^1[-1,1]$, then $T_{N,a}^{\varphi_1\varphi_2}=T_{N,a}^{\varphi_1}T_{N,a}^{\varphi_2}$.
\end{lemma}
\begin{proof}
This is a straightforward computation using the Lemmas \ref{Hua} and \ref{DeltaDistribution}.
\end{proof}

\begin{lemma}\label{UniformConvergence}
Suppose $(\varphi_n)_{n=1}^{\infty},\varphi\in C^1[-1,1]$ and $\varphi_n$ converges to $\varphi$ uniformly on $[-1,1]$. Then for any $\lambda,\mu\in\GT_N$, $T_{N,a}^{\varphi_n}(\mu,\lambda)$ converges to $T_{N,a}^{\varphi}(\mu,\lambda)$.
\end{lemma}
\begin{proof}
Since $\varphi_n$ converges uniformly to $\varphi$ and $\varphi$ is bounded, the dominated convergence theorem implies that each $I_a^{\varphi_n}(i,j)$ converges to $I_a^{\varphi}(i,j)$. Since $T_{N,a}^{\varphi_n}(\mu,\lambda)$ is continuous in the variables $I_a^{\varphi_n}(\mu_i-i+N,\lambda_j-j+N)$, it must converge to $T_{N,a}^{\varphi}(\mu,\lambda)$.
\end{proof}

\subsection{Generalities on Multivariate Markov Chains}\label{GoMMC}
Recall that in section \ref{CM}, we explained that the measures $P^{\omega}_{N,a}$ generalize to $P^{\omega}$. We just constructed a Markov Chain that maps $P^{\omega}_{N,a}$ to $P^{\tilde{\omega}}_{N,a}$, so it is natural to expect a Markov Chain that maps $P^{\omega}$ to $P^{\tilde{\omega}}$. Our goal now is to extend $T_{N,\pm}^{\varphi}$ to stochastic matrices on finite paths that map $P^{\omega}$ to $P^{\tilde{\omega}}$ (or rather their projections on finite paths). This section describes a general construction from \cite{kn:BF2} which builds a Markov chain from smaller ones. In the next section, this construction will be applied to our case. The original idea for bivariate Markov chains goes back to \cite{kn:DF}.

Let $\mathcal{S}_1,\ldots,\mathcal{S}_n$ be discrete sets. For $1\leq k\leq n$, let $T_k$ be a stochastic matrix with rows and columns indexed by $\mathcal{S}_k$. For $2\leq k\leq n$, let $\Lambda^{k}_{k-1}$ be a stochastic matrix with rows indexed by $\mathcal{S}_{k}$ and columns indexed by $\mathcal{S}_{k-1}$. Assume these matrices commute:
\[
\Delta_{k-1}^{k}:=\Lambda^{k}_{k-1}T_{k-1}=T_k\Lambda^{k}_{k-1}.
\]
The state space for the multivariate Markov Chain is
\[
\mathcal{S}^{(n)}_{\Lambda}=\{(x_1,\ldots,x_n)\in\mathcal{S}_1\times\ldots\times\mathcal{S}_n:\displaystyle\prod_{k=2}^{\infty}\Lambda_{k-1}^k(x_k,x_{k-1})\neq 0\}.
\]
Write $X_n=(x_1,\ldots,x_n),Y_n=(y_1,\ldots,y_n)\in\mathcal{S}_{\Lambda}^{(n)}$. The probability of a transition from $X_n$ to $Y_n$ is
\[
\begin{cases}
T_1(x_1,y_1)\displaystyle\prod_{k=2}^{n}\frac{T_k(x_k,y_k)\Lambda_{k-1}^k(y_k,y_{k-1})}{\Delta_{k-1}^k(x_k,y_{k-1})},\ \ \text{if}\ \prod_{k=2}^{n}\Delta_{k-1}^k(x_k,y_{k-1})>0\\
0,\ \  \text{otherwise}.
\end{cases}
\]

Let $\mathbf{T}$ denote this matrix of transition probabilities. One could think of $\mathbf{T}$ as follows.

Starting from $X=(x_1,\ldots,x_n)$, first choose $y_1$ according to
the transition matrix $T_1(x_1,y_1)$, then choose $y_2$ using $\frac{T_2(x_2,y_2)\Lambda^2_1(y_2,y_1)}{\Delta^2_1(x_2,y_1)}$ , which is
the conditional distribution of the middle point in the successive application
of $T_2$ and $\Lambda^2_1$ provided that we start at $x_2$ and finish at $y_1$. Then choose $y_3$ using the conditional distribution of the middle point in the successive
application of $T_3$ and $\Lambda^3_2$ provided that we start at $x_3$ and finish at $y_2$, and
so on. Thus, one could say that $Y$ is obtained by the \textit{sequential update} \cite{kn:BF}.

The next proposition will be used later.
\begin{proposition}\label{BFProposition}
Let $m_n(x_n)$ be a probability measure on $\mathcal{S}_n$. Consider the evolution of the measure $m_n(x_n)\Lambda_{n-1}^n(x_n,x_{n-1})\ldots\Lambda^2_1(x_2,x_1)$ on $\mathcal{S}_{\Lambda}^{(n)}$ under the Markov chain $\mathbf{T}$, and denote by $(x_1(j),\ldots,x_n(j))$ the result after $j=0,1,2,\ldots$ steps. Then for any $k_1\geq k_2\geq\ldots\geq k_n\geq 0$, the joint distribution of
\begin{multline*}
(x_n(1),\ldots,x_n(k_n),x_{n-1}(k_n),x_{n-1}(k_n+1)\ldots,x_{n-1}(k_{n-1}),\\
x_{n-2}(k_{n-1}),\ldots,x_2(k_2),x_1(k_2),\ldots,x_1(k_1))
\end{multline*}
coincides with the stochastic evolution of $m_n$ under transition matrices
\[
(\underbrace{T_n,\ldots,T_n}_{k_n},\Lambda_{n-1}^n,\underbrace{T_{n-1},\ldots,T_{n-1}}_{k_{n-1}-k_n},\Lambda_{n-2}^{n-1},\ldots,\Lambda_1^2,\underbrace{T_1,\ldots,T_1}_{k_1-k_2})
\]
\end{proposition}
\begin{proof}
See Proposition 2.5 of \cite{kn:BF}.
\end{proof}

Let $L$ be the linear subspace of $l^1(\mathcal{S}^{(n)}_{\Lambda})$
spanned by elements of the form
$m_n(x_n)\Lambda^n_{n-1}(x_n,x_{n-1})\ldots\Lambda^2_1(x_2,x_1)$,
where $m_n$ is a summable function on $\mathcal{S}_n$. Then
$\mathbf{T}$ can be thought of as a bounded linear operator of $L$.
Similarly, $T_n$ is a bounded linear operator on
$l^1(\mathcal{S}_n)$.

\begin{lemma}\label{NormsAreEqual}
With the notation from above,
\[
\Vert \mathbf{T} - Id \Vert_L = \Vert T_n - Id\Vert_{l^1(\mathcal{S}_n)}.
\]
\end{lemma}
\begin{proof}
By definition,
\begin{eqnarray*}
\Vert \mathbf{T} - Id \Vert_L &=& \displaystyle\sup_{f\in L} \frac{\Vert \mathbf{T}f - f\Vert_L}{\Vert f\Vert_L},\\
\Vert T_n - Id \Vert_{l^1(\mathcal{S}_n)} &=& \displaystyle\sup_{m_n\in l^1(\mathcal{S}_n)} \frac{\Vert T_n m_n - m_n\Vert_{l^1(\mathcal{S}_n)}}{\Vert m_n\Vert_{l^1(\mathcal{S}_n)}}.
\end{eqnarray*}

If $f\in L$, then $f$ must have the form
$m_n(x_n)\Lambda^n_{n-1}(x_n,x_{n-1})\ldots\Lambda^2_1(x_2,x_1)$ for
some $m_n\in l^1(\mathcal{S}_n)$. Since the matrices
$\Lambda_k^{k+1}$ are all stochastic, $\Vert f\Vert_L = \Vert
m_n\Vert_{l^1(\mathcal{S}_n)}$. By Proposition \ref{BFProposition},
$\mathbf{T}f=(T_nm_n)(x_n)\Lambda^n_{n-1}(x_n,x_{n-1})\ldots\Lambda^2_1(x_2,x_1)$,
which implies $\Vert \mathbf{T}f - f\Vert_L = \Vert T_n m_n -
m_n\Vert_{l^1(\mathcal{S}_n)}$. Thus, the lemma holds.
\end{proof}

\subsection{Markov Chain on Multiple Levels}\label{MCoML}
We need an implementation of the stochastic matrices $\Lambda_{k-1}^k$. Define
the matrix $T^{N,+}_{N,-}$ on $\GT_{N,+}\times\GT_{N,-}$ by
\[
T^{N,+}_{N,-}(\mu,\lambda)=\frac{\dim_{SO(2N)}\lambda}{\dim_{SO(2N+1)}\mu}
\det[\T(\mu_i-i+N,\lambda_j-j+N)]_{i,j=1}^N
\]
where
\[
\T(x,y)=
\begin{cases}
1,\ \ x\geq y=0,\\
2,\ \ x\geq y>0,\\
0,\ \ x<y.
\end{cases}.
\]
Also define the matrix $T^{N,-}_{N-1,+}$ on $\GT_{N,-}\times\GT_{N-1,+}$ by
\[
T^{N,-}_{N-1,+}(\mu,\lambda)=\frac{\dim_{SO(2N-1)}\lambda}
{\dim_{SO(2N)}\mu}\det[\phi(\mu_i-i+N,\lambda_j-j+N-1)]_{i,j=1}^N
\]
where
\[
\phi(x,y)=
\begin{cases}
1,\ \ x>y,\\
0,\ \ x\leq y,
\end{cases}
\]

Recall the definition of $\varkappa$ in Section \ref{CM}.

\begin{lemma}\label{Interlacing} For any $(\lambda,\mu)\in \GT_{N-1,+}\times\GT_{N,-}$,
\[
\det[\phi(\mu_i-i+N,\lambda_j-j+N-1)]_{i,j=1}^N=\varkappa(\lambda,\mu).
\]
For any $(\lambda,\mu)\in \GT_{N,-}\times\GT_{N,+}$,
\[
\det[\T(\mu_i-i+N,\lambda_j-j+N)]_{i,j=1}^N=\varkappa(\lambda,\mu).
\]
\end{lemma}
\begin{proof}
The argument is standard, see e.g. Proposition 3.4 of \cite{kn:BK}.
The proof of the second formula is exactly the same.
\end{proof}

\begin{proposition}\label{Stochastic} The matrices $T^{N,+}_{N,-}$ and $T_{N-1,+}^{N,-}$ are stochastic.
\end{proposition}
\begin{proof} First let us show that $T^{N,+}_{N,-}$ is stochastic. Taking dimensions of both sides of \eqref{Branching1} yields
\[
\dim_{SO(2N+1)}\mu=\displaystyle\sum_{\lambda\prec\mu}\varkappa(\lambda,\mu)\dim_{SO(2N)}\lambda.
\]
Therefore
\[
\displaystyle\sum_{\lambda\in\GT_{N,-}}T^{N,+}_{N,-}(\mu,\lambda)=1.
\]
By Lemma \ref{Interlacing}, $T^{N,+}_{N,-}$ has nonnegative entries, so it is stochastic.

Now we will show that $T_{N-1,+}^{N,-}$ is stochastic. Taking dimensions in \eqref{Branching2} yields
\[
\dim_{SO(2N)}\mu=\displaystyle\sum_{\lambda\prec\mu}\dim_{SO(2N-1)}\lambda.
\]
Therefore
\[
\displaystyle\sum_{\lambda\in\GT_{N-1,1/2}}T^{N,-}_{N-1,+}(\mu,\lambda)=1.
\]
The nonnegativity also follows from Lemma \ref{Interlacing}.
\end{proof}

The matrix elements of $T_{N,-}^{N,+}$ and $T_{N-1,+}^{N,-}$ are
\textit{cotransition} probabilities of the branching graph $\GT$,
see e.g. \cite{kn:K1}. Recall that in Section \ref{MCoOL} we defined
stochastic matrices $T_{N,\pm 1/2}$. We use the notation $T_{N,\pm}$
for convenience.
\begin{proposition}\label{CommutationRelations} Assume $\varphi(1)=1$. For any $N\geq 1$, we have the following commutation relations:
\[
T_{N,+}^{\varphi}T^{N,+}_{N,-}=T^{N,+}_{N,-}T_{N,-}^{\varphi}, \qquad T_{N,-}^{\varphi}T^{N,-}_{N-1,+}=T^{N,-}_{N-1,+}T_{N-1,+}^{\varphi}.
\]
\end{proposition}
\begin{proof}
We start by proving the first relation. By Lemma \ref{Hua},
\begin{multline*}
T_{N,+}^{\varphi}T^{N,+}_{N,-}(\mu,\lambda)=\frac{\dim_{SO(2N)}\lambda}{\dim_{SO(2N+1)}\mu}\\
\times\sum_{\nu\in\GT_{N,+}}\det\left[I_{1/2}^{\varphi}(\mu_i-i+N,\nu_j-j+N)\right]_{1\leq i,j\leq N} \det[\T(\nu_i-i+N,\lambda_j-j+N)]_{1\leq i,j\leq N}\\
=\frac{\dim_{SO(2N)}\lambda}{\dim_{SO(2N+1)}\mu}\det\Biggl[\displaystyle\sum_{x=0}^{\infty}I_{1/2}^{\varphi}(\mu_i-i+N,x)\T(x,\lambda_j-j+N)\Biggr]_{1\leq
i,j\leq N}.
\end{multline*}
Similarly,
\begin{multline*}
T^{N,+}_{N,-}T_{N,-}^{\varphi}(\mu,\lambda)\\
=\frac{\dim_{SO(2N)}\lambda}{\dim_{SO(2N+1)}\mu}\det\left[\displaystyle\sum_{x=0}^{\infty}\T(\mu_i-i+N,x) I_{-1/2}^{\varphi}(x,\lambda_j-j+N)\right]_{1\leq i,j\leq N}.
\end{multline*}
We thus need to check that
\begin{multline*}
\displaystyle\sum_{x=0}^{\infty}\T(\mu_i-i+N,x) I_{-1/2}^{\varphi}(x,\lambda_j-j+N)\\
=\displaystyle\sum_{x=0}^{\infty}I_{1/2}^{\varphi}(\mu_i-i+N,x)\T(x,\lambda_j-j+N).
\end{multline*}
By applying Lemma \ref{NeededSum2} to the right hand side and Lemma \ref{UsefulIdentities}(a) to the left hand side, one sees that both sides are equal to
\begin{equation*}
\frac{1}{\pi}\displaystyle\int_{-1}^1\mathsf{J}_{\mu_i-i+N}^{(1/2,-1/2)}(x)\mathsf{J}_{\lambda_j-j+N}^{(-1/2,-1/2)}(x)\varphi(x)(1-x)^{-1/2}(1+x)^{-1/2}dx.
\end{equation*}

Now we prove the second relation. Expanding $\det[\phi]_1^N$ along the $N$th column and using Lemma \ref{Hua}, we obtain
\begin{align}\label{DetFormula1}
&T^{N,-}_{N-1,+}T_{N-1,+}^{\varphi}(\mu,\lambda)\\
&=\frac{\dim_{SO(2N-1)}\lambda}{\dim_{SO(2N)}\mu}\displaystyle\sum_{\nu\in\GT_{N-1,1/2}}\sum_{k=1}^N(-1)^{N-k}\det\left[\phi(\mu_i-i+N,\nu_j-j+N-1)\right]_{\begin{subarray}{c} 1\leq i\neq k\leq N \\ 1\leq j\leq N-1 \end{subarray}}\\
&\qquad \ \ \times\det[I_{1/2}^{\varphi}(\nu_i-i+N-1,\lambda_j-j+N-1)]_{1\leq i,j\leq N-1}\\
&=\frac{\dim_{SO(2N-1)}\lambda}{\dim_{SO(2N)}\mu}\displaystyle\sum_{k=1}^N(-1)^{N-k}\det\left[\sum_{x=0}^{\infty}\phi(\mu_i-i+N,x)I_{1/2}^{\varphi}(x,\lambda_j-j+N-1)\right]_{\begin{subarray}{c} 1\leq i\neq k\leq N \\ 1\leq j\leq N-1 \end{subarray}}\\
&=\frac{\dim_{SO(2N-1)}\lambda}{\dim_{SO(2N)}\mu}\det\left[\displaystyle\sum_{r=0}^{\infty}\phi(\mu_i-i+N,r) I_{1/2}^{\varphi}(r,\lambda_j-j+N-1)\right]_{1\leq i,j\leq N}\label{DetFormula1.5},
\end{align}
where it is agreed that all matrix elements in the $N$th column ($j=N$) of \ref{DetFormula1.5} are equal to $1$.
Similarly,
\begin{multline}\label{DetFormula2}
T_{N,-}^{\varphi}T^{N,-}_{N-1,+}(\mu,\lambda)\\
=\frac{\dim_{SO(2N-1)}\lambda}{\dim_{SO(2N)}\mu}\det\left[\displaystyle\sum_{r=0}^{\infty}I_{-1/2}^{\varphi}(\mu_i-i+N,r) \phi(r,\lambda_j-j+N-1)\right]_{1\leq i,j\leq N}.
\end{multline}
By Lemma \ref{NeededSum}(a), the $N$th column of \eqref{DetFormula2} equals $\varphi(1)\mathsf{J}_{\mu_i-i+N}^{(-1/2,-1/2)}(1)=1$. Therefore the $N$th columns of \eqref{DetFormula1.5} and \eqref{DetFormula2} are equal.

By Lemma \ref{UsefulIdentities}(b), for $j\neq N$, the $(i,j)$-entry of \eqref{DetFormula1.5} equals
\begin{equation}\label{jthColumn}
\frac{1}{\pi}\displaystyle\int_{-1}^1\frac{1-\mathsf{J}_{\mu_i-i+N}^{(-1/2,-1/2)}(x)}{1-x}\mathsf{J}_{\lambda_j-j+N-1}^{(1/2,-1/2)}(x)\varphi(x)(1-x)^{1/2}(1+x)^{-1/2}dx,
\end{equation}
By Lemma \ref{NeededSum}(b), for $j\leq N$, the $(i,j)$-entry of \eqref{DetFormula2} equals
\begin{equation}\label{jthColumn2}
\frac{1}{\pi}\displaystyle\int_{-1}^1
\Bigl(1-\mathsf{J}_{\mu_i-i+N}^{(-1/2,-1/2)}(x)\varphi(x)\Bigr)\mathsf{J}_{\lambda_j-j+N-1}^{(1/2,-1/2)}(x)(1-x)^{-1/2}(1+x)^{1/2}dx.
\end{equation}
Their difference only depends on $j$, so for $j\neq N$,
\[
j\text{th column of } \eqref{DetFormula2}=j\text{th column of } \eqref{DetFormula1.5} + [\eqref{jthColumn2}-\eqref{jthColumn}](N\text{th column}),
\]
so the matrix in \eqref{DetFormula2} is obtained from the matrix in \eqref{DetFormula1.5} by elementary column operations. This means that their determinants are equal.
\end{proof}

Following the notation of section \ref{CM}, let us denote the set of finite paths in $\GT$ of length $(2N-1/2+a)$ by
\[
\GT^{(N),a}:=\{\lambda^{(1),-1/2}\prec\lambda^{(1),1/2}\prec\ldots\prec\lambda^{(N),a}\vert\lambda^{(i)}\in\GT_i\}.
\]
Using the construction in section \ref{GoMMC}, the stochastic matrices $T_{N,\pm}^{\varphi},T_{N,-}^{N,+},T_{N-1,+}^{N,-}$  allow us to construct a Markov Chain on $\GT^{(N),a}$.
Define the matrices $\Delta_{N,-}^{N,+}$ and $\Delta_{N-1,+}^{N,-}$ by
\[
\Delta_{N,-}^{N,+}=T^{N,+}_{N,-}T_{N,-}^{\varphi}=T_{N,+}^{\varphi}T_{N,-}^{N,+}
\]
\[
\Delta_{N-1,+}^{N,-}=T^{N,-}_{N-1,+}T_{N-1,+}^{\varphi}=T_{N,-}^{\varphi}T_{N-1,+}^{N,-}.
\]

Define the matrix $A_{N,1/2}^{\varphi}$ with rows and columns indexed by elements of $\GT^{(N),1/2}$ by
\begin{multline*}
A_{N,1/2}^{\varphi}(\lambda^{(1),-1/2}\prec\ldots\prec\lambda^{(N),1/2},\mu^{(1),-1/2}\prec\ldots\prec\mu^{(N),1/2}):=\\
T_{1,-}^{\varphi}(\lambda^{(1),-1/2},\mu^{(1),-1/2})\frac{T_{1,+}^{\varphi}(\lambda^{(1),1/2},\mu^{(1),1/2})T_{1,-}^{1,+}(\mu^{(1),1/2},\mu^{(1),-1/2})}{\Delta_{1,-}^{1,+}(\lambda^{(1),1/2},\mu^{(1),-1/2})}\times\\
\ldots\times\frac{T_{N,+}^{\varphi}(\lambda^{(N),1/2},\mu^{(N),1/2})T_{N,-}^{N,+}(\mu^{(N),1/2},\mu^{(N),-1/2})}{\Delta_{N,-}^{N,+}(\lambda^{(N),1/2},\mu^{(N),-1/2})}\\
\text{if}\ \ \Delta_{k,-}^{k,+}(\lambda^{(k),1/2},\mu^{(k),-1/2}),\Delta_{k,+}^{k+1,-}(\lambda^{(k+1),-1/2},\mu^{(k),1/2})\neq 0,\ 1\leq k\leq N\\
0,\ \ \text{otherwise}.
\end{multline*}
Simiarly define $A_{N,-1/2}^{\varphi}$ with rows and columns indexed by elements of $\GT^{(N),-1/2}$ by
\begin{multline*}
A_{N,-1/2}^{\varphi}(\lambda^{(1),-1/2}\prec\ldots\prec\lambda^{(N),-1/2},\mu^{(1),-1/2}\prec\ldots\prec\mu^{(N),-1/2}):=\\
T_{1,-}^{\varphi}(\lambda^{(1),-1/2},\mu^{(1),-1/2})\frac{T_{1,+}^{\varphi}(\lambda^{(1),1/2},\mu^{(1),1/2})T_{1,-}^{1,+}(\mu^{(1),1/2},\mu^{(1),-1/2})}{\Delta_{1,-}^{1,+}(\lambda^{(1),1/2},\mu^{(1),-1/2})}\times\\
\ldots\times\frac{T_{N,-}^{\varphi}(\lambda^{(N),-1/2},\mu^{(N),-1/2})T_{N-1,+}^{N,-}(\mu^{(N),-1/2},\mu^{(N-1),1/2})}{\Delta_{N-1,+}^{N,-}(\lambda^{(N),-1/2},\mu^{(N-1),1/2})}\\
\text{if}\ \ \Delta_{k,-}^{k,+}(\lambda^{(k),1/2},\mu^{(k),-1/2}),\Delta_{k,+}^{k+1,-}(\lambda^{(k+1),-1/2},\mu^{(k),1/2})\neq 0,\ 1\leq k\leq N\\
0,\ \ \text{otherwise}.
\end{multline*}
By the construction in section \ref{GoMMC}, $A_{N,\pm 1/2}^{\varphi}$ are stochastic for $\varphi(x)=1-p_1+p_1x$, $0\leq p_1\leq 1/2$, and $\varphi(x)=e^{t(x-1)},t\geq 0$.

For any $E(x)\in C^1[-1,1]$ such that $E(1)\neq 0$, let $P^{(N),a}$ be the (possibly signed) measure
\[
P_{N,a}(\lambda^{(N),a})\ldots T^{2,-}_{1,+}(\lambda^{(2),-1/2},\lambda^{(1),1/2})T^{1,+}_{1,-}(\lambda^{(1),1/2},\lambda^{(1),-1/2})
\]
on $\GT^{(N),a}$, where $P_{N,a}$ is as in Theorem \ref{theorem1}. Proposition \ref{BFProposition} and Proposition \ref{FormulaTheorem} imply the following.
\begin{proposition}\label{GibbsEvolution}
Let $E(x)\in C^1[-1,1]$ such that $E(1)\neq 0$. Consider the measure $P^{(N),a}$ on $\GT^{(N),a}$. After one step of the Markov chain $A_{N,a}^{\varphi}$, the resulting measure on $\GT^{(N),a}$ is
\[
\widetilde{P}_{N,a}(\lambda^{(N),a})\ldots T^{2,-}_{1,+}(\lambda^{(2),-1/2},\lambda^{(1),1/2}) T^{1,+}_{1,-}(\lambda^{(1),1/2},\lambda^{(1),-1/2}),
\]
where $\widetilde{P}_{N,a}$ is defined from the function $\widetilde{E}(x)=\varphi(x)E(x)$.
\end{proposition}

\subsection{A Continuous-time Markov Chain on Multiple Levels}\label{AEftG}
Define a matrix $Q_{N,a}$ on $\GT^{(N),a}\times\GT^{(N),a}$ as follows. Let us explicitly write $Q_{N,a}(\lambda^{(1),-1/2}\prec\lambda^{(1),1/2}\prec\ldots\prec\lambda^{(N),a},\mu^{(1),-1/2}\prec\mu^{(1),1/2}\prec\ldots\prec\mu^{(N),a})$. There are three cases to consider:

Case 1. This occurs when there exist $(n_0,a_0)\trianglelefteq (n_1,a_1)$ and $k\leq n_0$ such that the numbers $\mu_k^{(n^*),a^*}-1,\lambda_k^{(n^*),a^*}$ are all equal for $(n_0,a_0)\trianglelefteq (n^*,a^*)\trianglelefteq (n_1,a_1)$. Furthermore, $\mu_l^{(\bar{n}),\bar{a}}=\lambda_l^{(\bar{n}),\bar{a}}$ for all other $\bar{n},\bar{a},l$.

There are two subcases:

Case 1a. When case 1 is satisfied and $a_0=-1/2$ and $\lambda_k^{(n_0),a_0}=0$.

Case 1b. When case 1 is satisfied and case 1a is not satisfied.

Case 2. This occurs when there exist $(n_0,a_0)\trianglelefteq (n_1,a_1)$ and $k\leq n_0$ such that the numbers $\lambda_{k+d(n^*,a^*;n_0,a_0)}^{(n^*),a^*},\mu_{k+d(n^*,a^*;n_0,a_0)}^{(n^*),a^*}+1$ are all equal for $(n_0,a_0)\trianglelefteq (n^*,a^*)\trianglelefteq (n_1,a_1)$. Recall that $d(n_1,a_1;n_0,a_0)=\vert 2n_1+a_1-2n_0-a_0 \vert$. Furthermore, $\mu_l^{(\bar{n}),\bar{a}}=\lambda_l^{(\bar{n}),\bar{a}}$ for all other $\bar{n},\bar{a},l$.

Case 3. This occurs when the two paths $\lambda$ and $\mu$ are not equal and neither case 1 nor case 2 is satisfied.

When case 1b or case 2 occurs, the corresponding element of $Q_{N,a}$ is $1/2$. When case 1a occurs, the corresponding element is $1$. When case 3 occurs, the corresponding element is $0$. The diagonal entries are defined so that the rows of $Q_{N,a}$ sum to $0$.

Under the map $\mathcal{L}_{\mathfrak{Y}}$, the cases can be described more easily. Let $\{y_k^m\}=\mathcal{L}_{\mathfrak{Y}}(\boldsymbol{\lambda})$ and $\{z_k^m\}=\mathcal{L}_{\mathfrak{Y}}(\boldsymbol{\mu})$.
Case 1 occurs when there exist $m_0\leq m_1$ and $k\leq [\frac{m_0}{2}]$ such that
\begin{multline*}
z_k^{m_0}-2=y_k^{m_0}=z_k^{m_0+1}-1-2=y_k^{m_0+1}-1=\ldots\\
=z_k^{m_1}-(m_1-m_0)-2=y_k^{m_1}-(m_1-m_0).
\end{multline*}
Furthermore, $z_l^{\bar{m}}=y_l^{\bar{m}}$ for all other $l,\bar{m}$.

Case 1a occurs when case 1 is satisfied and $y_k^{m_0}=0$.

Case 1b occurs when case 1 is satisfied and case 1a is not satisfied.

Case 2 occurs when there exist $m_0\leq m_1$ and $k\leq [\frac{m_0}{2}]$ such that
\begin{multline*}
z_k^{m_0}+2=y_k^{m_0}=z_{k+1}^{m_0+1}+1+2=y_{k+1}^{m_0+1}+1=\ldots\\
=z_{k+m_1-m_0}^{m_1}+(m_1-m_0)+2=y_{k+m_1-m_0}^{m_1}+(m_1-m_0).
\end{multline*}
Furthermore, $z_l^{\bar{m}}=y_l^{\bar{m}}$ for all other $l,\bar{m}$.

Case 3 occurs when $\{y_k^m\}\neq\{z_k^m\}$ and neither case 1 nor case 2 is satisfied.

It is not hard to see that $Q_{N,a}$ is the generator of the continuous-time Markov Chain defined in Section \ref{Introduction}. In general, if $Q$ is a matrix with countably many rows and columns such that its rows add up to $0$, its off-diagonal entries are nonnegative, and its diagonal entries are uniformly bounded, then there is a unique continuous-time Markov chain with $Q$ as its generator (see e.g. Proposition 2.10 of \cite{kn:Ald}). In words, this Markov chain satisfies
\begin{itemize}
\item
In state $i$, a jump takes place after exponential waiting time with parameter $-Q_{ii}$.
\item
The system makes a jump to state $j$ with probability $-Q_{ij}/Q_{ii}$.
\end{itemize}

We aim for the following:

\begin{theorem}\label{exp}
Let $\varphi_t(x)=e^{t(x-1)}$. Let $P^{(N),a}$ be the (possibly signed) central measure on $\GT^{(N),a}$ corresponding to some $E(x)\in C^1[-1,1]$ satisfying $E(1)\neq 0$. Then $e^{tQ_{N,a}}\cdot P^{(N),a}=A^{\varphi_t}_{N,a}\cdot P^{(N),a}$.
\end{theorem}
\begin{proof}
This theorem relies on the following proposition. It can be found as Theorem 9.6.1 in  \cite{kn:HP}.

\begin{proposition}
Let $\{A(t):t>0\}$ be bounded linear operators on a Banach space $B$ such that $A(s+t)=A(s)A(t)$ for all $s,t>0$. If $\displaystyle\lim_{t\rightarrow 0^+} \Vert A(t) - I \Vert=0$, then there exists a bounded linear operator $Q$ on $B$ such that $A(t)=e^{tQ}$ for $t\geq 0$.
\end{proposition}

Each $A^{\varphi_t}_{N,a}$ is a linear operator on $l^{1}(\GT^{(N),a})$. Since its matrix is stochastic, it is a bounded operator.

Let $L$ be the linear subspace of $l^{1}(\GT^{(N),a})$ spanned by all measures corresponding to functions $F\in C^1[-1,1]$ satisfying $F(1)\neq 0$. Define the Banach space $B$ as the completion of $L$. By Proposition \ref{GibbsEvolution}, $A^{\varphi_t}_{N,a}A^{\varphi_s}_{N,a}=A^{\varphi_{t+s}}_{N,a}$ on $L$. By continuity, the same holds on $B$. So it suffices to show that $\Vert A^{\varphi_t}_{N,a} - I\Vert\rightarrow 0$. By Lemma \ref{NormsAreEqual}, it equivalent to show that $\Vert T^{\varphi_t}_{N,a} - I\Vert\rightarrow 0$.

More precisely, it must be shown that
\[
\displaystyle\lim_{t\rightarrow 0^+}\sup_{\lambda\in\GT^{(N),a}} \sum_{\mu\in\GT^{(N),a}} \vert T^{\varphi_t}_{N,a}(\lambda,\mu)-\delta_{\lambda\mu}\vert=0.
\]
Since $T_{N,a}^{\varphi_t}$ is stochastic, it is equivalent to show that
\[
\displaystyle\lim_{t\rightarrow 0^+}\sup_{\lambda\in\GT^{(N),a}} (2-2T_{N,a}^{\varphi_t}(\lambda,\lambda))=0.
\]
So it suffices to show that $\displaystyle\lim_{t\rightarrow 0^+}\inf_{\lambda^{(N),a}\in\GT_{N,a}} T_{N,a}^{\varphi_t}(\lambda^{(N),a},\lambda^{(N),a})=1$.

We prove that $T_{r,\pm}^{\varphi_t}(\lambda^{(r),\pm},\lambda^{(r),\pm})\geq \exp(-t(r+1/2))$. Let $\varphi_t^{(n)}=(1+t(x-1)/n)^n$. By Lemmas \ref{UniformConvergence}, \ref{GroupHomomorphism}, and Proposition \ref{NonnegativeEntries}, respectively,
\begin{eqnarray*}
T_{r,\pm}^{\varphi_t}(\lambda^{(r),\pm},\lambda^{(r),\pm}) &=& \displaystyle\lim_{n\rightarrow\infty} T_{r,\pm}^{\varphi_t^{(n)}}(\lambda^{(r),\pm},\lambda^{(r),\pm}) \\
&=& \displaystyle\lim_{n\rightarrow\infty} (T_{r,\pm}^{1+t(x-1)/n})^n(\lambda^{(r),\pm},\lambda^{(r),\pm}) \\
&\geq & \displaystyle\lim_{n\rightarrow\infty} \left[ T_{r,\pm}^{1+t(x-1)/n}(\lambda^{(r),\pm},\lambda^{(r),\pm}) \right]^n
\end{eqnarray*}
From Proposition \ref{NonnegativeEntries},
\[T_{r,\pm}^{1+t(x-1)/n}(\lambda^{(r),\pm},\lambda^{(r),\pm}) \geq \left(\frac{R_+^r(R_+-p_1/2)-R_-^r(R_--p_1/2)}{\sqrt{p_0^2-p_1^2}}\right),\]
where
\[
p_0=1-\frac{t}{n},\ \ p_1=\frac{t}{n}, \ \ R_{\pm}=\frac{p_0\pm\sqrt{p_0^2-p_1^2}}{2}.
\]
Since $R_-\leq p_1/2$,
\[
T_{r,\pm}^{\varphi_t}(\lambda^{(r),\pm},\lambda^{(r),\pm}) \geq \displaystyle\lim_{n\rightarrow\infty} \left(\frac{R_+^r(R_+-p_1/2)}{\sqrt{p_0^2-p_1^2}}\right)^n.
\]
Finally, notice that as $n\rightarrow\infty$,
\begin{align*}
R_+^{rn} \geq \left(\sqrt{p_0^2-p_1^2}\right)^{rn} = \left(1-\frac{2t}{n}\right)^{rn/2} & \rightarrow e^{-rt}\\
\left(R_+ -\frac{p_1}{2}\right)^n & \rightarrow e^{-3t/2}\\
\left(\sqrt{p_0^2-p_1^2}\right)^{-n} & \rightarrow e^t.
\end{align*}

We have just shown that $A^{\varphi_t}_{N,a}\cdot P^{(N),a}=e^{tQ}\cdot P^{(N),a}$ for some $Q$. To finish the proof, we show that $Q=\tfrac{d}{dt}A^{\varphi_t}_{N,a}\vert_{t=0}=Q_{N,a}$. Since we only need to calculate $Q$ up to terms of order $O(t^2)$, we can replace $\varphi_t(x)=e^{t(x-1)}$ with $1-t+tx$.

The problem now is to calculate $A_{N,a}^{\varphi_t}$ up to terms of order $O(t^2)$. There are three cases to consider: when all the particles on the $m$th level stay still, when one of the particles on the $m$th level is pushed by a particle on a lower level, and when one of the particles on the $m$th level moves by itself. As an example, consider particles on the $(m,1/2)$ level when one of them is pushed.

The expression that needs to be calculated is
\begin{equation}\label{InductionStep}
\frac{T_{m,1/2}(\lambda^{(m),1/2},\mu^{(m),1/2})\Lambda_{m,-1/2}^{m,1/2}(\mu^{(m),1/2},\mu^{(m),-1/2})}{\Delta_{m,-1/2}^{m,1/2}(\lambda^{(m),1/2},\mu^{(m),-1/2})}.
\end{equation}
Assume that $\mu^{(m),-1/2}\nprec\lambda^{(m),1/2}$. Since $\mu^{(m),-1/2}\prec\mu^{(m),1/2}$, this implies that $\lambda^{(m),1/2}\neq\mu^{(m),1/2}$. Similarly, $\lambda^{(m),-1/2}\neq\mu^{(m),-1/2}$, which means one of the particles on the $(m,-1/2)$ level is pushing a particle on the $(m,1/2)$ level. Conversely, if the $k$th particle on the $(m,-1/2)$ level is pushing a particle on the $(m,1/2)$ level, then $\mu^{(m),-1/2}_k>\lambda^{(m),1/2}_k$, so $\mu^{(m),-1/2}\not\prec\lambda^{(m),1/2}$.

The transition probability on the $(m,1/2)$ level is (because of \eqref{JacobiIdentities})
\[
T_{m,1/2}(\lambda^{(m),1/2},\mu^{(m),1/2})=\frac{t}{2}\frac{\dim_{SO(2m+1)}\mu^{(m),1/2}}{\dim_{SO(2m+1)}\lambda^{(m),1/2}}{}+O(t^2).
\]
Furthermore,
\begin{align*}
&\Lambda_{m,-1/2}^{m,1/2}(\mu^{(m),1/2},\mu^{(m),-1/2})=\frac{\dim_{SO(2m)}\mu^{(m),-1/2}}{\dim_{SO(2m+1)}\mu^{(m),1/2}},\\
&\Delta_{m,-1/2}^{m,1/2}(\lambda^{(m),1/2},\mu^{(m),-1/2})\\
&=\Lambda_{m,-1/2}^{m,1/2}(\lambda^{(m),1/2},\lambda^{(m),-1/2})\cdot T_{m,-1/2}(\lambda^{(m),-1/2},\mu^{(m),-1/2})+O(t^2)\\
&=W(\lambda^{(m),-1/2}_m)\frac{\dim_{SO(2m)}\lambda^{(m),-1/2}}{\dim_{SO(2m+1)}\lambda^{(m),1/2}}\cdot\frac{t}{2}\frac{W(\mu_m^{(m),-1/2})}{W(\lambda_m^{(m),-1/2})}\frac{\dim_{SO(2m)}\mu^{(m),-1/2}}{\dim_{SO(2m)}\lambda^{(m),-1/2}}+O(t^2).
\end{align*}
Therefore equation (\ref{InductionStep}) equals $1+O(t)$.

Similarly, when all the particles on a level stay still, the contribution is $1+O(t)$. When a particle against the wall moves, the contribution is $t+O(t^2)$. When a particle not against the wall moves without being pushed by a particle on a lower level, the contribution is $t/2+O(t^2)$.
\end{proof}

\section{The Correlation Kernel}\label{determ}

\begin{theorem}\label{theorem2}
For any $\omega\in\Omega$ with parameter $\beta_1<1$, the point process $\mathcal{P}^{\omega}_{\mathfrak{X}}$ is determinantal. Denote its correlation kernel by $K^{\omega}(n_1,a_1,s_1;n_2,a_2,s_2)$.  If $(n_1,a_1)\trianglerighteq(n_2,a_2)$, then $K^{\omega}(n_1,a_1,s_1;n_2,a_2,s_2)$ equals
\begin{multline*}
\frac{W^{(a_1,-1/2)}(s_1)}{\pi}\int_{-1}^1\mathsf{J}_{s_1}^{(a_1,-1/2)}(x)\mathsf{J}_{s_2}^{(a_2,-1/2)}(x)(x-1)^{n_1-n_2}(1-x)^{a_1}(1+x)^{-1/2}dx\\
+\frac{W^{(a_1,-1/2)}(s_1)}{\pi}\frac{1}{2\pi i}\int_{-1}^1\oint_C\frac{E^{\omega}(x)}{E^{\omega}(u)}\mathsf{J}_{s_1}^{(a_1,-1/2)}(x)\mathsf{J}_{s_2}^{(a_2,-1/2)}(u)\\
\times\frac{(x-1)^{n_1}}{(u-1)^{n_2}}\frac{(1-x)^{a_1}(1+x)^{-1/2}dudx}{x-u}.
\end{multline*}
If $(n_1,a_1)\triangleleft(n_2,a_2)$, then $K^{\omega}(n_1,a_1,s_2;n_2,a_2,s_2)$ equals
\begin{multline*}
\frac{W^{(a_1,-1/2)}(s_1)}{\pi}\frac{1}{2\pi i}\int_{-1}^1\oint_C\frac{E^{\omega}(x)}{E^{\omega}(u)}\mathsf{J}_{s_1}^{(a_1,-1/2)}(x)\mathsf{J}_{s_2}^{(a_2,-1/2)}(u)\\
\times\frac{(x-1)^{n_1}}{(u-1)^{n_2}}\frac{(1-x)^{a_1}(1+x)^{-1/2}dudx}{x-u}.
\end{multline*}
The $u$-contour $C$ is a positively oriented simple loop that encircles the interval $[-1,1]$ but does not encircle any zeroes of $E^{\omega}$. Recall that the functions $E^{\omega},\mathsf{J}_s$ and $W$ were defined in Section \ref{RoOG}.
\end{theorem}

\textbf{Remark}. The case $\beta_1=1$ can be obtained by the limiting transition $\beta_1\rightarrow 1$ from Theorem \ref{theorem2}.

\begin{corollary}\label{KDelta}
With the definition of the particle-hole involution $\Delta$ given in Appendix \ref{Appendix A}, if $(n_2,a_2)\triangleleft(n_1,a_1)$, then $K^{\omega}_{\Delta}(n_1,a_1,s_1;n_2,a_2,s_2)$ equals
\begin{multline*}
-\frac{W^{(a_1,-1/2)}(s_1)}{\pi}\int_{-1}^1\mathsf{J}_{s_1}^{(a_1,-1/2)}(x)\mathsf{J}_{s_2}^{(a_2,-1/2)}(x)(x-1)^{n_1-n_2}(1-x)^{a_1}(1+x)^{-1/2}dx\\
-\frac{W^{(a_1,-1/2)}(s_1)}{\pi}\frac{1}{2\pi i}\int_{-1}^1\oint\frac{E^{\omega}(x)}{E^{\omega}(u)}\mathsf{J}_{s_1}^{(a_1,-1/2)}(x)\mathsf{J}_{s_2}^{(a_2,-1/2)}(u)\\
\times\frac{(x-1)^{n_1}}{(u-1)^{n_2}}\frac{(1-x)^{a_1}(1+x)^{-1/2}dudx}{x-u}.
\end{multline*}
If $(n_1,a_1)\trianglelefteq(n_2,a_2)$, then $K^{\omega}_{\Delta}(n_1,a_1,s_2;n_2,a_2,s_2)$ equals
\begin{multline*}
-\frac{W^{(a_1,-1/2)}(s_1)}{\pi}\frac{1}{2\pi i}\int_{-1}^1\oint\frac{E^{\omega}(x)}{E^{\omega}(u)}\mathsf{J}_{s_1}^{(a_1,-1/2)}(x)\mathsf{J}_{s_2}^{(a_2,-1/2)}(u)\\
\times\frac{(x-1)^{n_1}}{(u-1)^{n_2}}\frac{(1-x)^{a_1}(1+x)^{-1/2}dudx}{x-u}.
\end{multline*}
\end{corollary}
\begin{proof}
This result follows from the orthogonality relations
\[
\delta_{s_1s_2}=\frac{W^{(a,-1/2)}(s_1)}{\pi}\int_{-1}^1\mathsf{J}_{s_1}^{(a,-1/2)}(x)\mathsf{J}_{s_2}^{(a,-1/2)}(x)(1-x)^a(1+x)^{-1/2}dx
\]
for $a=\pm 1/2$. Note that in Theorem \ref{theorem2}, the two cases are $(n_2,a_2)\trianglelefteq (n_1,a_1)$ and $(n_1,a_1)\triangleleft (n_2,a_2)$. Here, the two cases are $(n_2,a_2)\triangleleft (n_1,a_1)$ and $(n_1,a_1)\trianglelefteq (n_2,a_2)$.
\end{proof}

This proof uses Theorem 4.2 from \cite{kn:BF2}, which we will describe in the next subsection. We will alter the notation to make it more convenient later.

\subsection{Determinantal structure of the correlation functions}

Theorem \ref{theorem1} and Lemma \ref{Interlacing} imply that the measure of a finite path $\boldsymbol{\lambda}=(\lambda^{(1),-1/2}\prec\lambda^{(1),1/2}\prec\lambda^{(2),-1/2}\prec\ldots\prec\lambda^{(N),a})$ is
\begin{multline}\label{Boint}
P^{\omega}(\boldsymbol{\lambda})\\
=\mathrm{const}\cdot\displaystyle\prod_{n=1}^N\Bigg[\det\left[\phi(\lambda^{(n),-1/2}_l-l+n,\lambda^{(n-1),1/2}_k-k+n-1)\right]_{1\leq k,l\leq n}\\
\times\det\left[\T(\lambda^{(n),1/2}_l-l+n,\lambda^{(n),-1/2}_k-k+n)_{1\leq
k,l\leq n}\right]\Bigg]\\
\times\det\left[f^{N,a}_{l}(\lambda^{(N),a}_k-k+N)\right]_{1\leq k,l\leq N}
\end{multline}
where $\phi$ and $\T$ were defined in Section \ref{MCoML}, and $f_j^{N,a}$ was defined in the statement of Theorem \ref{theorem1}. If $a=-1/2$, then the final determinant with the $\T$ does not occur.

Recall that we have set $\lambda_n^{(n-1),1/2}$ to be equal to zero, so $\lambda_n^{(n-1),1/2}-n+n-1=-1$. We will refer to $-1$ as a ``virtual variable,'' or ``virt.''

Set
\begin{equation}\label{DefinitionofPsi}
\Psi_{N-l}^{N,a}(s)=\frac{W^{(a,-1/2)}(s)}{\pi}\int_{-1}^1 E^{\omega}(x)\mathsf{J}_s^{(a,-1/2)}(x)(x-1)^{N-l}(1-x)^a(1+x)^{-1/2}dx.
\end{equation}
Observe that Span$\{f_1,\ldots,f_N\}=$ Span$\{\Psi_1,\ldots,\Psi_N\}$. Thus, if we replace $f_l$ in \eqref{Boint} with $\Psi_{N-l}$, the measure is not going to change.

Let $*$ denote convolution. More explicitly,
\[
(f*g)(x,y) = \displaystyle\sum_{z\geq 0}f(x,z)g(z,y),\qquad (f*h)(x) = \displaystyle\sum_{z\geq 0}f(z)h(z,x).
\]

For $(n_1,a_1)\triangleleft(n_2,a_2)$, set
\[
\phi^{(n_1,a_1),(n_2,a_2)}=
\begin{cases}
(\T*\phi)^{n_2-n_1}*\T,\ &\text{if}\ a_1=-1/2,\ a_2=1/2,\\
(\T*\phi)^{n_2-n_1},\ &\text{if}\ a_1=1/2,\ a_2=1/2,\\
(\phi*\T)^{n_2-n_1-1}*\phi,\ &\text{if}\ a_1=1/2,\ a_2=-1/2,\\
(\phi*\T)^{n_2-n_1},\ &\text{if}\ a_1=-1/2,\ a_2=-1/2.
\end{cases}
\]
For $(n_1,a_1)\trianglerighteq(n_2,a_2)$, set $\phi^{(n_1,a_1),(n_2,a_2)}=0$.

Let $M$ be the $N\times N$ matrix with entries
\[
M_{kl}=
\begin{cases}
\Psi_{N-l}^{N,a}*(\T*\phi)^{N-k+1}(virt),\ \ &\text{if}\ a=1/2,\\
\Psi_{N-l}^{N,a}*\phi*(\T*\phi)^{N-k}(virt),\ \ &\text{if}\ a=-1/2.
\end{cases}
\]

For $k\leq N$, define $\Psi_{n-k}^{n,a_1}=\Psi_{N-k}^{N,a}*\phi^{(n,a_1),(N,a)}$.

Theorem 4.2 from \cite{kn:BF2} also says that if $M$ is upper triangular and invertible, then there exist functions $\Phi_{n-k}^{n,a_2}(s)$ such that
\begin{itemize}
\item
$\{\Phi^{n,a_2}_{n-k}(s)\}_{k=1,\ldots,n}$ is a basis of the linear span of
\[
\begin{cases}
\{(\T*(\phi*\T)^{n-k}*\phi)(s,virt)\}_{k=1,\ldots,n},\ \ &\text{if}\ a_2=1/2\\
\{((\phi*\T)^{n-k}*\phi)(s,virt)\}_{k=1,\ldots,n},\ \ &\text{if}\
a_2=-1/2
\end{cases}
\]

\item For $0\leq i,j\leq n-1$,
\[
\displaystyle\sum_{s\geq 0}\Phi_i^{n,a}(s)\Psi_j^{n,a}(s)=\delta_{ij}.
\]
\end{itemize}
The formula for the correlation kernel is given by
\begin{equation}\label{FormulaForTheKernel}
K(n_1,a_1,s_1;n_2,a_2;s_2)=-\phi^{(n_1,a_1),(n_2,a_2)}(s_2,s_1)+\displaystyle\sum_{k=1}^{n_2}\Psi_{n_1-k}^{n_1,a_1}(s_1)\Phi_{n_2-k}^{n_2,a_2}(s_2).
\end{equation}

In section \ref{MatrixM}, we prove that $M$ is upper triangular. In section \ref{CalculatingPsi}, we calculate $\Psi_{n_1-k}^{n_1,a_1}$. In section \ref{Need2Sum}, we calculate $\displaystyle\sum_{k=1}^{n_2}\Psi_{n_1-k}^{n_1,a_1}(s_1)\Phi_{n_2-k}^{n_2,a_2}(s_2)$. In section \ref{Conv}, we calculate $\phi^{(n_1,a_1),(n_2,a_2)}(s_2,s_1)$. Finally in section \ref{AddingTogether}, we add all these expressions together to get the expression in Theorem \ref{theorem2}.

\subsection{The Matrix M}\label{MatrixM}

\begin{lemma}
The matrix $M$ is upper triangular and invertible.
\end{lemma}
\begin{proof}
By definition,
\[
M_{kl}=
\begin{cases}
\displaystyle\sum_{s\geq 0}\Psi_{N-l}^{N,a}(s)(\T*\phi)^{N-k+1}(s,virt),\ &a=1/2,\\
\displaystyle\sum_{s\geq
0}\Psi_{N-l}^{N,a}(s)*\phi*(\T*\phi)^{N-k}(s,virt),\ &a=-1/2.
\end{cases}
\]
Define
\[
g_k^{(a)}(s)=
\begin{cases}
(\T*\phi)^{N-k+1}(s,virt), \ &a=1/2,\\
\phi*(\T*\phi)^{N-k}(s,virt),\ &a=-1/2,\\
\end{cases}
\]
so that $M_{kl}=\langle \Psi_{N-l}^{N,a},g_k^{(a)} \rangle$. The definitions of $\mathcal{T}$ and $\phi$ imply that $g_k^{(a)}(s)$ is a polynomial of degree $2N-2k+1/2+a$.

Define
\begin{equation}\label{JackAndTheBeanstalk}
\Phi_{N-k}^{N,a}(s)=\frac{1}{2\pi i}\oint\frac{1}{E^{\omega}(u)}\mathsf{J}_s^{(a,-1/2)}(u)\frac{du}{(u-1)^{N-k+1}},
\end{equation}
where the integration contour is a positively oriented simple loop around $u=1$ that does not contain any zeroes of $E^{\omega}(u)$. (If  $\beta_1<1$, then $E^{\omega}(u)$ has no zeroes in $[-1,1]$, so the contour always exists). The integrand has a pole only at $u=1$, hence
\[
\Phi_{N-k}^{N,a}(s)=\frac{1}{(N-k)!}\frac{d^{N-k}}{du^{N-k}}\frac{1}{E^{\omega}(u)}\mathsf{J}_s^{(a,-1/2)}(u)\big|_{u=1}.
\]
Clearly, $\Phi_{N-k}^{N,a}(s)$ is a polynomial of degree $2N-2k+1/2+a$.

Since  $g_k^{(a)}$ and $\Phi_{N-k}^{N,a}$ are polynomials of degree $2N-2k+1/2+a$, there exists an invertible upper triangular matrix $A$ such that
\[
g_k^{(a)}=\displaystyle\sum_{m=1}^NA_{km}\Phi_{N-m}^{N,a}.
\]
Therefore
\[
M_{kl} = \langle g_k^{(a)},\Psi_{N-l}^{N,a}\rangle = \langle\displaystyle\sum_{m=1}^NA_{km}\Phi_{N-m}^{N,a},\Psi_{N-l}^{N,a}\rangle = \displaystyle\sum_{m=1}^NA_{km} \langle\Phi_{N-m}^{N,a},\Psi_{N-l}^{N,a}\rangle.
\]
Finally, by Lemma \ref{DeltaDistribution},
\[
\langle\Phi_{N-m}^{N,a},\Psi_{N-l}^{N,a}\rangle=\frac{1}{2\pi i}\oint\frac{du}{(u-1)^{l-m+1}}=\delta_{ml},
\]
so $M=A$, which is upper triangular and invertible.
\end{proof}

\subsection{\texorpdfstring{Calculating $\Psi^{n,a}_k$}{Calculating Psi}}\label{CalculatingPsi}
For ease of notation, let $E$ denote $E^{\omega}$ in the remaining sections.

The purpose of this section is to prove the following:
\begin{theorem}
For $l\leq n_1$,
\begin{multline}\label{Pork}
\Psi_{n_1-l}^{n_1,a_1}(s)=\frac{W^{(a_1,-1/2)}(s)}{\pi}\int_{-1}^1 E(x)\mathsf{J}_s^{(a_1,-1/2)}(x)(x-1)^{n_1-l}(1-x)^{a_1}(1+x)^{-1/2}dx.
\end{multline}

For $l>n_1$,
\begin{multline}\label{Beans}
\Psi_{n_1-l}^{n_1,a_1}(s)=\frac{W^{(a_1,-1/2)}(s)}{\pi}\\
\times\int_{-1}^1\frac{E(x)-E(1)-E'(1)(x-1)-\ldots-\frac{E^{(l-n_1-1)}(1)}{(l-n_1-1)!}(x-1)^{l-n_1-1}}{(x-1)^{l-n_1}}\\
\times \mathsf{J}_s^{(a_1,-1/2)}(x)(1-x)^{a_1}(1+x)^{-1/2}dx.
\end{multline}
\end{theorem}
\begin{proof}
Start with the proof of \eqref{Pork}. It will be done by induction on $d(n_1,a_1;N,a)=2(N-n_1)+a-a_1$. When $d(n_1,a_1;N,a)=0$, then \eqref{Pork} is true by \eqref{DefinitionofPsi}.
Now assume that \eqref{Pork} holds whenever $d(n_1,a_1;N,a)=m$. Either $a_1=1/2$ or $a_1=-1/2$. If $a_1=1/2$, then convoluting both sides of \eqref{Pork} by $\T$ and applying Lemma \ref{NeededSum2} with $T(x)=E(x)(x-1)^{n_1-l}$ shows that \eqref{Pork} holds for $n_1$ and $a_1=-1/2$. If $a_1=-1/2$, then convoluting both sides of \eqref{Pork} by $\phi$ and applying Lemma \ref{NeededSum}(b) with $T(x)=E(x)(x-1)^{n_1-l}$ shows that \eqref{Pork} holds for $n_1-1$ and $a_1=1/2$. Either way, \eqref{Pork} must hold whenever $d(n_1,a_1;N,a)=m+1$.

Now on to the proof of \eqref{Beans}. It also will be done by induction on $d(n_1,a_1;N,a)$. The base case occurs when $l-n_1=1$, and $a_1=1/2$. We just proved that \eqref{Pork} holds when $n_1=l$ and $a_1=-1/2$. Convolute both sides of \eqref{Pork} by $\phi$ and apply Lemma \ref{NeededSum}(b) with $T(x)=E(x)$. This proves the base case.

Now assume that \eqref{Beans} holds for some $n_1$ and $a_1=-1/2$. Convolute both sides of \eqref{Beans} by $\phi$. Apply Lemma \ref{NeededSum}(b) by setting
\[
T(x)=
\frac{E(x)-E(1)-E'(1)(x-1)-\ldots-\frac{E^{(l-n_1-1)}(1)}{(l-n_1-1)!}(x-1)^{l-n_1-1}}{(x-1)^{l-n_1}}.
\]
This shows that \eqref{Beans} holds for $n_1-1$ and $a_1=1/2$. Now assume that \eqref{Beans} holds for some $n_1$ and $a_1=1/2$. Convolute both sides of \eqref{Beans} by $\T$. By Lemma \ref{NeededSum2} with
\[
T(x)=\frac{E(x)-E(1)-E'(1)(x-1)-\ldots-\frac{E^{(l-n_1-1)}(1)}{(l-n_1-1)!}(x-1)^{l-n_1-1}}{(x-1)^{l-n_1}},
\]
\eqref{Beans} also holds for $n_1$ and $a_1=-1/2$.

\end{proof}

\subsection{\texorpdfstring{Calculating $\Phi_k^{n,a}$}{Calculating Phi}}\label{Need2Sum}
Define for $1\leq k\leq n\leq N$,
\[
\Phi^{n,a}_{n-k}(s)=\frac{1}{2\pi i}\oint\frac{1}{E(u)}\mathsf{J}_s^{(a,-1/2)}(u)\frac{du}{(u-1)^{n-k+1}}
\]
where the contour contains the interval $[-1,1]$ and does not contain any zeroes of $E(u)$. Note that this agrees with \eqref{JackAndTheBeanstalk}.
\begin{lemma}
(a) $\{\Phi^{n,a}_{n-k}(s)\}_{k=1,\ldots,n}$ is a basis of the linear span of
\[
\begin{cases}
\{(\T*(\phi*\T)^{n-k}*\phi)(s,virt)\}_{k=1,\ldots,n},\ \ &\text{if}\ a=1/2\\
\{((\phi*\T)^{n-k})*\phi(s,virt)\}_{k=1,\ldots,n},\ \ &\text{if}\
a=-1/2
\end{cases}
\]

(b) For $0\leq i,j\leq n-1$,
\[
\displaystyle\sum_{s\geq 0}\Phi_i^{n,a}(s)\Psi_j^{n,a}(s)=\delta_{ij}.
\]
\end{lemma}
\begin{proof}
(a) $\Phi^{n,a}_{n-k}(s)$ only has a pole at $u=1$, so it equals
\[
\frac{1}{(n-k)!}\frac{d^{n-k}}{dx^{n-k}}\frac{\mathsf{J}_s^{(a,-1/2)}(u)}{E(u)}\big|_{u=1},
\]
which is a polynomial in $s$ of degree $2n-2k+1/2+a$. Also,
$(\T*(\phi*\T)^{n-k}*\phi)(s,virt)$ is a polynomial in $s$ of degree $2n-2k+1$ and $((\phi*\T)^{n-k}*\phi)(s,virt)$ is a polynomial in $s$ of degree $2n-2k$. This proves (a).

(b) By Lemma \ref{DeltaDistribution}, (the contour below contains $[-1,1]$)
\[
\displaystyle\sum_{s\geq 0}\Phi_i^{n,a}(s)\Psi_j^{n,a}(s)=\frac{1}{2\pi i}\oint\frac{du}{(u-1)^{i-j+1}}=\delta_{ij}.
\]
\end{proof}

\begin{proposition}\label{AnnoyingSum}
In the expressions below, the $u$-contour is a positively oriented simple loop that encircles the interval $[-1,1]$ but does not encircle any zeroes of $E$.

For any $n_1\geq n_2\geq 1$ and $s_1,s_2\in\Z_{\geq 0}$, we have
\begin{multline}\label{SumToCalculate}
\displaystyle\sum_{k=1}^{n_2}\Psi_{n_1-k}^{n_1,a_1}(s_1)\Phi_{n_2-k}^{n_2,a_2}(s_2)\\
=\frac{W^{(a_1,-1/2)}(s_1)}{\pi}\int_{-1}^1\mathsf{J}_{s_1}^{(a_1,-1/2)}(x)\mathsf{J}_{s_2}^{(a_2,-1/2)}(x)(x-1)^{n_1-n_2}(1-x)^{a_1}(1+x)^{-1/2}dx\\
+\frac{W^{(a_1,-1/2)}(s_1)}{\pi}\frac{1}{2\pi i}\int_{-1}^1\oint\frac{E(x)}{E(u)}\mathsf{J}_{s_1}^{(a_1,-1/2)}(x)\mathsf{J}_{s_2}^{(a_2,-1/2)}(u)\\
\times\frac{(x-1)^{n_1}}{(u-1)^{n_2}}\frac{(1-x)^{a_1}(1+x)^{-1/2}dudx}{x-u}
\end{multline}
If $1\leq n_1<n_2$, then
\begin{multline*}
\displaystyle\sum_{k=1}^{n_2}\Psi_{n_1-k}^{n_1,a_1}(s_1)\Phi_{n_2-k}^{n_2,a_2}(s_2)\\
=\frac{W^{(a_1,-1/2)}(s_1)}{\pi}\int_{-1}^1\frac{\mathsf{J}_{s_1}^{(a_1,-1/2)}\mathsf{J}_{s_2}^{(a_2,-1/2)}}{E(x)}\left(E(x)-E(1)-\ldots-\frac{E^{(n_2-n_1-1)}(1)}{(n_2-n_1-1)!}(x-1)^{n_2-n_1-1}\right)\\
\times(x-1)^{n_1-n_2}(1-x)^{a_1}(1+x)^{-1/2}dx\\
+\frac{W^{(a_1,-1/2)}(s_1)}{\pi}\frac{1}{2\pi i}\int_{-1}^1\oint\frac{-E(1)-\ldots-\frac{E^{(n_2-n_1-1)}(1)}{(n_2-n_1-1)!}(u-1)^{n_2-n_1-1}}{E(u)}\\
\times\mathsf{J}_{s_1}^{(a_1,-1/2)}(x)\mathsf{J}_{s_2}^{(a_2,-1/2)}(u)(u-1)^{n_1-n_2}\frac{(1-x)^{a_1}(1+x)^{-1/2}dudx}{x-u}\\
+\frac{W^{(a_1,-1/2)}(s_1)}{\pi}\frac{1}{2\pi i}\int_{-1}^1\oint\frac{E(x)}{E(u)}\mathsf{J}_{s_1}^{(a_1,-1/2)}(x)\mathsf{J}_{s_2}^{(a_2,-1/2)}(u)\\
\times\frac{(x-1)^{n_1}}{(u-1)^{n_2}}\frac{(1-x)^{a_1}(1+x)^{-1/2}dudx}{x-u}.
\end{multline*}
\end{proposition}
\begin{proof}
First assume that $n_1\geq n_2$. Then the left hand side of \eqref{SumToCalculate} equals
\begin{multline*}
\frac{W^{(a_1,-1/2)}(s_1)}{\pi}\frac{1}{2\pi i}\int_{-1}^1\oint\frac{E(x)}{E(u)}\mathsf{J}_{s_1}^{(a_1,-1/2)}(x)\mathsf{J}_{s_2}^{(a_2,-1/2)}(u)\\
\times\frac{(x-1)^{n_1}}{(u-1)^{n_2}}\left(1-\left(\frac{u-1}{x-1}\right)^{n_2}\right)\frac{(1-x)^{a_1}(1+x)^{-1/2}dx}{x-u}
\end{multline*}
The expression
\begin{multline*}
\frac{W^{(a_1,-1/2)}(s_1)}{\pi}\frac{1}{2\pi i}\int_{-1}^1\oint\frac{E(x)}{E(u)}\mathsf{J}_{s_1}^{(a_1,-1/2)}(x)\mathsf{J}_{s_2}^{(a_2,-1/2)}(u)\\
\times\frac{(x-1)^{n_1}}{(u-1)^{n_2}}\left(\frac{u-1}{x-1}\right)^{n_2}\frac{(1-x)^{a_1}(1+x)^{-1/2}dx}{x-u}
\end{multline*}
has residues only at $u=x$, so equals
\begin{equation}\label{ExtraWeirdTerm}
-\frac{W^{(a_1,-1/2)}(s_1)}{\pi}\int_{-1}^1\mathsf{J}_{s_1}^{(a_1,-1/2)}(x)\mathsf{J}_{s_2}^{(a_2,-1/2)}(x)(x-1)^{n_1-n_2}(1-x)^{a_1}(1+x)^{-1/2}dx.
\end{equation}
Hence \eqref{SumToCalculate} holds.

Now assume $n_1<n_2$. Then $\displaystyle\sum_{k=1}^{n_1}\Psi_{n_1-k}^{n_1,a_1}(s_1)\Phi_{n_2-k}^{n_2,a_2}(s_2)$ equals
\begin{multline}\label{TheFirstSum}
\frac{W^{(a_1,-1/2)}(s_1)}{\pi}\frac{1}{2\pi i}\int_{-1}^1\oint\frac{E(x)}{E(u)}\mathsf{J}_{s_1}^{(a_1,-1/2)}(x)\mathsf{J}_{s_2}^{(a_2,-1/2)}(u)\\
\times\frac{(x-1)^{n_1}}{(u-1)^{n_2}}\left(1-\left(\frac{u-1}{x-1}\right)^{n_1}\right)\frac{(1-x)^{a_1}(1+x)^{-1/2}dx}{x-u}.
\end{multline}
To evaluate $\displaystyle\sum_{k=n_1+1}^{n_2}\Psi_{n_1-k}^{n_1,a_1}(s_1)\Phi_{n_2-k}^{n_2,a_2}(s_2)$, we first evaluate
\[
\displaystyle\sum_{k=n_1+1}^{n_2}\frac{E(x)-E(1)-E'(1)(x-1)-\ldots-\frac{E^{(k-n_1-1)}(1)}{(k-n_1-1)!}(x-1)^{k-n_1-1}}{(x-1)^{k-n_1}(u-1)^{n_2-k+1}}.
\]
It can be rearranged as
\begin{multline*}
\displaystyle (E(x)-E(1))\left(\sum_{k=n_1+1}^{n_2}\frac{(u-1)^{k-n_2-1}}{(x-1)^{k-n_1}}\right)-E'(1)(x-1)\left(\sum_{k=n_1+2}^{n_2}\frac{(u-1)^{k-n_2-1}}{(x-1)^{k-n_1}}\right)-\ldots\\
-\frac{E^{(n_2-n_1-1)}(1)}{(n_2-n_1-1)!}(x-1)^{n_2-n_1-1}\left(\sum_{k=n_2}^{n_2}\frac{(u-1)^{k-n_2-1}}{(x-1)^{k-n_1}}\right).
\end{multline*}
Each sum is a geometric series, which can be explicitly evaluated. After simplifying, we get
\begin{multline*}
(E(x)-E(1))\frac{(u-1)^{n_1-n_2}}{x-u}\left(1-\left(\frac{u-1}{x-1}\right)^{n_2-n_1}\right)\\
-E'(1)\frac{(u-1)^{n_1-n_2+1}}{x-u}\left(1-\left(\frac{u-1}{x-1}\right)^{n_2-n_1-1}\right)-\ldots\\
-\frac{E^{(n_2-n_1-1)}(1)}{(n_2-n_1-1)!}\frac{(u-1)^{-1}}{x-u}\left(1-\frac{u-1}{x-1}\right).
\end{multline*}

Therefore $\displaystyle\sum_{k=n_1+1}^{n_2}\Psi_{n_1-k}^{n_1,a_1}(s_1)\Phi_{n_2-k}^{n_2,a_2}(s_2)$ equals
\begin{multline}\label{TheSecondSum}
-\frac{W^{(a_1,-1/2)}(s_1)}{\pi}\frac{1}{2\pi i}\int_{-1}^1\oint\frac{E(x)-E(1)-\ldots-\frac{E^{(n_2-n_1-1)}(1)}{(n_2-n_1-1)!}(x-1)^{n_2-n_1-1}}{E(u)}\\
\times\mathsf{J}_{s_1}^{(a_1,-1/2)}(x)\mathsf{J}_{s_2}^{(a_2,-1/2)}(u)(x-1)^{n_1-n_2}\frac{(1-x)^{a_1}(1+x)^{-1/2}dx}{x-u}\\
+\frac{W^{(a_1,-1/2)}(s_1)}{\pi}\frac{1}{2\pi i}\int_{-1}^1\oint\frac{E(x)-E(1)-\ldots-\frac{E^{(n_2-n_1-1)}(1)}{(n_2-n_1-1)!}(u-1)^{n_2-n_1-1}}{E(u)}\\
\times\mathsf{J}_{s_1}^{(a_1,-1/2)}(x)\mathsf{J}_{s_2}^{(a_2,-1/2)}(u)(u-1)^{n_1-n_2}\frac{(1-x)^{a_1}(1+x)^{-1/2}dx}{x-u}.
\end{multline}
The first term in \eqref{TheSecondSum} has residues only at $u=x$, so simplifies to
\begin{multline*}
\frac{W^{(a_1,-1/2)}(s_1)}{\pi}\int_{-1}^1\frac{\mathsf{J}_{s_1}^{(a_1,-1/2)}\mathsf{J}_{s_2}^{(a_2,-1/2)}}{E(x)}\\ \times\left(E(x)-E(1)-\ldots-\frac{E^{(n_2-n_1-1)}(1)}{(n_2-n_1-1)!}(x-1)^{n_2-n_1-1}\right)(x-1)^{n_1-n_2}(1-x)^{a_1}(1+x)^{-1/2}dx.
\end{multline*}
Adding \eqref{TheFirstSum} and \eqref{TheSecondSum} and rearranging finishes the proof.
\end{proof}

\subsection{\texorpdfstring{Calculating $\phi^{(n_1,a_1),(n_2,a_2)}$}{Calculating phi}}\label{Conv}
For $(n_1,a_1)\triangleleft (n_2,a_2)$, set
\begin{multline*}
\Gamma(n_1,a_1,s_1;n_2,a_2,s_2)\\
=-\frac{W^{(a_1,-1/2)}(s_1)}{\pi}\frac{1}{2\pi i}\int_{-1}^1\oint\mathsf{J}_{s_1}^{(a_1,-1/2)}(x)\mathsf{J}_{s_2}^{(a_2,-1/2)}(u)\\
\times(u-1)^{n_1-n_2}\frac{(1-x)^{a_1}(1+x)^{-1/2}dudx}{x-u},
\end{multline*}
where the $u$-contour contains $[-1,1]$. In this section, we prove that $\phi^{(n_1,a_1),(n_2,a_2)}(s_2,s_1)=\Gamma(n_1,a_1,s_1;n_2,a_2,s_2)$.

\begin{proposition}\label{BLAH2} Assume $(n_1,a_1)\triangleleft (n_2,a_2)$. Then
\begin{eqnarray*}
(\T*\phi)^{n_2-n_1}*\T(s_2,s_1)&=&\Gamma(n_1,-1/2,s_1;n_2,1/2,s_2),\\
(\T*\phi)^{n_2-n_1}(s_2,s_1)&=&\Gamma(n_1,1/2,s_1;n_2,1/2,s_2),\\
(\phi*\T)^{n_2-n_1-1}*\phi(s_2,s_1)&=&\Gamma(n_1,1/2,s_1;n_2,-1/2,s_2),\\
(\phi*\T)^{n_2-n_1}(s_2,s_1)&=&\Gamma(n_1,-1/2,s_1;n_2,-1/2,s_2).
\end{eqnarray*}
\end{proposition}
\begin{proof}
Proceed by induction on $2n_2+a_2-(2n_1+a_1)$. First assume $2n_2+a_2-(2n_1+a_1)=1$ with $n_1=n_2$. Then
\begin{align*}
&\Gamma(n_1,-1/2,s_1;n_1,1/2,s_2)\\
&=-\frac{W^{(-1/2,-1/2)}(s_1)}{\pi}\frac{1}{2\pi i}\int_{-1}^1\oint\mathsf{J}_{s_1}^{(-1/2,-1/2)}(x)\mathsf{J}_{s_2}^{(1/2,-1/2)}(u)\\
&\ \qquad \times\frac{(1-x)^{-1/2}(1+x)^{-1/2}dudx}{x-u}\\
&=\frac{W^{(-1/2,-1/2)}(s_1)}{\pi}\frac{1}{2\pi i}\int_{-1}^1 \mathsf{J}_{s_1}^{(-1/2,-1/2)}(x)\mathsf{J}_{s_2}^{(1/2,-1/2)}(x)(1-x)^{-1/2}(1+x)^{-1/2}dx\\
&=\frac{W^{(-1/2,-1/2)}(s_1)}{\pi}\frac{1}{2\pi i}\int_{-1}^1 \mathsf{J}_{s_1}^{(-1/2,-1/2)}(x)\displaystyle\sum_{r=0}^{s_2}W^{(-1/2,-1/2)}(r)\mathsf{J}_r^{(-1/2,-1/2)}(x)\\
&\ \qquad \times(1-x)^{-1/2}(1+x)^{-1/2}dx\\
&=\mathcal{T}(s_2,s_1).
\end{align*}
The second equality follows by evaluating the residues at $u=x$, the third equality follows from Lemma \ref{UsefulIdentities}(a), and the fourth equality follows from the orthogonality relations.

Now assume $2n_2+a_2-(2n_1+a_1)=1$ with $n_1\neq n_2$. Then
\begin{align*}
&\Gamma(n_1,1/2,s_1;n_1+1,-1/2,s_2)\\
&=-\frac{W^{(1/2,-1/2)}(s_1)}{\pi}\frac{1}{2\pi i}\int_{-1}^1\oint\mathsf{J}_{s_1}^{(1/2,-1/2)}(x)\mathsf{J}_{s_2}^{(-1/2,-1/2)}(u)(u-1)^{-1}\\
&\ \qquad \times\frac{(1-x)^{1/2}(1+x)^{-1/2}dudx}{x-u}\\
&=\frac{W^{(1/2,-1/2)}(s_1)}{\pi}\frac{1}{2\pi i}\int_{-1}^1\mathsf{J}_{s_1}^{(1/2,-1/2)}(x)(\mathsf{J}_{s_2}^{(-1/2,-1/2)}(x)-1)(x-1)^{-1}\\
&\ \qquad \times(1-x)^{1/2}(1+x)^{-1/2}dx\\
&=\frac{W^{(1/2,-1/2)}(s_1)}{\pi}\frac{1}{2\pi i}\int_{-1}^1\mathsf{J}_{s_1}^{(1/2,-1/2)}(x)\displaystyle\sum_{r=0}^{s_2-1}\mathsf{J}_r^{(1/2,-1/2)}(x)\\
&\ \qquad \times(1-x)^{1/2}(1+x)^{-1/2}dx\\
&=\phi(s_2,s_1).
\end{align*}
The second equality follows by evaluating the residues at $u=x$ and $u=1$, the third equality follows from Lemma \ref{UsefulIdentities}(b), and the fourth equality follows from the orthogonality relations.

The inductive step is proved using a similar argument with the help of Lemmas \ref{UsefulIdentities}(a) and (b).
\end{proof}

\subsection{Computing the Kernel}\label{AddingTogether}
Let us now to compute the correlation kernel, which is given by \eqref{FormulaForTheKernel}.

The first case is when $(n_2,a_2)\trianglelefteq (n_1,a_1)$. Then $\phi^{(n_1,a_1),(n_2,a_2)}=0$. Furthermore, $n_1\geq n_2$, and $\displaystyle\sum_{k=1}^{n_2}\Psi_{n_1-k}^{n_1,a_1}(s_1)\Phi_{n_2-k}^{n_2,a_2}(s_2)$ was calculated in Proposition \ref{AnnoyingSum}.

The second case is when $(n_1,a_1)\triangleleft(n_2,a_2)$ and $n_1=n_2$. This happens only when $a_1=-1/2$ and $a_2=1/2$. Then $-\phi^{(n_1,a_1),(n_2,a_2)}=-\T$, which cancels with the single integral in \eqref{SumToCalculate}.

The final case is when $(n_1,a_1)\triangleleft(n_2,a_2)$ and $n_1<n_2$. Adding Propositions \ref{AnnoyingSum} and \ref{BLAH2}, we have the desired expression, plus an ``extra'' term:
\begin{multline*}
\frac{W^{(a_1,-1/2)}(s_1)}{\pi}\int_{-1}^1\frac{\mathsf{J}_{s_1}^{(a_1,-1/2)}\mathsf{J}_{s_2}^{(a_2,-1/2)}}{E(x)}\left(E(x)-E(1)-\ldots-\frac{E^{(n_2-n_1-1)}(1)}{(n_2-n_1-1)!}(x-1)^{n_2-n_1-1}\right)\\
\times(x-1)^{n_1-n_2}(1-x)^{a_1}(1+x)^{-1/2}dx\\
+\frac{W^{(a_1,-1/2)}(s_1)}{\pi}\frac{1}{2\pi i}\int_{-1}^1\oint\frac{E(u)-E(1)-\ldots-\frac{E^{(n_2-n_1-1)}(1)}{(n_2-n_1-1)!}(u-1)^{n_2-n_1-1}}{E(u)}\\
\times\mathsf{J}_{s_1}^{(a_1,-1/2)}(x)\mathsf{J}_{s_2}^{(a_2,-1/2)}(u)(u-1)^{n_1-n_2}\frac{(1-x)^{a_1}(1+x)^{-1/2}dudx}{x-u},
\end{multline*}
where the $u$-contour contains $[-1,1]$ and does not contain any zeroes of $E(u)$.
We prove this equals zero.

Let us evaluate the double integral. We use the identity
\begin{multline}\label{AlexeiIsAwesomeForComingUpWithThis}
\left(E(u)-E(1)-\ldots-\frac{E^{(n_2-n_1-1)}(1)}{(n_2-n_1-1)!}(u-1)^{n_2-n_1-1}\right)(u-1)^{n_1-n_2}\\
=\frac{1}{2\pi i}\oint\frac{E(w)dw}{(w-u)(w-1)^{n_2-n_1}},
\end{multline}
where the $w$-contour contains $u$ and $1$. The first double integral is now
\begin{multline*}
\frac{W^{(a_1,-1/2)}(s_1)}{\pi}\left(\frac{1}{2\pi i}\right)^2\int_{-1}^1\oint\oint\frac{E(w)}{E(u)}\\
\times\mathsf{J}_{s_1}^{(a_1,-1/2)}(x)\mathsf{J}_{s_2}^{(a_2,-1/2)}(u)\frac{dw}{(w-u)(w-1)^{n_2-n_1}}\frac{(1-x)^{a_1}(1+x)^{-1/2}dudx}{x-u}.
\end{multline*}
The $u$-contour has poles only at $u=x$, so we get
\begin{multline}\label{FirstDoubleIntegral}
\frac{W^{(a_1,-1/2)}(s_1)}{\pi}\frac{1}{2\pi i}\int_{-1}^1\oint\frac{E(w)}{E(x)}\mathsf{J}_{s_1}^{(a_1,-1/2)}(x)\mathsf{J}_{s_2}^{(a_2,-1/2)}(x)\\
\times\frac{(1-x)^{a_1}(1+x)^{-1/2}dwdx}{(x-w)(w-1)^{n_2-n_1}}.
\end{multline}

By \eqref{AlexeiIsAwesomeForComingUpWithThis}, the single integral equals
\begin{multline*}
\frac{W^{(a_1,-1/2)}(s_1)}{\pi}\frac{1}{2\pi i}\int_{-1}^1\oint\frac{E(w)}{E(x)}\mathsf{J}_{s_1}^{(a_1,-1/2)}(x)\mathsf{J}_{s_2}^{(a_2,-1/2)}(x)\\
\times\frac{(1-x)^{a_1}(1+x)^{-1/2}dwdx}{(w-x)(w-1)^{n_2-n_1}},
\end{multline*}
which cancels \eqref{FirstDoubleIntegral}.

The proof of Theorem \ref{theorem2} is now complete.

\section{Asymptotics of the Kernel}\label{AotK}
In this section, we analyze the large-time asymptotics of our system as $\gamma\rightarrow\infty$. Figure~\ref{RefereeFigure} shows the result of a computer simulation of the Markov chain. Notice three
distinct regions: one region where the particles are densely packed, another region where there are no particles, and an intermediate region. In section \ref{BL}, we find explicit formulas for the curves $q_1$ and $q_2$ that separate these three regions. Compare Figures \ref{RefereeFigure} and \ref{QPic}.

The appropriate global scaling is to take the time parameter $\gamma$ to vary proportionally to $tN$, while $n_i$ and $s_i$ vary proportionally to $lN$ and $dN$, respectively. Assume the pairwise differences $n_i-n_j$ and $s_i-s_j$ remain finite and constant. These limits are known as the \textit{bulk limits}. In the limit $N\rightarrow\infty$, the behavior in the intermediate region is described by the \textit{incomplete beta kernel}, which is an extension of the ubiquitous sine kernel. See Theorem \ref{BulkLimits} for the precise statement.

There are also two other scaling limits that we consider. The first
occurs when $\gamma\propto tN$, $n_i\propto lN$ with $n_i-n_j$ finite
constants, and $s_i$ are finite constants. In other words, we are
considering the large-time behavior of our point process at a finite
distance from the wall on the left. This behavior is described by
the \textit{discrete Jacobi kernel}, which we introduce. See
Theorem \ref{DiscreteJacobiKernel} for the precise statement. The
second edge limit occurs when $\gamma\propto N/2$, $n_i\propto
N+\eta_i\sqrt{N}$ for some $\eta_i$, and $s_i\propto \sigma_i
N^{1/4}$ for some $\sigma_i$. In other words, we are zooming in at
the point where $q_1$ meets the $y$-axis in Figure~\ref{QPic}. The
behavior here is described by the \textit{symmetric Pearcey kernel}.
It is an analog of the Pearcey kernel, which has previously appeared
in \cite{kn:ABK, kn:BK, kn:BH, kn:BH2, kn:OR, kn:TW2}. See Theorem
\ref{SymmetricPearceyKernel} for the precise statement.

\subsection{Bulk Limits}\label{BL}

Define the polynomials
\[
R_{t,d,l}(z)= t+(2l+2d-t)z+(2l-2d-t)z^2+tz^3
\]
and
\[
Q_{t,l}(z)=l(l-2t)^3-(2l^2+10lt-t^2)z^2+z^4.
\]

Let
\[
q_2(t,l)=\sqrt{-\dfrac{t^2}{2l^2}+\dfrac{5t}{l}+1+\dfrac{1}{2}\dfrac{t^2}{l^2}\left(1+\dfrac{4l}{t}\right)^{3/2}}
\]
and
\[
q_1(t,l)=
\begin{cases}
\sqrt{-\dfrac{t^2}{2l^2}+\dfrac{5t}{l}+1-\dfrac{1}{2}\dfrac{t^2}{l^2}\left(1+\dfrac{4l}{t}\right)^{3/2}},\ & 0\leq\dfrac{t}{l}<  \dfrac{1}{2},\\
0,\ \ & \dfrac{1}{2}\leq  \dfrac{t}{l}.
\end{cases}
\]
Note that $q_1(t,l)$ and $q_2(t,l)$ only depend on $t/l$. They are graphed in Figure~\ref{QPic}.

\begin{proposition}\label{prove!} Assume $l,d,t>0$.

(1) $R_{t,d,l}$ has two complex conjugate roots iff $l\cdot q_1(t,l)<d<l\cdot q_2(t,l)$.

(2) Let $z_0$ denote the nonreal root of $R_{t,d,l}$ in the upper-half plane, if it exists. Then $\vert z_0\vert>1$.

(3) Let $z_{max}$ denote the largest real root of $R_{t,d,l}$. If $d \geq l\cdot q_2(t,l)$, then $z_{max}>1$.

(4) Let $z_{min}$ denote the smallest real root of $R_{t,d,l}$. If $d \leq l\cdot q_1(t,l)$, then $z_{min}<-1$.
\end{proposition}
\begin{proof}
(1) $R_{t,d,l}(z)$ has nonreal roots iff its discriminant is negative. The discriminant of $R_{t,d,l}(z)$ is $16Q_{t,l}(d)$. Since $Q_{t,l}(z)=Q_{t,l}(-z)$, it crosses $(0,\infty)$ at most two times. If $t/l>1/2$, then $Q_{t,l}(0)<0$ and $Q_{t,l}(+\infty)=+\infty$, so therefore $Q_{t,l}$ crosses $(0,\infty)$ an odd number of times. Thus $Q_{t,l}$ has one positive real root. By using the explicit formula for the roots of a quadratic polynomial, we see that this root is $q_2(t,l)\cdot l$. So in this case, $Q_{t,l}(d)$ is negative iff $q_1(t,l)\cdot l=0<d<q_2(t,l)\cdot l$.

If $0<t/l<1/2$, then $Q_{t,l}(0)>0$ and $Q_{t,l}(l)=lt(-16l^2+13lt-8t^2)<0$ and $Q_{t,l}(+\infty)=+\infty$, so $Q_t$ has two positive real roots. By using the same formula, we see that the roots are $q_1(t,l)$ and $q_2(t,l)$. Once again, $Q_{t,l}(d)$ is negative iff $q_1(t,l)\cdot l<d<q_2(t,l)\cdot l$.

(2) The product of the roots of $R_{t,d,l}$ equals $-1$. So it suffices to show that $R_{t,d,l}$ has a root in the interval $(-1,1)$. In fact, $R_{t,d,l}$ has a root in the interval $(-1,0)$, because $R_{t,d,l}(-1)=-4d<0$ and $R_{t,d,l}(0)=t>0$.

(3) By (1), $R_{t,d,l}$ has three real roots. The product of these roots is $-1$, and their sum is $1+2(d-l)/t \geq 1 + 2l(q_2(t,l)-1)/t>0$ (this follows from the explicit expression for $q_2(t,l)$). We just showed in (2) that one of these roots is in the interval $(-1,0)$. Therefore the sum of the other two roots is positive, and their product is greater than $1$. This holds only if $z_{max}>1$.

(4) Because $q_1(t,l)$ is positive, $t/l$ must be less than $1/2$. By (1), $R_{t,d,l}$ has three real roots. The product of these roots is $-1$, and their sum is $1+2(d-l)/t \leq 1 + 2l(q_1(t,l)-1)/t < -1$. We just showed in (2) that one of these roots is in the interval $(-1,0)$. Therefore the sum of the other two roots is negative, and their product is greater than $1$. This holds only if $z_{min}<-1$.
\end{proof}

Using the notation of Proposition \ref{prove!}, define $z_0=z_0(t,d,l)$ to be
\[
z_0(t,d,l)=
\begin{cases}
z_{min},\ \text{if} \ d \leq l\cdot q_1(t,l)\\
z_0,\ \text{if} \ l\cdot q_1(t,l)<d<l\cdot q_2(t,l)\\
z_{max},\ \text{if}\ d \geq l\cdot q_2(t,l).
\end{cases}
\]

Define the function $S_{t,d,l}(z)$ as
\begin{equation}\label{S}
S_{t,d,l}(z)=t\frac{z+z^{-1}}{2}+l\log\left(\frac{z+z^{-1}}{2}-1\right)-d\log z.
\end{equation}
Note that $dS_{t,d,l}/dz=R_{t,d,l}(z)/(2z^2(z-1))$, so $z_0$ is a critical point of $S_{t,d,l}(z)$.

\begin{center}
\begin{figure}[htp]
\caption{The $x$-axis is $d$. The $y$-axis is $l$. The left curve is $d=l\cdot q_1(1/2,l)$ and the right curve is $d=l\cdot q_2(1/2,l)$.}
\begin{center}
\includegraphics[height=7cm, width=7cm]{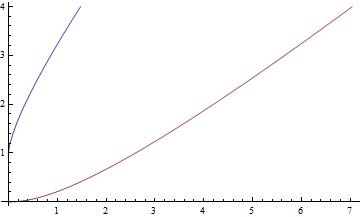}
\end{center}
\label{QPic}
\end{figure}
\end{center}

The incomplete beta kernel $B(k,l;\zeta)$ is defined by (cf. \cite{kn:OR0})
\[
B(k,l;\zeta)=\frac{1}{2\pi i}\int_{\bar{\zeta}}^{\zeta} (1-z)^k z^{-l-1}dz,
\]
where the contour of integration crosses $(0,1)$ if $k\geq 0$ and $(-\infty,0)$ if $k<0$.

\begin{theorem}\label{BulkLimits}
Let $\gamma$ and $\{s_i: \ 1\leq i\leq k\}$ and $\{n_i: \ 1\leq
i\leq k\}$ depend on $N$ in such a way that $\gamma/N\rightarrow
t>0$ and $s_i/N\rightarrow d>0$ and $n_i/N\rightarrow l$.
Furthermore, assume the differences $s_i-s_j$ and $n_i-n_j$ are all
finite and independent of $N$. For each $j$, set $\tilde{n}_j=2n_j+a_j-1/2$ and
$\tilde{s}_j=s_j-n_j$. Then 
\begin{multline*}
\lim_{N\to\infty}\rho_{k,\mathfrak{X}}^{\gamma}(n_1,a_1,s_1;\ldots;n_k,a_k,s_k)\\
=\begin{cases}
\det\left[B(\tilde{n}_i-\tilde{n}_j,\tilde{s}_i+\tilde{n}_i-\tilde{s}_j-\tilde{n}_j;z_0)\right]_{i,j=1}^k,\ &\text{if}\ q_1(t,l) < d/l < q_2(t,l),\\
 0,\ &\text{if}\ d/l \geq q_2(t,l),\\
 1,\ &\text{if}\ d/l \leq q_1(t,l).
 \end{cases}
\end{multline*}
\end{theorem}

\begin{proof}
Start with the case that $q_1(t,l) < d/l < q_2(t,l).$ Let $e^{i\theta}$ be a point on the unit circle such that $\Re(S(e^{i\theta})-S(z_0))>0$, where $S(z)=S_{t,d,l}(z)$, cf. \eqref{S}. This may be any point in the dark region in Figures \ref{Picture1},\ref{Picture2}; the existence of such points is easiliy verified by looking at the level lines $\Re((S(z))=\Re(S(z_0))$. Recall that $z_0$ is a critical point of $S(z)$. In the expression for the correlation kernel (Theorem \ref{theorem2}), deform the $u$-contour to a circle centered at $1$ and passing through $\cos\theta$. This causes the integral to pick up residues at $u=x$, where $x$ varies from $-1$ to $\cos\theta$. These residues occur as expression \eqref{ExtraIntegrall} below.

Now make the substitution $x=(z+z^{-1})/2$. The interval $[-1,1]$ becomes the unit circle and $[-1,\cos\theta]$ becomes an arc from $e^{-i\theta}$ to $e^{i\theta}$ that crosses $(-\infty,0)$. Let us also make the change of variable $u=(v+v^{-1})/2$. Set the $v$-contour to be an arc outside the unit circle that connects $e^{-i\theta}$ to $e^{i\theta}$. The weight on $[-1,1]$ becomes
\[
(1-x)^{a_1}(1+x)^{-1/2}dx\rightarrow
\begin{cases}
\dfrac{dz}{2iz},\ & a_1=-1/2,\\
\dfrac{-(z^{1/2}-z^{-1/2})^2dz}{4iz}, \ & a_1=1/2.
\end{cases}
\]
Denote the right hand side by $m_{a_1}(dz)$. Then $K_N^{\gamma}$ equals
\begin{multline}\label{TheDoubleIntegral}
K_N^{\gamma}(n_1,a_1,s_1;n_2,a_2,s_2)\\
=\frac{W^{(a_1,-1/2)}(s_1)}{2\pi^2 i}\int_{e^{-i\theta}}^{e^{i\theta}}\oint_{\vert z\vert=1}\frac{E^{\omega}(\frac{z+z^{-1}}{2})}{E^{\omega}(\frac{v+v^{-1}}{2})} \mathsf{J}_{s_1}^{(a_1,-1/2)}\left(\frac{z+z^{-1}}{2}\right)\mathsf{J}_{s_2}^{(a_2,-1/2)}\left(\frac{v+v^{-1}}{2}\right)\\
\times\frac{(\frac{z+z^{-1}}{2}-1)^{n_1}}{(\frac{v+v^{-1}}{2}-1)^{n_2}} \frac{1-v^{-2}}{z+z^{-1}-v-v^{-1}} m_{a_1}(dz) dv\\
\end{multline}
\begin{multline}\label{ExtraIntegral}
+\mathbf{1}_{(n_1,a_1)\trianglerighteq(n_2,a_2)}\Bigg(\frac{W^{(a_1,-1/2)}(s_1)}{\pi}\oint_{\vert z\vert=1}\mathsf{J}_{s_1}^{(a_1,-1/2)}\left(\frac{z+z^{-1}}{2}\right)\mathsf{J}_{s_2}^{(a_2,-1/2)}\left(\frac{z+z^{-1}}{2}\right)\\
\times\left(\frac{z+z^{-1}}{2}-1\right)^{n_1-n_2}m_{a_1}(dz)\Bigg)
\end{multline}
\begin{multline}\label{ExtraIntegrall}
+\Bigg(\frac{W^{(a_1,-1/2)}(s_1)}{\pi}\int_{e^{-i\theta}}^{e^{i\theta}}\mathsf{J}_{s_1}^{(a_1,-1/2)}\left(\frac{z+z^{-1}}{2}\right)\mathsf{J}_{s_2}^{(a_2,-1/2)}\left(\frac{z+z^{-1}}{2}\right)\\
\times\left(\frac{z+z^{-1}}{2}-1\right)^{n_1-n_2}m_{a_1}(dz)\Bigg).
\end{multline}

\begin{lemma}\label{Neo}
\begin{multline*}
\displaystyle\lim_{N\rightarrow\infty}\eqref{ExtraIntegral} = \mathbf{1}_{\tilde{n}_1\geq\tilde{n}_2} \\
\times\left(\frac{1}{2\pi i}\oint_{\vert z\vert=const} z^{s_2+n_2+a_2-(s_1+n_1+a_1)-1}(1-z)^{2n_1+a_1-2n_2-a_2}dz \frac{(-1)^{a_1-a_2}}{2^{n_1+a_1-n_2-a_2}}\right).
\end{multline*}
\end{lemma}
\begin{proof}
By using \eqref{ztothes}, when $a_1=a_2=-1/2$, \eqref{ExtraIntegral} equals (up to $\mathbf{1}_{(n_1,a_1)\trianglerighteq(n_2,a_2)}$)
\begin{equation*}
\frac{2}{2\pi i}\oint_{\vert z\vert=1}\left(\frac{z^{s_1}+z^{-s_1}}{2}\right)\left(\frac{z^{s_2}+z^{-s_2}}{2}\right)\left(\frac{z+z^{-1}}{2}-1\right)^{n_1-n_2}\frac{dz}{z}.
\end{equation*}
Expanding the first two parantheses yields four terms. For the term corresponding to $z^{s_1+s_2}$, deform the contour to a circle of radius less than $1$. This will make the integral exponentially small as $N\rightarrow\infty$. For the term corresponding to $z^{-s_1-s_2}$, deform the contour to a circle of radius greater than $1$; again, this integral vanishes as $N\rightarrow\infty$. The remaining term is
\begin{equation}\label{Pine}
\frac{1}{4\pi i}\oint_{\vert z\vert=1} z^{s_2-s_1-1}\left(\frac{z+z^{-1}}{2}-1\right)^{n_1-n_2}dz+\frac{1}{4\pi i}\oint_{\vert z\vert=1} z^{s_1-s_2-1}\left(\frac{z+z^{-1}}{2}-1\right)^{n_1-n_2}dz.
\end{equation}
Making the substitution $z\rightarrow z^{-1}$ in the second integral, \eqref{Pine} becomes
\begin{equation}\label{Amber}
\frac{1}{2\pi i}\oint_{\vert z\vert=1} z^{s_2-s_1-1}\left(\frac{z+z^{-1}}{2}-1\right)^{n_1-n_2}dz.
\end{equation}
When $a_1=a_2=1/2$, \eqref{ExtraIntegral} equals
\[
\frac{1}{\pi}\oint_{\vert z\vert=1} \frac{z^{s_1+1/2}-z^{-s_1-1/2}}{z^{1/2}-z^{-1/2}}\frac{z^{s_2+1/2}-z^{-s_2-1/2}}{z^{1/2}-z^{-1/2}}\left(\frac{z+z^{-1}}{2}-1\right)^{n_1-n_2}(2-z-z^{-1})\frac{dz}{4iz}.
\]
Making similar deformations and substitutions, we see that this expression also converges to \eqref{Amber}.
By a similar argument, when $a_1=-1/2$ and $a_2=1/2$, \eqref{ExtraIntegral} converges to
\begin{equation}\label{Banana}
\frac{1}{2\pi i}\oint_{\vert z\vert=1} (z^{s_2-s_1-1/2})(z^{1/2}-z^{-1/2})\left(\frac{z+z^{-1}}{2}-1\right)^{n_1-n_2-1}dz.
\end{equation}
When $a_1=1/2$ and $a_2=-1/2$, \eqref{ExtraIntegral} converges to
\begin{equation}\label{Coke}
\frac{1}{4\pi i}\oint_{\vert z\vert=1} z^{s_2-s_1-3/2}(z^{1/2}-z^{-1/2})\left(\frac{z+z^{-1}}{2}-1\right)^{n_1-n_2}dz.
\end{equation}
Using $\frac{z+z^{-1}}{2}-1=\frac{z^{-1}(z-1)^2}{2}$ in \eqref{Amber},\eqref{Banana} and \eqref{Coke} shows that the lemma holds in all cases.
\end{proof}

\begin{lemma}\label{Three}
\begin{multline*}
\displaystyle\lim_{N\rightarrow\infty}\eqref{ExtraIntegrall} = 
\left(\frac{1}{2\pi i}\int_{e^{-i\theta}}^{e^{i\theta}} z^{s_2+n_2+a_2-(s_1+n_1+a_1)-1}(1-z)^{2n_1+a_1-2n_2-a_2}dz \frac{(-1)^{a_1-a_2}}{2^{n_1+a_1-n_2-a_2}}\right),
\end{multline*}
where the contour of integration crosses $(-\infty,0)$.
\end{lemma}
\begin{proof}
The proof is almost exactly the same as the proof of Lemma \ref{Neo}. The only difference is that the integration in \eqref{ExtraIntegrall} is over an arc from $e^{-i\theta}$ to $e^{i\theta}$ that crosses $(-\infty,0)$, rather than the unit circle.
\end{proof}

\begin{lemma}\label{Wot}
Assume $q_1(t,l) < d/l < q_2(t,l)$. Then
\begin{multline*}
\displaystyle\lim_{N\rightarrow\infty}\eqref{TheDoubleIntegral} \rightarrow \frac{1}{2\pi i}\int_{\bar{z_0}}^{e^{-i\theta}} z^{s_2-n_2-(s_1-n_1)-1}(1-z)^{2n_1+a_1-2n_2-a_2}dz\frac{(-1)^{a_1-a_2}}{2^{n_1+a_1-n_2-a_2}}\\
+\frac{1}{2\pi i}\int_{e^{i\theta}}^{z_0} z^{s_2-n_2-(s_1-n_1)-1}(1-z)^{2n_1+a_1-2n_2-a_2}dz\frac{(-1)^{a_1-a_2}}{2^{n_1+a_1-n_2-a_2}}.
\end{multline*}
\end{lemma}
\begin{proof}
Recall that in expression \eqref{Wot}, the $v$-arc goes outside the unit circle. We only do the calculation explicitly when $a_1=a_2=-1/2$, because the other cases are similar. By \eqref{ztothes},
\[
\mathsf{J}_{s_1}^{(-1/2,-1/2)}\left(\frac{z+z^{-1}}{2}\right)\mathsf{J}_{s_2}^{(-1/2,-1/2)}\left(\frac{v+v^{-1}}{2}\right)=\left(\frac{z^{s_1}+z^{-s_1}}{2}\right)\left(\frac{v^{s_2}+v^{-s_2}}{2}\right)
\]
Expanding the parantheses on the right hand side yields four terms. This means that \eqref{TheDoubleIntegral} can be written as the sum of four terms. We now proceed to evaluate each of these terms separately.

The term corresponding to $z^{-s_1}v^{s_2}$ equals
\begin{multline}\label{TheFirstTerm}
\frac{W^{(a_1,-1/2)}(s_1)}{4\cdot 2\pi^2 i}\oint\int_{e^{-i\theta}}^{e^{i\theta}}\frac{E^{\omega}(\frac{z+z^{-1}}{2})}{E^{\omega}(\frac{v+v^{-1}}{2})} \frac{z^{-s_1}}{v^{-s_2}} \frac{(\frac{z+z^{-1}}{2}-1)^{n_1}}{(\frac{v+v^{-1}}{2}-1)^{n_2}}\\
\times\frac{1}{\frac{z+z^{-1}}{2}-\frac{v+v^{-1}}{2}}\frac{1-v^{-2}}{2}dv\frac{dz}{2 i z}.
\end{multline}
The part of the integrand that depends on $N$ equals $\exp(N(S(z)-S(z_0))$.

\begin{figure}[htp]
\caption{On the left is $\Re(S(z)-S(z_0))$. Shaded regions indicate $\Re>0$ and white regions indicate $\Re<0$. The double zero occurs at $z_0$. The arc $v$ goes from $e^{-i\theta}$ to $e^{i\theta}$. The unit circle has been drawn on the right.}
\includegraphics[totalheight=0.30\textheight]{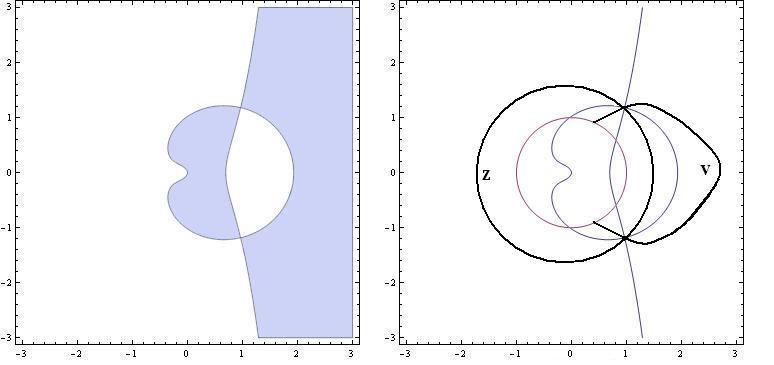}
\label{Picture1}
\end{figure}

With the deformations as shown in Figure~\ref{Picture1}, the double integral asymptotically evaluates to zero. Since $\vert z_0\vert>1$ (Proposition \ref{prove!}), the contours can be deformed without picking up residues at $v=z^{-1}$. The residues at $v=z$ are
\begin{multline}\label{Residues1}
\frac{1}{4\pi i}\int_{\bar{z_0}}^{e^{-i\theta}}z^{s_2-s_1-1}\left(\frac{z+z^{-1}}{2}-1\right)^{n_1-n_2}dz\\ +\frac{1}{4\pi i}\int_{e^{i\theta}}^{z_0}z^{s_2-s_1-1}\left(\frac{z+z^{-1}}{2}-1\right)^{n_1-n_2}dz.
\end{multline}

To calculate the term corresponding to $z^{s_1}v^{s_2}$, make the substitution $z\leftrightarrow z^{-1}$. Then the integrand and contour remain the same, so this term also equals \eqref{Residues1}.

It remains to calculate the terms corresponding to $z^{s_1}v^{-s_2}$ and $z^{-s_1}v^{-s_2}$. Because we can substitute $z\leftrightarrow z^{-1}$, it suffices to calculate the term corresponding to $z^{-s_1}v^{-s_2}.$ Substituting $v\leftrightarrow v^{-1}$, the double integral again becomes \eqref{TheFirstTerm}, except now with the $v$-arc \emph{inside} the unit circle. In this case, the double integral is asymptotically zero, because we can deform the contours as shown in Figure~\ref{Picture2} without picking up any residues.

\begin{figure}[htp]
\caption{On the left is $\Re(S(z)-S(z_0))$. Shaded regions indicate $\Re>0$ and white regions indicate $\Re<0$. The double zero occurs at $z_0$. The arc $v$ goes from $e^{-i\theta}$ to $e^{i\theta}$. The unit circle has been drawn on the right.}
\includegraphics[totalheight=0.30\textheight]{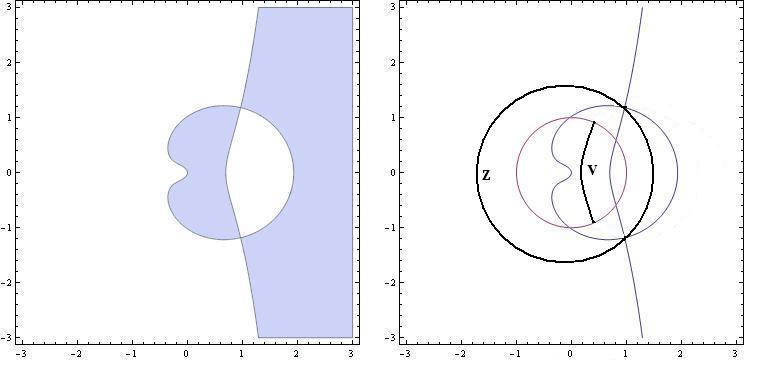}
\label{Picture2}
\end{figure}

Collecting all the terms shows that we get
\begin{multline}
\frac{1}{2\pi i}\int_{\bar{z_0}}^{e^{-i\theta}}z^{s_2-s_1-1}\left(\frac{z+z^{-1}}{2}-1\right)^{n_1-n_2}dz\\
+\frac{1}{2\pi i}\int_{e^{i\theta}}^{z_0}z^{s_2-s_1-1}\left(\frac{z+z^{-1}}{2}-1\right)^{n_1-n_2}dz
\end{multline}
\end{proof}

Lemmas \ref{Neo}, \ref{Wot} and \ref{Three} prove the theorem when $q_1(t,l) < d/l < q_2(t,l)$.

Now assume $d/l \geq q_2(t,l)$. This time, do not deform the $u$-contour. With the same substitutions, we have 

\begin{multline}\label{TheDoubleIntegral'}
K_N^{\gamma}(n_1,a_1,s_1;n_2,a_2,s_2)\\
=\frac{W^{(a_1,-1/2)}(s_1)}{2\pi^2 i}\oint\oint\frac{E^{\omega}(\frac{z+z^{-1}}{2})}{E^{\omega}(\frac{v+v^{-1}}{2})} \mathsf{J}_{s_1}^{(a_1,-1/2)}\left(\frac{z+z^{-1}}{2}\right)\mathsf{J}_{s_2}^{(a_2,-1/2)}\left(\frac{v+v^{-1}}{2}\right)\\
\times\frac{(\frac{z+z^{-1}}{2}-1)^{n_1}}{(\frac{v+v^{-1}}{2}-1)^{n_2}} \frac{1-v^{-2}}{z+z^{-1}-v-v^{-1}} m_{a_1}(dz) dv\\
\end{multline}
\begin{multline}\label{ExtraIntegral'}
+\mathbf{1}_{(n_1,a_1)\trianglerighteq(n_2,a_2)}\Bigg(\frac{W^{(a_1,-1/2)}(s_1)}{\pi}\oint\mathsf{J}_{s_1}^{(a_1,-1/2)}\left(\frac{z+z^{-1}}{2}\right)\mathsf{J}_{s_2}^{(a_2,-1/2)}\left(\frac{z+z^{-1}}{2}\right)\\
\times\left(\frac{z+z^{-1}}{2}-1\right)^{n_1-n_2}m_{a_1}(dz)\Bigg),
\end{multline}
where the $z$-contour is the unit circle and the $v$-contour goes outside the unit circle.
Once again, there are four terms in \eqref{TheDoubleIntegral'}, corresponding to $z^{\pm s_1}v^{\pm s_2}$. First let us calculate the term corresponding to $z^{-s_1}v^{s_2}$. This term equals 
\begin{multline}\label{TheFirstTerm'}
\frac{W^{(a_1,-1/2)}(s_1)}{4\cdot 2\pi^2 i}\oint\oint\frac{E^{\omega}(\frac{z+z^{-1}}{2})}{E^{\omega}(\frac{v+v^{-1}}{2})} \frac{z^{-s_1}}{v^{-s_2}} \frac{(\frac{z+z^{-1}}{2}-1)^{n_1}}{(\frac{v+v^{-1}}{2}-1)^{n_2}}\\
\times\frac{1}{\frac{z+z^{-1}}{2}-\frac{v+v^{-1}}{2}}\frac{1-v^{-2}}{2}dv\frac{dz}{2 i z}.
\end{multline}

Let $z_{max}$ denote the largest real root of $R_{t,d,l}$ and deform the countours as shown in Figure~\ref{AuxPicture2}. Then \eqref{TheFirstTerm'} asymptotically evaluates to zero. Since $z_{max}>1$ by Proposition \ref{prove!}, these deformations can be made without picking up residues at $v=z^{-1}$. The residues at $v=z$ equal
\begin{equation}\label{Residues2}
\frac{1}{4\pi i}\oint z^{s_2-s_1-1}\left(\frac{z+z^{-1}}{2}-1\right)^{n_1-n_2}dz,
\end{equation}
where the contour crosses $(0,\infty)$.

\begin{figure}[htp]
\caption{On the left is $\Re(S(z)-S(z_{max}))$. Shaded regions indicate $\Re>0$ and white regions indicate $\Re<0$. The double zero occurs at $z_{max}$.}
\includegraphics[totalheight=0.30\textheight]{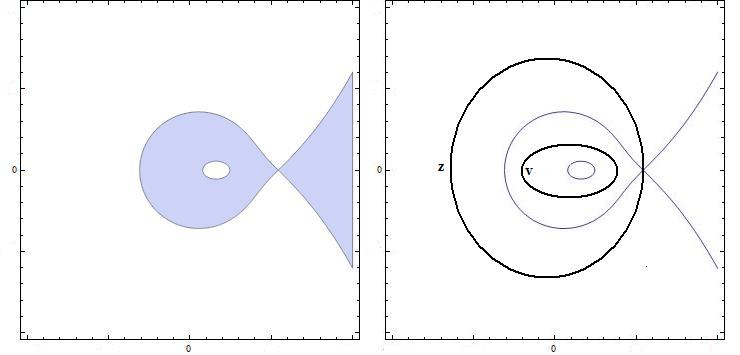}
\label{AuxPicture2}
\end{figure}

\begin{figure}[htp]
\caption{On the left is $\Re(S(z)-S(z_{min}))$. Shaded regions indicate $\Re>0$ and white regions indicate $\Re<0$. The double zero occurs at $z_{min}$.}
\includegraphics[totalheight=0.30\textheight]{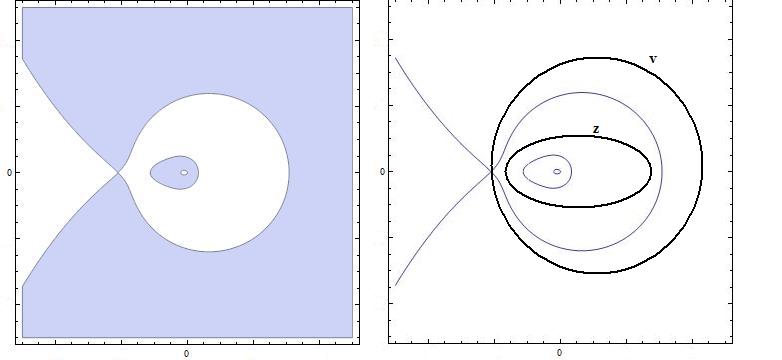}
\label{AuxPicture1}
\end{figure}

Similarly, as before, the term corresponding to $z^{s_1}v^{s_2}$ also equals \eqref{Residues2}, and the terms corresponding to $z^{\pm s_2}v^{-s_2}$ equal zero. Thus \eqref{TheDoubleIntegral'} and \eqref{ExtraIntegral'} asymptotically cancel out, so $K^{\gamma}(n_i,a_i,s_i;n_j,a_j,s_j)$ converges to $0$ when $(n_1,a_1)\trianglelefteq (n_2,a_2)$. Therefore the determinant equals $0$.

When $d/l \leq q_1(t,l)$, the argument is similar. Let $z_{min}$ be the smallest real root of $R_{t,d,l}$. Make the deformations in \eqref{TheDoubleIntegral'} as shown in Figure~\ref{AuxPicture1}. Since $z_{min}<-1$ by Proposition \ref{prove!}, these deformations can be made without picking up residues at $v=z^{-1}$. The integral does not pick up residues at $v=z$, so \eqref{TheDoubleIntegral'} converges to zero. Thus $\det[K^{\gamma}]_1^k$ converges to a triangular matrix. The diagonal entries are given by Lemma \ref{Neo}, which all evaluate to $1$. Therefore the determinant converges $1$.
\end{proof}

\subsection{Limit Shape}\label{limitshape}
Let $H:\R_{\geq 0}\times\mathfrak{Y}$ be the height function defined by
\[
H(\gamma,s,n)=\vert\{(y,m)\in\mathcal{L}_{\mathfrak{Y}} : m=n,y>s\}\vert,
\]
where $\mathcal{L}_{\mathfrak{Y}}$ is the random point configuration of $\mathcal{P}_{\mathfrak{Y}}^{\gamma}$. In other words, $H(\gamma,s,n)$ is the number of particles to the right of $(s,n)$ at time $\gamma$. Define $h$ to be
\begin{equation}\label{HeightDefinition}
h(t,d,l)=\displaystyle\lim_{N\rightarrow\infty}\frac{1}{N}\mathbb{E}H(tN,[xN],[lN])
\end{equation}

Recall that we defined $z_0=z_0(t,d,l)$ and $S(z)=S_{t,d,l}(z)$ in the previous section.
\begin{proposition} The pointwise limit \eqref{HeightDefinition} exists and
\begin{equation}\label{height}
h(t,d,l)=\Im\left(\frac{S(z_0)}{2\pi}\right).
\end{equation}
\end{proposition}
\begin{proof}
Note that
\[
\mathbb{E}H(\gamma,s,n)=\displaystyle\sum_{\begin{subarray}{c} y\geq s \\ y\equiv\ n+1\ \text{mod}\ 2 \end{subarray}} \rho_{1,\mathfrak{Y}}^{\gamma}(y,n),
\]
where $\rho_{1,\mathfrak{Y}}^{\gamma}$ is the first correlation function.
Using Theorem \ref{BulkLimits} and the dominated convergence theorem, we get
\[
\displaystyle\lim_{N\rightarrow\infty} \rho_{1,\mathfrak{Y}}^{tN}([dN],[lN])=\frac{1}{\pi}\arg z_0
\]
and
\[
h(t,d,l)=\displaystyle\lim_{N\rightarrow\infty}\frac{1}{N}\mathbb{E}H(tN,[dN],[lN])=\frac{1}{2\pi}\int_d^{\infty}\arg z_0(t,x,l) dx.
\]
Now take the partial derivative of the right hand side of \eqref{height} with respect to $d$. The result is
\[
\Im\left(\frac{1}{2\pi}\frac{\partial S(z_0)}{\partial d}\right)=-\frac{\arg z_0}{2\pi}+S'(z_0)\frac{\partial z_0}{\partial d}.
\]
Since $S'(z_0)=0$, we can conclude that
\[
\Im\left(\frac{S(z_0)}{2\pi}\right)=\frac{1}{2\pi}\int_d^{\infty}\arg z_0(t,x,l)dx + const.
\]
Evaluating both sides at $d=+\infty$ proves that $const=0$.
\end{proof}

\subsection{Discrete Jacobi Kernel}
For $-1< u < 1$, define the discrete Jacobi kernel
$L(n_1,a_1,s_1,n_2,a_2,s_2;u)$ on $\mathfrak{X}\times\mathfrak{X}$
as follows. If $(n_1,a_1)\trianglerighteq(n_2,a_2)$, then
\begin{multline*}
L(n_1,a_1,s_1,n_2,a_2,s_2;u)\\
=\frac{W^{(a_1,-1/2)}(s_1)}{\pi}\int_u^1 \mathsf{J}_{s_1}^{(a_1,-1/2)}(x)\mathsf{J}_{s_2}^{(a_2,-1/2)}(x)(x-1)^{n_1-n_2}(1-x)^{a_1}(1+x)^{-1/2}dx.
\end{multline*}
If $(n_1,a_1)\triangleleft(n_2,a_2)$, then
\begin{multline*}
L(n_1,a_1,s_1,n_2,a_2,s_2;u)\\
=-\frac{W^{(a_1,-1/2)}(s_1)}{\pi}\int_{-1}^u
\mathsf{J}_{s_1}^{(a_1,-1/2)}(x)\mathsf{J}_{s_2}^{(a_2,-1/2)}(x)(x-1)^{n_1-n_2}(1-x)^{a_1}(1+x)^{-1/2}dx.
\end{multline*}

For $n_1=n_2$ and $(a_1,a_2)=(-1/2,-1/2)$,  the integral can be
evaluated. Set $v=\cos^{-1}(u)$. Then
\[
L(n_1,-\tfrac 12,s_1,n_1,-\tfrac12,s_2;u)=(1-\tfrac 12
\delta_{s_1=0})\left(\frac{\sin v(s_1-s_2)}{\pi(s_1-s_2)} +
 \frac{\sin v(s_1+s_2)}{\pi(s_1+s_2)}\right).
\]
When $s_1=s_2$, the above expression is evaluated by L'H\^{o}pital's
rule. This can be viewed as a discrete analog of the Bessel kernel,
which arises at the hard edge in random matrix models, see
(1.2)--(1.3) from \cite{kn:TW} or (2.6) from \cite{kn:F}.

The discrete Jacobi kernel arises in the following limit.

\begin{theorem}\label{DiscreteJacobiKernel}
Let $\gamma$ depend on $N$ in such a way that $\gamma/N\rightarrow t>0$. Assume $t/l>1/2$. Let $s_1,\ldots,s_k$ be fixed finite constants. Let $n_1,\ldots,n_k$ depend on $N$ in such a way that
$n_i/N\rightarrow l$ and their differences $n_i-n_j$ are fixed finite constants. Then 
\begin{multline*}
\lim_{N\rightarrow\infty}\rho^{\gamma}_{k,\mathfrak{X}}(n_1,a_1,s_1;\ldots;n_k,a_k,s_k)\\=
\begin{cases}
\det[L(n_i,a_i,s_i,n_j,a_j,s_j;1-l/t)]_{i,j=1}^k, & \ \text{if}\ t/l>1/2,\\
1, & \ \text{if}\ t/l\leq 1/2.
\end{cases}
\end{multline*}
\end{theorem}
\begin{proof}
Let $A(z)=tz+l\cdot\log(1-z)$. The kernel $K^{\gamma}(n_1,a_1,s_1;n_1,a_2,s_2)$ equals
\begin{equation}\label{DoubleIntegral}
\frac{W^{(a_1,-1/2)}(s_1)}{2\pi^2 i}\oint\int_{-1}^1 \frac{e^{N(A(x)-A(1-l/t))}}{e^{N(A(u)-A(1-l/t))}} \mathsf{J}_{s_1}^{(a_1,-1/2)}(x)\mathsf{J}_{s_2}^{(a_2,-1/2)}(u)\frac{(1-x)^{a_1}(1+x)^{-1/2}}{x-u}dxdu
\end{equation}
\begin{equation}\label{SingleIntegral}
+\mathbf{1}_{(n_1,a_1)\trianglerighteq (n_2,a_2)}\left(\frac{W^{(a_1,-1/2)}(s_1)}{\pi}\int_{-1}^1 \mathsf{J}_{s_1}^{(a_1,-1/2)}(x)\mathsf{J}_{s_2}^{(a_2,-1/2)}(u)(x-1)^{n_1-n_2}(1-x)^{a_1}(1+x)^{-1/2}dxdu\right).
\end{equation}

\begin{figure}[htp]
\caption{On the left is $\Re(A(z)-A(1-l/t))$. Shaded regions indicate $\Re>0$ and white regions indicate $\Re<0$. The double zero occurs at $1-l/t$.}
\includegraphics[totalheight=0.30\textheight]{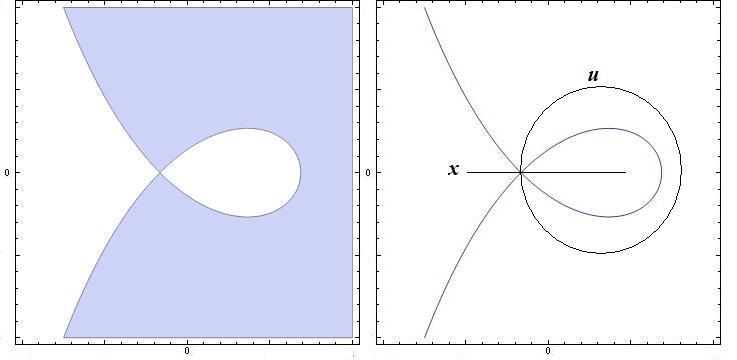}
\label{DoubleZero}
\end{figure}

Recall from Theorem~\ref{theorem2} that the $u$-contour is a positively oriented simple loop that encircles the interval $[-1,1]$. Now deform the $u$ contour as shown in Figure~\ref{DoubleZero}. With this deformation,
\[
\left|\frac{e^{N(A(x)-A(1-l/t))}}{e^{N(A(u)-A(1-l/t))}}\right|=\frac{e^{N\Re(A(x)-A(1-l/t))}}{e^{N\Re(A(u)-A(1-l/t))}}\rightarrow 0,
\]
so the integrand converges to zero. However, for $t/l>1/2$, the deformations cause the double integral to pick up residues at $u=x$. Thus, expression \eqref{DoubleIntegral} converges to
\begin{equation}\label{DoubleIntegral2}
-\frac{W^{(a_1,-1/2)}(s_1)}{\pi}\int_{-1}^{1-l/t} \mathsf{J}_{s_1}^{(a_1,-1/2)}(x)\mathsf{J}_{s_2}^{(a_2,-1/2)}(x)(x-1)^{n_1-n_2}(1-x)^{a_1}(1+x)^{-1/2}dx.
\end{equation}
Adding \eqref{SingleIntegral} to \eqref{DoubleIntegral2} shows that $K^{\gamma}$ converges to the discrete Jacobi kernel.

If $t/l\leq 1/2$, then the double integral does not pick up residues at $u=x$. Thus $\det[K^{\gamma}]$ converges to a triangular matrix. The diagonal entries are given by \eqref{SingleIntegral}, which all evaluate to $1$. Therefore the determinant equals $1$.
\end{proof}

\subsection{Symmetric Pearcey Kernel}
Define a kernel $\mathcal{K}$ on $\R_+\times\R$ as follows. In the expressions below, the $u$-contour is integrated on rays from $\infty e^{i\pi/4}$ to $0$ to $\infty e^{-i\pi/4}$. If $\eta_1\leq\eta_2$, then
\begin{multline}\label{GaussianLikeKernel}
\mathcal{K}(\sigma_1,\eta_1,\sigma_2,\eta_2)= \frac{2}{\pi^2 i} \int\int_0^{\infty}\exp(-\eta_1 x^2 + \eta_2 u^2 + u^4 - x^4)\\
\times  \cos(\sigma_1 x) \cos(\sigma_2 u) \frac{u}{u^2-x^2}dxdu.
\end{multline}
If $\eta_2<\eta_1$, then
\begin{multline}\label{GaussianLikeKernel2}
\mathcal{K}(\sigma_1,\eta_1,\sigma_2,\eta_2)=-\frac{1}{2\sqrt{\pi(\eta_1-\eta_2)}}\left(\exp\frac{(\sigma_1+\sigma_2)^2}{4(\eta_2-\eta_1)}+\exp\frac{(\sigma_1-\sigma_2)^2}{4(\eta_2-\eta_1)}\right)\\
+\frac{2}{\pi^2 i} \int\int_0^{\infty}\exp(-\eta_1 x^2 + \eta_2 u^2
+ u^4 - x^4) \cos(\sigma_1 x) \cos(\sigma_2 u)
\frac{u}{u^2-x^2}dxdu.
\end{multline}

This kernel arises as follows. Let $\rho_{k,\mathfrak{X}}^{\gamma,\Delta}$ denote the correlation function of $\mathcal{P}^{\gamma}_{\mathfrak{X},\Delta}$, which is the pushforward of $\mathcal{P}^{\gamma}_{\mathfrak{X}}$ under $\Delta$. See Appendix \ref{Appendix A}.

\begin{theorem}\label{SymmetricPearceyKernel}
For $1\leq i\leq k$, let $s_i$ depend on $N$ in such a way that $s_i/N^{1/4}\rightarrow 2^{-5/4}\sigma_i>0$ as $N\rightarrow\infty$.
Let $\gamma$ depend on $N$ in such a way that $\gamma/N\rightarrow 1/2$ as $N\rightarrow\infty$. Let $n_i$ depend on $N$ in such a way
that $(n_i-N)/\sqrt{N}\rightarrow 2^{-1/2}\eta_i$. Then there is the pointwise limit
\[
\lim_{N\rightarrow\infty}\left(\frac{N^{1/4}}{2^{5/4}}\right)^k
\rho^{\gamma,\Delta}_{k,\mathfrak{X}}(n_1,a_1,s_1;\ldots;n_k,a_k,s_k) =
\det[\mathcal{K}(\sigma_i,\eta_i,\sigma_j,\eta_j)]_{1\leq i,j\leq
k}.
\]
\end{theorem}
\begin{proof}
From Corollary \ref{KDelta}, the left hand side of the above equation is equal to $\det[(N^{1/4}/2^{5/4})K^{\omega}_{\Delta}(n_i,a_i,s_i,n_j,a_j,s_j)]_{i,j=1}^k$ and $(N^{1/4}/2^{5/4})K^{\omega}_{\Delta}(n_1,a_1,s_1,n_2,a_2,s_2)$ equals
\begin{multline}\label{Limits1}
\frac{N^{1/4}}{2^{5/4}\cdot 2\pi i}\frac{W^{(a_1,-1/2)}}{\pi}\oint_{\vert u\vert=1}\int_{-1}^1 \frac{e^{\gamma x}}{e^{\gamma u}}\mathsf{J}_{s_1}^{(a_1,-1/2)}(x)\mathsf{J}_{s_2}^{(a_2,-1/2)}(u)\\
\times\frac{(x-1)^{n_1}}{(u-1)^{n_2}} \frac{(1-x)^{a_1}(1+x)^{-1/2}}{u-x}dxdu
\end{multline}
\begin{multline}\label{Limits2}
-\mathbf{1}_{(n_1,a_1)\triangleright (n_2,a_2)}\bigg(\frac{N^{1/2}}{2^{5/4}}\frac{W^{(a_1,-1/2)}(s_1)}{\pi}\int_{-1}^1\mathsf{J}_{s_1}^{(a_1,-1/2)}(x)\mathsf{J}_{s_2}^{(a_2,-1/2)}(x)\\
(x-1)^{n_1-n_2}(1-x)^{a_1}(1+x)^{-1/2}dx\bigg).
\end{multline}

Deform the contours as shown in Figure~\ref{DoubleZero}, with the double critical point at $-1$. Then, asymptotically, nonvanishing contributions to \eqref{Limits1} and \eqref{Limits2} come from near $-1$. This justifies the substitutions $x'=N^{1/2}(x+1)$ and $u'=N^{1/2}(u+1)$. For large $N$, $x'$ is integrated from $0$ to $\infty$ and $u'$ is integrated from $i\infty$ to $-i\infty$. There are also the following asymptotic relations:
\begin{eqnarray}
(1-x)^{a_1}(1+x)^{-1/2}& \approx & 2^{a_1}N^{1/4}(x')^{-1/2},\\
\frac{dxdu}{u-x}& \approx & N^{-1/2}\frac{du'dx'}{u'-x'}.
\end{eqnarray}

Let us show that if $s/N^{1/4}\rightarrow 2^{-5/4}\sigma$, then
\begin{equation}
(-1)^s \mathsf{J}_s^{(\pm 1/2,-1/2)}(x)\rightarrow\cos(2^{-3/4}\sigma(x')^{1/2})
\end{equation}
For $a_1=-1/2$,
\[
\mathsf{J}_s^{(-1/2,-1/2)}(x)=\cos(s\theta),\ \ x=\cos\theta.
\]
Hence,
\begin{eqnarray*}
& &(-1)^s \mathsf{J}_s^{(-1/2,-1/2)}(x)=(-1)^s\cdot\cos(s\cdot\cos^{-1}(-1+N^{-1/2}x'))\\
&\approx &(-1)^s\cdot\cos(s\pi-2^{-3/4}\sigma(x')^{1/2}+o(1))\\
&\approx & (-1)^s(\cos(s\pi)\cos(2^{-3/4}\sigma(x')^{1/2})+\sin(s\pi)\sin(2^{-3/4}\sigma(x')^{1/2}))\\
&=& \cos(2^{-3/4}\sigma(x')^{1/2})
\end{eqnarray*}
Similarly, for $a=1/2$,
\begin{eqnarray*}
& &(-1)^s \mathsf{J}_s^{(1/2,-1/2)}(x) = (-1)^s\cdot\frac{\sin((s+1/2)\theta)}{\sin(\theta/2)} \\
&\approx &  (-1)^s\cdot \frac{\sin((s+1/2)\pi-2^{-3/4}\sigma(x')^{1/2}+o(1))}{\sqrt{(2+o(1))/2}}\\
&\approx & (-1)^s(\sin((s+\tfrac{1}{2})\pi)\cos(2^{-3/4}\sigma(x')^{1/2})- \cos((s+\tfrac{1}{2})\pi)\sin(2^{-3/4}\sigma(x')^{1/2}))\\
&=& \cos(2^{-3/4}\sigma(x')^{1/2})
\end{eqnarray*}
For $a_1=\pm 1/2$ and $s>0$,
\begin{equation}
\frac{W^{(a_1,-1/2)}(s)}{\pi}=\frac{2^{-a_1}\sqrt{2}}{\pi}
\end{equation}

Let $A(z)=z/2+\log(1-z)$. We have
\[
A(z)-A(-1)=-\frac{1}{8}(z+1)^2+O((z+1)^3),
\]
and
\begin{multline}\label{Limits3}
(-2)^{n_2-n_1}\exp(\gamma(x-u))\frac{(x-1)^{n_1}}{(u-1)^{n_2}}\approx\frac{e^{N(A(x)-A(-1))+2^{-1/2}\eta_1\sqrt{N}(-\log 2+\log(1-x))}}{e^{N(A(u)-A(-1))+2^{-1/2}\eta_2\sqrt{N}(-\log 2+\log(1-x))}}\\
\rightarrow \frac{e^{-(x')^2/8-\eta_1x'/2^{3/2}}}{e^{-(u')^2/8-\eta_2u'/2^{3/2}}}.
\end{multline}
The kernel can be multiplied by the conjugating factor $(-1)^{s_1-s_2}(-2)^{n_2-n_1}$ without changing the determinant. Combining $\eqref{Limits1}-\eqref{Limits3}$ shows that
\begin{eqnarray*}
& &(-1)^{s_1-s_2}(-2)^{n_2-n_1}\eqref{Limits1}\\
&=&\frac{N^{1/4}}{2^{5/4}}\frac{(-1)^{s_1+s_2}(-2)^{n_2-n_1}}{2\pi i}\frac{W^{(a_1,-1/2)}}{\pi}\oint_{\vert u\vert=1}\int_{-1}^1 \frac{e^{\gamma x}}{e^{\gamma u}}\mathsf{J}_{s_1}^{(a_1,-1/2)}(x)\mathsf{J}_{s_2}^{(a_2,-1/2)}(u)\\
& &\ \ \times \frac{(x-1)^N}{(u-1)^N} \frac{(1-x)^{a_1}(1+x)^{-1/2}}{u-x}dxdu\\
&\rightarrow & \frac{2^{1/4}}{4\pi^2 i} \int_{i\infty}^{-i\infty}\int_0^{\infty}\exp\left(-\frac{1}{2^{3/2}}\eta_1 x' +\frac{1}{2^{3/2}}\eta_2 u' + \frac{1}{8}(u'^2-x'^2)\right) \cos(2^{-3/4}\sigma_1x'^{1/2})\\
& &\ \ \times \cos(2^{-3/4}\sigma_2u'^{1/2}) \frac{x'^{-1/2}}{u'-x'}dx'du'\\
&=& \eqref{GaussianLikeKernel}.
\end{eqnarray*}
In the last equality, the substitutions $u'=2^{3/2}\tilde{u}^2$ and $x'=2^{3/2}\tilde{x}^2$ were needed. We need one final additional calculation:
\[
(-2)^{n_2-n_1}(x-1)^{n_1-n_2}\approx (1-\tfrac{1}{2}N^{-1/2}x')^{(\eta_1-\eta_2)2^{-1/2}\sqrt{N}}\rightarrow \exp(\tfrac{1}{2^{3/2}}(\eta_2-\eta_1)x'),
\]
which shows that
\begin{multline*}
(-1)^{s_1-s_2}(-2)^{n_2-n_1}\eqref{Limits2}\rightarrow\\
\mathbf{1}_{\eta_1>\eta_2}\left(-\frac{2^{1/4}}{2\pi}\int_0^{\infty}\exp\left(\frac{1}{2^{3/2}}(\eta_2-\eta_1)x\right)\cos(2^{-3/4}\sigma_1x^{1/2})\cos(2^{-3/4}\sigma_2x^{1/2})x^{-1/2}dx\right)\\
=-\frac{\mathbf{1}_{\eta_1>\eta_2}}{2\sqrt{\pi(\eta_1-\eta_2)}}\left(\exp\frac{(\sigma_1+\sigma_2)^2}{4(\eta_2-\eta_1)}+\exp\frac{(\sigma_1-\sigma_2)^2}{4(\eta_2-\eta_1)}\right)
\end{multline*}
Therefore $(-1)^{s_1-s_2}(-2)^{n_2-n_1}(\eqref{Limits1}+\eqref{Limits2})\rightarrow\eqref{GaussianLikeKernel2}$.
\end{proof}

\appendix
\section{Generalities on Random Point Processes.}\label{Appendix A}

Let $\X$ be a locally compact separable topological space. A {\it
point configuration\/} $X$ in $\X$ is a locally finite collection of
points of the space $\X$. For our purposes it suffices to assume
that the points of $X$ are always pairwise distinct. Denote by
$\Conf(\X)$ the set of all point configurations in $\X$.

A relatively compact Borel subset $A\subset \X$ is called {\it a
window}. For a window $A$ and $X\in\Conf(\X)$, set $N_A(X)=|A\cap
X|$ (number of points of $X$ in the window). Thus, $N_A$ is a
function on $\Conf(\X)$. $\Conf(\X)$ is equipped with the Borel
structure generated by functions $N_A$ for all windows $A$.

A {\it random point process\/} on $\X$ is a probability measure on
$\Conf(\X)$. One often uses the term {\it particles\/} for the
elements of a random point configuration.

Given a random point process on $\X$, one can usually define a
sequence $\{\rho_n\}_{n=1}^\infty$, where $\rho_n$ is a symmetric
measure on $\X^n$ called the $n$th {\it correlation measure}. Under
mild conditions on the point process, the correlation measures exist
and determine the process uniquely.

The correlation measures are characterized by the following
property: For any $n\ge1$ and a compactly supported bounded Borel
function $f$ on $\X^n$ one has
$$
\int_{\X^n}f\rho_n=\left\langle \sum_{x_{i_1},\dots,x_{i_n}\in
X}f(x_{i_1},\dots,x_{i_n})\right \rangle_{X\in\Conf(\X)}
$$
where $\langle\,\cdot\,\rangle$ denotes averaging with respect to
our point process, and the sum on the right is taken over all
$n$-tuples of pairwise distinct points of the random point
configuration $X$.

Often one has a natural measure $\mu$ on $\X$ (called {\it reference
measure\/}) such that the correlation measures have densities with
respect to $\mu^{\otimes n}$, $n=1,2,...\,$. Then the density of
$\rho_n$ is called the $n$th {\it correlation function\/} and it is
usually denoted by the same symbol $\rho_n$.

The first correlation function $\rho_1$ is often called the {\it
density function\/} as it measures the average density of particles.

For point processes on a finite or countable discrete space $\X$ it
is natural to choose the counting measure as the reference measure
$\mu$, and then there is a simpler way to define the correlation
functions: For any $n=1,2,\dots$ and any pairwise distinct
$x_1,\dots,x_n\in\X$,
$$
\rho_n(x_1,\dots,x_n)=\Prob\{X\in \Conf(\X)\mid X\supset
\{x_1,\dots,x_n\}\}.
$$

If $\X$ is discrete, a random point process on $\X$ is always
uniquely determined by its correlation functions.

The reader can find more information on random point processes in
\cite{kn:DVJ}.

A point process on $\X$ is called {\it determinantal\/} if there
exists a function $K(x,y)$ on $\X\times\X$ such that the correlation
functions (with respect to some reference measure) are given by the
determinantal formula
$$
\rho_n(x_1,\dots,x_n)={\det[K(x_i,x_j)]}_{i,j=1}^n
$$
for all $n=1,2,\dots$. The function $K$ is called the {\it
correlation kernel\/}.

Note that the correlation kernel is not defined uniquely: $K(x,y)$
and $\frac{f(x)}{f(y)}K(x,y)$ define the same correlation functions
for an arbitrary nonzero function $f$ on $\X$.

Assume that $\X$ is discrete. Define a map $\Delta$ by
$$
\Delta: \Conf(\X)\rightarrow \Conf(\X), \ \ X\mapsto \X\backslash X.
$$
Given a point process $\mathcal{P}$ on $\X$, its pushforward under
$\Delta$ is also a point process on $\X$; denote it by $\mathcal{P}_{\Delta}$. The map $\Delta$ is often referred to as {\it
particle-hole involution\/}, because the particles of $\mathcal{P}_{\Delta}$ are located exactly at those points of $\X$ where there
are no particles of $\mathcal{P}$.  With this notation, we have the
following proposition.

\begin{Proposition 0}\label{Proposition in Appendix A} If $\mathcal{P}$ is a determinantal point
process with correlation kernel $K(x,y)$, then $\mathcal{P}_{\Delta}$
is also a determinantal point process with correlation kernel
$$
K_\Delta(x,y)=\delta_{x,y}-K(x,y).
$$
\end{Proposition 0}

The proof is an application of the inclusion-exclusion principle,
see Proposition A.8 of \cite{kn:BOO}.

\frenchspacing
\bibliographystyle{plain}

\end{document}